\documentclass{amsart}

\usepackage[utf8]{inputenc}
\usepackage[T1]{fontenc}
\usepackage[pdfencoding=auto, psdextra]{hyperref}
\usepackage{amsmath}
\usepackage{amssymb}
\usepackage{amsthm}
\usepackage{mathtools}  
\usepackage{microtype}
\usepackage{graphicx}
\usepackage{subcaption}
\usepackage{clrscode3e} 
\usepackage{algorithm} 
\usepackage{float}
\usepackage[noend]{algorithmic}
\usepackage{etoolbox}
\patchcmd{\thebibliography}{\chapter*}{\section*}{}{}
\usepackage{marginnote}
\usepackage{bbm}
\usepackage{siunitx}
\usepackage{enumitem}
\usepackage[font=small,labelfont=bf]{caption} 
\usepackage[section]{placeins} 
\usepackage[textsize=tiny]{todonotes}

\usepackage{marginnote}

\newcommand*{\mailto}[1]{\href{mailto:#1}{\nolinkurl{#1}}}
\newcommand{\Dx}{{\Delta x}}
\newcommand{\Dt}{{\Delta t}}
\newcommand{\R}{\mathbb{R}}
\renewcommand{\id}{\text{\normalfont{id}}}

\newcommand{\D}{\mathcal{D}}
\newcommand{\F}{\mathcal{F}}
\newcommand{\M}{\mathcal{M}}

\DeclareMathOperator*{\argmin}{arg\,min}

\apptocmd{\lim}{\limits}{}{}
\numberwithin{equation}{section}
{
    \theoremstyle{plain}
    \newtheorem{definition}{Definition}[section]
    \newtheorem{remark}[definition]{Remark}
    \newtheorem{prop}[definition]{Proposition}
    \newtheorem{example}[definition]{Example}

    \newtheorem{theorem}[definition]{Theorem}
    \newtheorem{corollary}[definition]{Corollary}
    \newtheorem{lemma}[definition]{Lemma}   
}
\newcommand{\lref}[1]{\hyperref[{#1}]{\ref*{#1}}}

\title[A numerical view on \fontsize{11pt}{13pt}{$\alpha$}-dissipative solutions]{A numerical view on \fontsize{15pt}{17pt}{$\alpha$}-dissipative solutions of the Hunter--Saxton equation}

\author[T. Christiansen ]{Thomas Christiansen}
\address{Department of Mathematical Sciences\\ NTNU Norwegian University of Science and Technology\\ NO-7491 Trondheim\\ Norway}
\email{\mailto{thomas.christiansen@ntnu.no}}
\urladdr{\url{https://www.ntnu.edu/employees/thomachr}}

\author[K. Grunert]{Katrin Grunert}
\address{Department of Mathematical Sciences\\ NTNU Norwegian University of Science and Technology\\ NO-7491 Trondheim\\ Norway}
\email{\mailto{katrin.grunert@ntnu.no}}
\urladdr{\url{https://www.ntnu.edu/employees/katrin.grunert}}

\thanks{Research supported by the grant {\it Wave Phenomena and Stability --- a Shocking Combination (WaPheS)} from the Research Council of Norway.}  

\begin{document}

\maketitle
\counterwithin{equation}{section}
\raggedbottom
 \allowdisplaybreaks

\begin{abstract}
We propose a new numerical method for $\alpha$-dissipative solutions of the Hunter--Saxton equation, where $\alpha$ belongs to $W^{1, \infty}(\R, [0, 1))$. The method combines a projection operator with a generalized method of characteristics and an iteration scheme, which is based on enforcing minimal time steps whenever breaking times cluster. Numerical examples illustrate that these minimal time steps increase the efficiency of the algorithm substantially. Moreover, convergence of the wave profile is shown in $C([0, T], L^{\infty}(\R))$ for any finite $T \geq 0$.
\end{abstract}

\section{Introduction}
In this paper we derive a novel numerical algorithm for $\alpha$-dissipative solutions to the Cauchy problem of the Hunter--Saxton (HS) equation
\begin{align}
    u_t(t,x) + uu_x(t,x) &= \frac{1}{4}\int_{-\infty}^xu_x^2(t,z)dz - \frac{1}{4}\int_{x}^{\infty}u_x^2(t,z)dz, \quad u|_{t=0} = u_0. 
    \label{eq:Hunter-Saxton}
\end{align}
Originally proposed as an asymptotic model describing nonlinear instabilities in the director field of nematic liquid crystals in \cite{DynamicsDirector}, the equation has been subject to extensive research due to its intriguing properties. A main concern, from a numerical perspective, is the fact that weak solutions in general develop singularities within finite time, causing a loss of uniqueness for weak solutions. 

The formation of singularities is referred to as wave breaking. Any time $t_c$ for which wave breaking takes place, is characterized by the existence of a null set $E \subset \R$ such that for any $x \in E$, $u_x(t, x) \rightarrow -\infty$ pointwise as $t \uparrow t_c$, while $u(t, \cdot)$ remains H{\"o}lder continuous and bounded, and $u_{x}(t, \cdot)$ belongs to $L^2(\R)$, see \cite{DafermosContinuous, AlphaHS}. Furthermore, energy can concentrate on sets of zero measure. Consequently, in order to properly describe the solution $u$ along with its energy, a finite, nonnegative Radon measure $\mu$ is introduced. This measure satisfies $d\mu_{\mathrm{ac}} = u_x^2dx$ and contains all the information about wave breaking. In particular, it is precisely at the times for which $d\mu \neq u_x^2dx$ that energy can concentrate on sets of zero measure.

The task of continuing weak solutions past wave breaking is intricate, and various continuations have been proposed and studied in the literature, see \cite{BressanDissipative,  HSConservative, DafermosCharacteristics, AlphaHS, DynamicsDirector, NonlinearVariational}. A common denominator for these is that they require the energy to be nonincreasing as a function of time, thereby implicitly imposing the measured-valued inequality
\begin{equation}\label{eq:measure_ineq}
	\mu_t  + (u\mu)_x \leq 0.
\end{equation}
The particular case of equality, in which case no energy dissipation takes place,  leads to conservative solutions, for which existence was shown in \cite{HSConservative} and uniqueness in \cite{UniquenessConservative}. On the opposite end of the spectrum, one has the family of solutions experiencing maximal energy dissipation \cite{MaximalDissipation} -- known as dissipative solutions, for which existence and uniqueness were established in \cite{BressanDissipative} and \cite{DafermosCharacteristics}, respectively. In this work on the other hand, we focus on the family of $\alpha$-dissipative solutions, where $\alpha \in W^{1, \infty}(\R, [0, 1))$. To be more precise, if wave breaking occurs at $(t, x) = (t_c, x_c)$, an $\alpha(x_c)$-fraction of the concentrated energy is removed. Thus, the amount of energy removed at any wave breaking depends on the spatial location of its occurrence. Existence of such solutions was established in \cite{AlphaHS}, whilst uniqueness is still an open question.

Over a time interval where no wave breaking takes place, one can apply the method of characteristics to show that the HS equation is formally equivalent to a system of ODEs. This is done in \cite{HSConservative} and \cite{AlphaHS}. However,  at wave breaking, the classical method of characteristics is no longer applicable. To overcome this issue, one introduces a mapping which transforms the Eulerian initial data $(u_0, \mu_0, \nu_0)$, into a tuple $X_0 = (y_0, U_0, V_0, H_0)$ in Lagrangian coordinates, see \cite{HSConservative, UniquenessConservative, AlphaHS, TandyLipschitz}. Here $\nu_0$ is a measure  added exclusively for technical reasons, in the sense that the same solution pair $(u, \mu)(t)$ for $t\geq 0$, is recovered for any choice of $\nu_0$, see \cite[Lem. 2.13]{NewestLipschitzMetric}. Under this coordinate transformation, the HS equation rewrites into a system of ODEs that attains unique, global solutions. In particular, the evolution of the $\alpha$-dissipative solution in Lagrangian coordinates emanating from $X_0$ is governed by, see \cite{AlphaHS}, 
\begin{subequations} 
\label{eq:IntroLagrSys}
\begin{align}
	y_t(t, \xi) &= U(t, \xi), \label{eq:IntroLagrSys1} \\
	U_t(t, \xi) &= \frac{1}{2}V(t, \xi) - \frac{1}{4}V_{\infty}(t),\\
	V(t, \xi) &= \int_{-\infty}^{\xi} \left(1 - \alpha(y(\tau(\eta), \eta))\chi_{\{\omega: t \geq \tau(\omega) > 0\}}(\eta) \right)\!V_{0, \xi}(\eta)d\eta,  \label{eq:IntroLagrSys3}\\
	H_t(t, \xi) &= 0. 
\end{align}
\end{subequations}
Here $\displaystyle V_{\infty}(t) =\lim_{\xi\to\infty} V(t, \xi)$ is the total cumulative energy in the system at time $t$ and $\tau: \R \rightarrow [0, \infty]$ is the wave breaking function given by 
\begin{align}\label{eq:IntroWaveFunc}
	\tau(\xi) &= \begin{cases}
		0, & U_{0, \xi}(\xi) = y_{0, \xi}(\xi) = 0, \\
		- \frac{2y_{0, \xi}}{U_{0, \xi}}(\xi), & U_{0, \xi}(\xi) < 0, \\
		\infty, & \text{otherwise}. 
	\end{cases}
\end{align}
Note, the subindex $0$ refers to the initial data, whereas a subindex $\xi$ or $t$ refers to partial derivatives with respect to $\xi$ or $t$, respectively, for instance, $y_{0, \xi}(\xi) = \partial_{\xi}y_0(\xi)$. We will stick with this notation throughout the paper. 

One way to physically interpret the Lagrangian quantities is to consider $\xi$ as a variable identifying a particle. Then $y(t, \xi)$ represents the particle trajectory emanating from $y_0(\xi)$, and $\tau(\xi)$ the collision time of the particle. This way, $\tau(\xi)=\infty$ indicates that the particle does not participate in any collision. Furthermore, $U(t, \xi) = u(t, y(t, \xi))$ represents the particle velocity, while $V_{\xi}(t, \xi)$ is its energy. Finally, $H(t, \xi)$ is a conserved quantity stemming from $\nu$, lacking an obvious physical interpretation.  

Naturally, there has been proposed a variety of numerical methods for the HS equation. A collection of finite differences schemes of upwind-type were proved to converge towards dissipative solutions in \cite{FDDissipative}.  In \cite{DGMethod1} and \cite{DGMethod2} two discontinuous Galerkin methods for conservative and dissipative solutions, respectively, were introduced. The latter one turns out to be equivalent with a finite-difference scheme from \cite{FDDissipative} when choosing piecewise constant DG-elements, and it therefore converges. 
A convergent scheme of Godunov-type tailored to conservative solutions was introduced in \cite{NumericalConservative}. In addition, this scheme was shown to converge in $L^{\infty}(\mathbb{R})$ with order a half prior to wave breaking. Moreover, both multi-symplectic and Hamiltonian-preserving discretizations were proposed in \cite{GeometricInt}, albeit without any proof of convergence. 

One drawback with the aforementioned methods, is that they are all tailored to one particular class of weak solutions; either conservative or dissipative. More recently a convergent numerical method for $\alpha$-dissipative solutions was proposed in \cite{AlphaAlgorithm}, where $\alpha \in [0, 1]$ is held fixed. A major benefit of this scheme is its ability to continuously interpolate between conservative ($\alpha=0$) and dissipative ($\alpha=1$) solutions. In other words, it provides a numerical framework for uniformly treating solutions with nonincreasing energy. 

The method in \cite{AlphaAlgorithm} relies on the observation that if the initial data $u_0$ for \eqref{eq:Hunter-Saxton} is piecewise linear, the solution $u(t, \cdot)$ remains piecewise linear for $t\geq 0$, see \cite{PropertiesHS}. Taking advantage of this fact, the scheme combines a novel piecewise linear projection operator, which preserves some of the essential underlying structure of solutions to \eqref{eq:Hunter-Saxton}, with exact evolution via the generalized method of characteristics, i.e., according to \eqref{eq:IntroLagrSys}. The algorithm presented here can be viewed as an extension of the one from \cite{AlphaAlgorithm}, as it is applicable to a larger class of $\alpha$-dissipative solutions, but it is still based on the same underlying idea. However, some new challenges appear due to $\alpha \in W^{1, \infty}(\R, [0, 1))$.

Firstly, the amount of energy to be removed at any wave breaking occurrence cannot be computed a priori, cf. \eqref{eq:IntroLagrSys3}. In particular, if the solution experiences wave breaking along the characteristic labeled by $\xi$, then the fraction of removed energy is $\alpha(y(\tau(\xi), \xi))$. However, the time evolution of $y(t, \xi)$ depends heavily on $V_{\xi}(t, \cdot)$ on all of $\R$, see \eqref{eq:IntroLagrSys1}. Consequently, the value of $y(\tau(\xi), \xi)$ depends on all other wave breaking occurrences taking place prior to $t=\tau(\xi)$. To resolve this issue, we introduce an iteration scheme, inspired by \cite[Lem. 2.3]{AlphaHS}, which is based on computing successive approximations of the energy to be removed at wave breaking.

A second challenge is posed by the possibility of accumulating wave breaking times. This was not a problem in  \cite{AlphaAlgorithm}, because when $\alpha$ is constant, the time evolution of $X_{\xi}(t) = (y_{\xi}, U_{\xi}, V_{\xi}, H_{\xi})(t)$ is given by a closed system of ODEs.  However, as already observed, $V_{\xi}(t, \xi)$ and $y(\tau(\xi), \xi)$ are closely intertwined for  $\alpha \in W^{1, \infty}(\R, [0, 1))$, cf. \eqref{eq:IntroLagrSys3}. Thus, the differentiated Lagrangian variables can be evolved between successive breaking times, but in order to determine the correct amount of energy to be removed at wave breaking, one needs to compute $y(t, \cdot)$. To avoid the risk of having to compute $y(t, \cdot)$ repeatedly at nearly identical wave breaking times, we extract a finite sequence of breaking times $\{\tau_k^*\}$, which acts as a non-uniform temporal discretization of $[0, T]$, for any finite $T \geq 0$. The proposed numerical scheme then evolves the differentiated Lagrangian variables, with the iteration scheme, between successive times in this extracted sequence $\{\tau_k^*\}$ until the final time $T$ is reached.

The content of this paper is organized as follows. In Section~\ref{sec:Background} we present the generalized method of characteristics based on \cite{HSConservative, AlphaHS}, before we in Section~\ref{sec:NumericalImplementation} recall the projection operator from \cite{AlphaAlgorithm} and derive the numerical algorithm. This derivation consists of: motivating the minimal time evolution criterion, presenting the iteration scheme, and discussing the numerical implementation.  In addition, we show that the proposed iteration scheme terminates and derive a uniform bound on the number of iterates. Section~\ref{sec:ConvergenceMethod} is devoted to proving convergence for any $t\in [0, T]$. This is done in two steps. First we equip the set of Lagrangian coordinates by a metric introduced in \cite{AlphaHS} which allows us to analyze the error introduced by the projection operator. We thereafter introduce a function, inspired by the metric from \cite{NewestLipschitzMetric}, which provides an upper bound on the metric from \cite{AlphaHS} and allows us to estimate the error introduced by the iteration scheme. Finally, in Section~\ref{sec:NumericalExperiments} we complement the numerical analysis by applying the method on two examples to illustrate the effect of the minimal time step.

\section{$\alpha$-dissipative solutions and their stability}\label{sec:Background}
In this section we outline the construction of $\alpha$-dissipative solutions via the generalized method of characteristics from  \cite{HSConservative} and \cite{AlphaHS}. We restrict our attention to the sets of Eulerian and Lagrangian coordinates, relevant mappings between these sets and a discussion about the time evolution in Lagrangian coordinates. In addition, we recall the definition of a metric from \cite{AlphaHS}, which renders $\alpha$-dissipative solutions stable with respect to the initial data in Lagrangian coordinates. 

Let $\alpha \in W^{1, \infty}(\R, [0, 1))$ be fixed and denote by $\M^{+}(\R)$ the space of positive, finite Radon measures on $\R$. To introduce the set of Eulerian coordinates, we need to recall some function spaces from \cite{HSConservative} and \cite{AlphaHS}. 

We start by defining the Banach space
\begin{equation*}
    E := \left \{ f \in L^{\infty}(\R): f' \in L^2(\R) \right \}\!,
\end{equation*}
equipped with the norm 
\begin{equation*}
    \|f \|_{E} := \|f \|_{\infty} + \|f' \|_{2}.
\end{equation*}
Moreover, let
\begin{align*}
    H_d^1(\R) &:= H^1(\R) \times \R^d, \quad d = 1, 2,
\end{align*}
and choose a partition of unity $\chi^+$ and $\chi^-$ on $\R=(-\infty, 1) \cup (-1, \infty)$, i.e., a pair of functions $\chi^+, \chi^{-} \in C^{\infty}(\R)$ satisfying 
\begin{itemize}
    \item $\chi^+ + \chi^- = 1$,
    \item $0 \leq \chi^{\pm} \leq 1$,
    \item $\text{supp}(\chi^+) \subset (-1, \infty)$ and $\text{supp}(\chi^-) \subset (-\infty, 1)$.
\end{itemize}
This allows us to define the following linear, continuous, and injective mappings 
\begin{align*}
    R_1: H_1^1(\R) \rightarrow E, & \qquad  (\bar{f}, a) \mapsto f= \bar{f} + a \chi^+\!, \\
    R_2: H_2^1(\R) \rightarrow E, & \qquad  (\bar{f}, a, b) \mapsto f= \bar{f} + a \chi^+ + b \chi^-\!.
\end{align*} 
Accordingly, we introduce the Banach spaces $E_1$ and $E_2$ as the images of $H_1^1(\R)$ and $H_2^1(\R)$, respectively, i.e., 
\begin{align*}
    E_1 := R_1(H_1^1(\R))  \quad \text{and} \quad 
    E_2 := R_2(H_2^1(\R)), 
\end{align*}
endowed with the following norms
\begin{align*}
    \|f\|_{E_1} &:= \|\bar{f} + a \chi^+ \|_{E_1} =  \left( \|\bar{f}\|_{H^1(\R)}^2 + a^2 \right )^{\frac{1}{2}} \!\!, \\
    \|f\|_{E_2} &:= \|\bar{f} + a \chi^+ + b\chi^- \|_{E_2} = \left( \|\bar{f}\|_{H^1(\R)}^2 + a^2 + b^2 \right )^{\frac{1}{2}} \!\!.
\end{align*}
By construction, the spaces $E_1$ and $E_2$ do not rely on the particularly chosen partition of unity, see \cite{GlobalDissipative2CH}. Furthermore, the mapping $R_1$ is also well-defined when applied to functions in $L_1^2(\R) = L^2(\R) \times \R$. Consequently, let
\begin{equation*}
    E_1^0 := R_1(L_1^2(\R)),
\end{equation*}
and equip this space with the norm 
\begin{equation*}
    \|f \|_{E_1^0} := \|\bar{f} + a \chi^+ \|_{E_1^0} = \left( \|\bar{f} \|_2^2 + a^2 \right)^{\frac{1}{2}} \!.
\end{equation*}
Finally, we can define the set of Eulerian coordinates $\D^{\alpha}$.

\begin{definition}\label{def:EulSet}
   The space $\D^{\alpha}$ is composed of all triplets $(u, \mu, \nu)$ such that
    \begin{enumerate}[label=(\roman*)]
        \item $u \in E_2$,
        \item $\mu \leq \nu \in \M^{+}(\R)$, \label{def:EulSet2}
        \item $\mu_{\mathrm{ac}} \leq \nu_{\mathrm{ac}}$,
        \item $d\mu_{\mathrm{ac}} = u_x^2dx$,
        \item $\mu( (-\infty, \cdot)) \in E_1^0$,
        \item $\nu((-\infty, \cdot))  \in E_1^0$,
        \item $\frac{d\mu}{d\nu}(x) > 0$, and $\frac{d\mu_{\mathrm{ac}}}{d\nu_{\mathrm{ac}}}(x) = 1$ if $u_x(x) < 0$. 
    \end{enumerate}
\end{definition}
Since $\mu, \nu \in \M^+(\R)$ we can introduce the primitive functions $F(\cdot) = \mu((-\infty, \cdot))$ and $G(\cdot) = \nu((-\infty, \cdot))$. These are bounded, left-continuous, increasing, and satisfy
\begin{equation*}
    \lim_{x \rightarrow - \infty} F(x) = \lim_{x \rightarrow -\infty} G(x) = 0. 
\end{equation*}
\vspace{-0.05cm}
By \cite[Thm. 1.16]{RealAnalysisFolland} there is a one-to-one relationship between $(F, G)$ and $(\mu, \nu)$. We will therefore interchangeably use the notation $(u, F, G)$ and $(u, \mu, \nu)$ to refer to the same triplet in $\D^{\alpha}$. Furthermore, $F$ admits a decomposition into an absolutely continuous part, $F_{\mathrm{ac}}(x)\! =\! \mu_{\mathrm{ac}}((-\infty, x))$, and a singular part, $F_{\mathrm{sing}}(x)\! =\! \mu_{\mathrm{sing}}((-\infty, x))$, i.e., 
\begin{equation*}
	F(x) = F_{\mathrm{ac}}(x) + F_{\mathrm{sing}}(x). 
\end{equation*}
A similar decomposition exists for $G$ as well. 

Before we introduce the set of Lagrangian coordinates,  $\F^{\alpha}$, let $B:=E_2 \times E_2 \times E_1 \times E_1$, equipped with the norm 
\begin{equation*}
   \|(f_1, f_2, f_3, f_4) \|_B := \|f_1 \|_{E_2} + \|f_2 \|_{E_2} + \|f_3 \|_{E_1} + \|f_4 \|_{E_1} \!.
\end{equation*}

\begin{definition}\label{def:LagSet} The set $\F^{\alpha}$ contains all quadruplets $X = (y, U, V, H)$ with \phantom{a} $(y-\id, U, V, H) \in B$ satisfying
\begin{enumerate}[label=(\roman*)]
  \item $(y-\id, U, V, H) \in \big[W^{1, \infty}(\R) \big]^4$,
  \item $ y_{\xi}, H_{\xi} \geq 0$ and there exists $c > 0$ such that $y_{\xi} + H_{\xi} \geq c $ a.e.,
  \item \label{def:importantRelation} $y_{\xi}V_{\xi}= U_{\xi}^2 $ a.e., 
  \item $0 \leq V_{\xi} \leq H_{\xi}$ a.e., 
  \item There exists $\kappa: \R \rightarrow (0, 1]$ such that $V_{\xi}(\xi) = \kappa(y(\xi))H_{\xi}(\xi)$ a.e., with $\kappa(y(\xi))=1$ whenever $U_{\xi}(\xi) < 0$. 
\end{enumerate}
\end{definition}
 The following subsets of $\F^{\alpha}$ will play an important role in the upcoming convergence analysis, 
\begin{align*}
    \F^{\alpha, 0} &= \{X \in \F^{\alpha}: H(\xi) = V(\xi) \text{ for all }\xi \in \R \}, 
    \\
    \F_0^{\alpha} &= \{ X \in \F^{\alpha}: y + H = \id \}, 
\end{align*}
and the intersection of these two  
\begin{equation*}
    \F_0^{\alpha, 0} = \F_0^{\alpha} \cap \F^{\alpha, 0}\!.
\end{equation*} 

The construction of $\alpha$-dissipative solutions via a generalized method of characteristics is based on examining the time evolution in Lagrangian coordinates rather than in Eulerian coordinates. Thus, the mappings between $\D^{\alpha}$ and $\F^{\alpha}$ are essential.

\begin{definition}\label{def:MapL}Let $L: \D^{\alpha} \rightarrow \F_0^{\alpha}$ be defined by $L((u, \mu, \nu)) = (y, U, V, H)$, where
\begin{subequations}
\begin{align}
    y(\xi) &= \sup \{x \in \R: x + \nu((-\infty, x)) < \xi \}, \label{eq:L_eq1}\\
    U(\xi) &= u (y(\xi)), \label{eq:L_eq2} \\
    H(\xi) &= \xi - y(\xi), \label{eq:L_eq3} \\
    V(\xi) &= \int_{-\infty}^{\xi} \frac{d\mu}{d\nu} (y(\eta)) H_{\xi}(\eta) d\eta \label{eq:L_eq4}. 
\end{align}
\end{subequations}
\end{definition}

\begin{definition}\label{def:MapM}
    Define $M:\F^{\alpha} \rightarrow \D^{\alpha}$ by $M((y, U, V, H)) = (u, \mu, \nu)$, where 
    \begin{subequations}
        \begin{align}
            u(x) &= U(\xi) \quad \text{ for any } \xi \in \R \text{ such that } x = y(\xi), \\ 
            \mu &= y_{\#} \! \left (V_{\xi}d\xi \right)\!, \\
            \nu &= y_{\#}(H_{\xi}d\xi).
        \end{align}
    \end{subequations}
   Here $y_{\#}(V_{\xi}d\xi)$ denotes the pushforward of the measure $V_{\xi}d\xi$ by the function $y$. 
\end{definition}

Well-definedness of these transformations follows from \cite[Prop. 2.1.5 and 2.1.7]{PhdThesisNordli}. Furthermore, we emphasize that triplets $(u, \mu, \nu)$ are mapped to quadruplets $(y, U, V, H)$. Hence there cannot be a one-to-one correspondence between Eulerian and Lagrangian coordinates. However, one can identify equivalence classes in Lagrangian coordinates, such that each equivalence class corresponds to exactly one triplet in Eulerian coordinates, see \cite{PhdThesisNordli}. 

Moreover, it should be pointed out that all the important information about the solution in Eulerian coordinates is contained in $(u, \mu)$, and therefore encoded in the triplet $(y, U, V)$ in Lagrangian coordinates. Despite this, $\nu$ heavily influences the mapping $L$, in the sense that changing $\nu$ does not only change $H$, but also $(y, U, V)$. 

Up to this point the formalism has been stationary, transforming back and forth between Eulerian and Lagrangian coordinates. Next, we turn our attention to the time evolution. 

The $\alpha$-dissipative solution in Lagrangian coordinates, $X(t) = (y, U, V, H)(t)$, with initial data $X(0) = X_0 \in \F^{\alpha}$, is the unique solution to the following system of differential equations 
\begin{subequations}
\label{eq:intLagrSystem}
\begin{align}
    y_t(t, \xi) &= U(t, \xi), \label{eq:int_ODE1} 
    \\
    U_t(t, \xi) &= \frac{1}{2}V(t, \xi) - \frac{1}{4}V_{\infty}(t), \label{eq:int_ODE2} 
    \\
    V(t, \xi) &= \int_{-\infty}^{\xi}\!\left(1 - \alpha(y(\tau(\eta), \eta)) \chi_{\{\omega: t \geq \tau(\omega) > 0\}}(\eta) \right)\! V_{0, \xi}(\eta) d\eta, \label{eq:int_ODE3} 
    \\
    H_t(t, \xi) &= 0. \label{eq:int_ODE4}
\end{align}
\end{subequations}
Here  $\tau: \R \rightarrow [0, \infty]$ is the wave breaking function given by
\begin{align}\label{eq:waveBreakingFunc}
    \tau(\xi) &= \begin{cases}
     0, & \text{if }y_{0, \xi}(\xi) = U_{0, \xi}(\xi) = 0, \\ 
    - \frac{2y_{0, \xi}(\xi)}{U_{0, \xi}(\xi)}, & \text{if } U_{0, \xi}(\xi) < 0,\\
    \infty, & \text{otherwise},
    \end{cases}
\end{align}
and \small$V_{\infty}(t) = \lim_{\xi \rightarrow \infty}V(t, \xi)$ \normalsize is the total Lagrangian energy at time $t$. For later use, we also introduce \small $H_{\infty} = \lim_{\xi \rightarrow \infty}H(t, \xi) =  \lim_{\xi \rightarrow \infty} H_0(\xi)$\normalsize. 

Note that the right-hand side of the system \eqref{eq:intLagrSystem} becomes discontinuous at fixed times given by the function \eqref{eq:waveBreakingFunc}. Nevertheless, existence and uniqueness of solutions is shown in \cite[Lem. 2.3]{AlphaHS} through a standard fixed point argument. 

Based on \eqref{eq:intLagrSystem} we can define the solution operator $S_t$.

\begin{definition}\label{def:SolOPLagr}
    For any $t \geq 0$ and $X_0 \in \F^{\alpha}$, define $S_t(X_0) = X(t)$, where $X(t)$ denotes the unique $\alpha$-dissipative solution to \eqref{eq:intLagrSystem} with initial data $X(0) = X_0$. 
\end{definition}

Finally, we obtain the $\alpha$-dissipative solution in Eulerian coordinates by combining the solution operator $S_t$ with the mappings $L$ and $M$. 
\begin{definition}\label{def:AlphaSol}
  For $(u_0, \mu_0, \nu_0) \in \D^{\alpha}$, the $\alpha$-dissipative solution at time $t\geq 0$ is given by 
    \begin{equation*}
        (u, \mu, \nu)(t) = T_t((u_0, \mu_0, \nu_0)) = M \circ S_t \circ L ((u_0, \mu_0, \nu_0)).
    \end{equation*}
\end{definition}

As mentioned  earlier, $L((u, \mu, \nu))$ is dependent on the choice of $\nu$. In spite of that, it has been established in \cite[Lem. 2.13]{NewestLipschitzMetric} that the choice of $\nu$ has no influence on $(u, \mu)$ in the following sense: Given $(u_{\mathrm{A}, 0}, \mu_{\mathrm{A}, 0}, \nu_{\mathrm{A}, 0})$ and $(u_{\mathrm{B}, 0}, \mu_{\mathrm{B}, 0}, \nu_{\mathrm{B}, 0})$ in $ \D^{\alpha}$ such that
\begin{equation*}
    u_{\mathrm{A}, 0} = u_{\mathrm{B}, 0} \hspace{0.3cm} \text{and} \hspace{0.3cm} \mu_{\mathrm{A}, 0} = \mu_{\mathrm{B}, 0}, 
\end{equation*}
it then holds that
\begin{equation*}
    u_{\mathrm{A}}(t) = u_{\mathrm{B}}(t) \hspace{0.3cm} \text{and} \hspace{0.3cm} \mu_{\mathrm{A}}(t) = \mu_{\mathrm{B}}(t) \hspace{0.3cm} \text{for all } t \geq 0. 
\end{equation*}
Therefore, we will from now on restrict ourselves to initial data $(u_0, \mu_0, \nu_0)$ belonging to the subset 
\begin{equation}\label{eq:SubsetD}
    \D_0^{\alpha} := \{(u, \mu, \nu) \in \D^{\alpha}: \mu = \nu \} \subset \D^{\alpha} \!. 
\end{equation}

To simplify the upcoming convergence analysis, we equip the set $\F^{\alpha}$ with a metric rendering $\alpha$-dissipative solutions in Lagrangian coordinates stable with respect to the Lagrangian initial data. Various metrics have been introduced for this very purpose, see \cite{AlphaHS} and \cite{NewestLipschitzMetric}. In our setting, it turns out to be most convenient to work with the metric $d$ presented in \cite[Def. 4.6]{AlphaHS}. This metric relies on a function, $g_1(X)$, which models the energy loss at wave breaking continuously in time, in contrast to the actual energy density $V_{\xi}$, which may drop abruptly. It is defined by 
 \begin{align}\label{eq:g}
        g_1(X)(\xi) &:= \begin{cases}
        (1 - \alpha(y(\xi)))V_{\xi}(\xi), & \xi \in \Omega_d(X), \\
        V_{\xi}(\xi), & \xi \in \Omega_c(X),
     \end{cases}
    \end{align}
 where 
\begin{align}\label{def:omegas}
    \Omega_d \left(X \right) &:= \left \{ \xi \in \R \!: U_{\xi}(\xi) < 0 \right\} \!, \notag 
    \\
    \Omega_c(X) &:= \{\xi \in \R \!: U_{\xi}(\xi) \geq 0 \}. 
\end{align}
Here $\Omega_d$ and $\Omega_c$ split $\R$ into a set of points where the $\alpha$-dissipative solution will eventually experience wave breaking,  $\Omega_d$, and its complement, $\Omega_c$, containing all points where no wave breaking takes place in the future. 
   
 Furthermore, the metric also involves the two additional functions 
\begin{align}
    g_2(X)(\xi) &:= \begin{cases}
    \|\alpha ' \|_{\infty} H_{\infty}U_{\xi}(\xi), & \xi \in \Omega_d(X), \\
    0, & \xi \in \Omega_c(X), \end{cases} \label{eq:g_2}\\
    g_3(X)(\xi) &:= \begin{cases}
    \|\alpha ' \|_{\infty}UU_{\xi}(\xi), & \xi \in \Omega_d(X), \\
    0, & \xi \in \Omega_c(X). \end{cases} \label{eq:g_3}
\end{align}

\begin{definition}\label{def:metric}
    Define $d: \F^{\alpha} \times \F^{\alpha} \rightarrow [0, \infty)$ by  
    \begin{align*}
        d(X, \hat{X}) &= \|y - \hat{y} \|_{\infty} + \|U-\hat{U}\|_{\infty} +  \|H_{\xi} - \hat{H}_{\xi} \|_1 + \|\alpha ' \|_{\infty}\|UH_{\xi} - \hat{U}\hat{H}_{\xi} \|_2 \\
        & \quad + \|y_{\xi} - \hat{y}_{\xi}\|_2 + \|U_{\xi} - \hat{U}_{\xi}\|_2  + \|g_1(X) + y_{\xi} - g_1(\hat{X}) - \hat{y}_{\xi} \|_2 \\
        & \quad  + \|H_{\xi} - \hat{H}_{\xi}\|_2 + \|g_2(X) - g_2(\hat{X}) \|_2 + \|g_3(X) - g_3(\hat{X}) \|_2.
    \end{align*}
\end{definition}

Although originally introduced as a preliminary metric, since it separates Lagrangian coordinates belonging to the same equivalence class, this metric is well-suited for our analysis due to the following stability estimate.

\begin{theorem}[{\cite[Thm. 4.18]{AlphaHS}}] \label{thm:Lipschitz} Let $X_0\! \in\! \F^{\alpha, 0}$ and $\hat{X}_0 \!\in \!\F^{\alpha, 0}_0$. Furthermore, assume that $\max(H_{\infty}, \hat H_{\infty})\leq M$ and  $t \geq 0$, then
	\begin{equation}\label{eq:LipschitzEst}
	d \left (S_t(X_0), S_t(\hat{X}_0) \right) \leq C(t)e^{D(t)t}d(X_0, \hat{X}_0),
	\end{equation}
	where $C(t) = C(t, M, \|\alpha '\|_{\infty})$ and $D(t) = D(t, M, \|\alpha '\|_{\infty})$ are given by 
	\begin{align}
	C(t) &= 3 + \frac{3}{2}t + \frac{1}{2}t^2 + \frac{3}{16}t^3 + \sqrt{M} \bigg(1 + \frac{1}{4}t + \frac{1}{4}t^2 + \frac{1}{16}t^3 \bigg) \nonumber
	\\ & \quad + \|\alpha ' \|_{\infty} \sqrt{M} \bigg(5 + 2t + t^2 + \frac{3}{8}t^3 \bigg)  +  \|\alpha ' \|_{\infty} M \bigg(3 + \frac{5}{4}t + \frac{1}{2}t^2 + \frac{1}{8}t^3 \bigg), \label{eq:func_C}\raisetag{-20pt}\\
	D(t) &= 2 + \|\alpha '\|_{\infty}\sqrt{M} + \sqrt{M} \bigg(\frac{1}{2} + \frac{1}{8}t + \frac{1}{16}t^2 \bigg)  + \|\alpha '\|_{\infty}M \bigg(1 + \frac{1}{4}t + \frac{1}{8}t^2 \bigg). \nonumber
	\end{align}
\end{theorem}

\section{The numerical algorithm}\label{sec:NumericalImplementation}
This section is devoted to the derivation of the numerical algorithm.  As mentioned in the introduction, the energy loss at any wave breaking occurrence is unknown a prior.  To overcome this obstacle we combine an iteration scheme, producing approximations of \eqref{eq:intLagrSystem}, with a minimal time evolution criterion, which reduces the computational complexity, especially when accumulating breaking times are present. The latter phenomenon is thoroughly discussed for the related Camassa--Holm equation in \cite{AccumulationCH}.

\subsection{The projection operator}

We start by recalling the projection operator $P_{\Dx}$ introduced in \cite[Def. 3.2]{AlphaAlgorithm}. 
Let $\{x_j \}_{j \in \mathbb{Z}}$ be a uniform discretization of $\R$, where $x_j = j \Delta x$ for $j \in \mathbb{Z}$ and $\Delta x > 0$. Furthermore, to ease the notation, introduce for any function $f:\R \rightarrow \R$ the sequence $\{f_j\}_{j \in \mathbb{Z}}$ where $f_j = f(x_j)$ and the difference operators,
\begin{align*}
    D_+f_j &:= \frac{f_{j+1} - f_j}{\Delta x}, \\ 
    Df_{2j} &:= \frac{D_+ f_{2j+1} + D_+f_{2j}}{2} = \frac{f_{2j+2} - f_{2j}}{2\Delta x}.
\end{align*}
Then we can define the projection operator $P_{\Dx}$ as follows.  
\begin{definition}\label{def:ProjOp}
    Define $P_{\Dx}:\! \D_0^{\alpha}\! \rightarrow\! \D_0^{\alpha}$ by $P_{\Delta x} \!\left((u, F, G) \right)\! =\!(u_{\Delta x}, F_{\Delta x}, G_{\Delta x})$, where  
     \begin{align}
        u_{\Delta x}(x) &= \begin{cases}
        u_{2j} + \left(Du_{2j} \mp q_{2j} \right)(x-x_{2j}), & x_{2j} < x \leq x_{2j+1}, 
        \\ 
        u_{2j+2} + \left(Du_{2j} \pm q_{2j} \right) \left(x-x_{2j+2} \right) \!, & x_{2j+1} < x \leq x_{2j+2}, 
        \end{cases}
        \label{eq:Proj_u}
    \end{align}
    with $q_{2j}$ given by
    \begin{equation*}
    q_{2j} := \sqrt{DF_{\mathrm{ac}, 2j} - (Du_{2j})^2},
\end{equation*}
      and
    \begin{align}
        G_{\Delta x}(x) &= F_{\Delta x}(x) = F_{\Delta x, \mathrm{ac}}(x) + F_{\Delta x, \mathrm{sing}}(x), \quad \text{ for all } x\in \R,  
        \label{eq:Proj_F}
    \end{align}
   where the absolutely continuous and singular part are given by 
    \begin{align}
    F_{\Delta x, \mathrm{ac}}(x) &= \begin{cases} F_{\mathrm{ac}, 2j} + \left(Du_{2j} \mp q_{2j}\right)^2 \left(x - x_{2j} \right), & x_{2j} < x \leq x_{2j+1}, 
    \\ 
    \frac{1}{2}\left(F_{\mathrm{ac}, 2j+2} + F_{\mathrm{ac}, 2j}\right) \mp 2Du_{2j} q_{2j}\Delta x \\ \quad + \left(Du_{2j} \pm q_{2j} \right)^2(x-x_{2j+1}), & x_{2j+1} < x \leq x_{2j+2},
    \end{cases}
    \label{eq:Proj_Fac} \end{align}
    and 
    \begin{align}
        F_{\Delta x, \mathrm{sing}}(x) &= F_{2j+2} - F_{\mathrm{ac}, 2j+2}= F_{\mathrm{sing}, 2j+2}, & x_{2j} < x \leq x_{2j+2},
        \label{eq:Proj_Fs}
    \end{align}
    respectively.
\end{definition}

\begin{remark}
	Note that \eqref{eq:Proj_u} and \eqref{eq:Proj_Fac} leave us with two possible sign choices over each interval $[x_{2j}, x_{2j+2}]$. The accuracy of the numerical approximation is affected by the choice of sign, see e.g. \cite[Ex. 3.4]{AlphaAlgorithm}. We therefore implement the sign-selection criterion from \cite{AlphaAlgorithm}, which is based on minimizing the difference between $u(x_{2j+1})$ and $u_{\Dx}(x_{2j+1})$. In particular, let
\begin{equation*}
	k_{2j} := \argmin_{m \in \{0, 1\}} \left \{ (D_+u_{2j} - Du_{2j})\Dx + (-1)^{m+1}q_{2j}\Dx \right \},
\end{equation*}
then we use the sign $(-1)^{k_{2j}}$ over $[x_{2j}, x_{2j+1}]$ and $(-1)^{k_{2j}+1}$ over $[x_{2j+1}, x_{2j+2}]$. 
\end{remark}

\subsection{Numerical implementation of the algorithm}

In this section the focus is on the implementation of the algorithm. For this purpose, fix a discretization parameter $\Delta x > 0$, a final time $T \geq 0$, and an initial datum $(u, F, G)(0) \in \D_0^{\alpha}$. 
\subsubsection{Implementation of $L$ and computation of $\tau_{\Dx}$}
The numerical Lagrangian initial data is given by  
\begin{equation}
    X_{\Delta x}(0) = \left(y_{\Delta x}, U_{\Delta x}, V_{\Delta x}, H_{\Delta x} \right)(0) = L \circ P_{\Delta x} \left((u, F, G)(0) \right) \!. 
\label{eq:numLagrInitial}
\end{equation}
For a  detailed account about how to numerically implement the mapping $L$ the interested reader is referred to \cite[Sec. 3.2.1]{AlphaAlgorithm}. We here emphasize that the uniform grid in Eulerian coordinates $\{x_j\}_{j \in \mathbb{Z}}$ is transformed into a nonuniform Lagrangian discretization $\{\xi_{j}\}_{j \in \mathbb{Z}}$, where each pair $[x_{2j}, x_{2j+1}]$ and $[x_{2j+1}, x_{2j+2}]$ 
possibly results in three Lagrangian grid cells with gridpoints given by 
\begin{align}\label{eq:LagrGridpoints}
	\xi_{3j} &= x_{2j} + G_{\Dx}(0, x_{2j}), \hspace{1.4cm} \xi_{3j+1} = x_{2j} + G_{\Dx}(0, x_{2j}+), \nonumber \\ 
	\xi_{3j+2} &= x_{2j+1} + G_{\Dx}(0, x_{2j+1}), \hspace{0.7cm}
	\xi_{3j+3} = x_{2j+2} + G_{\Dx}(0, x_{2j+2}),
\end{align} 
where the notation $x_{2j}+$ denotes the right limit at $x_{2j}$. Note that $\xi_{3j} = \xi_{3j+1}$ if $G_{\Dx}(0,\cdot)$ does not admit a jump discontinuity at $x=x_{2j}$. 
Moreover, due to Definition~\ref{def:MapL} and Definition~\ref{def:ProjOp}, $X_{\Dx}(0, \cdot)$ will be continuous and piecewise linear with nodes located at $\{\xi_j\}_{j \in \mathbb{Z}}$. Consequently, one can recover $X_{\Dx}(0, \xi)$ for any $\xi \in \R$ by linear interpolation based on the values attained at $\{\xi_j\}_{j \in \mathbb{Z}}$, i.e, $\{X_j(0)\}_{j\in\mathbb{Z}}= \{X_{\Dx}(0, \xi_j)\}_{j\in\mathbb{Z}}$. 

\begin{figure}
	\centering
	\captionsetup{width=.9\linewidth}
	\includegraphics{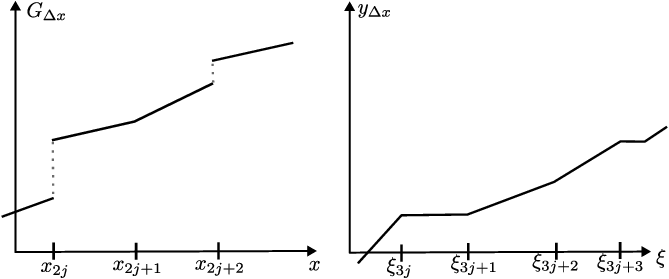}
	\caption{The relation between the Eulerian and Lagrangian gridpoints. Note, that points, where $G_{\Dx}$ has a jump, are mapped to intervals where $y_{\Dx}$ is constant.} 
	\label{fig:eulerToLagrange}
\end{figure}

The relationship between the Eulerian and Lagrangian gridpoints is illustrated in Figure \ref{fig:eulerToLagrange}, from which we observe that $y_{\Dx}(0, \xi) \!= x_{2j}$ for all $\xi \in [\xi_{3j}, \xi_{3j+1}]$ and by \eqref{eq:L_eq2}, $y_{\Dx, \xi}(0, \xi) \!=\! 0 = \!U_{\Dx, \xi}(0, \xi)$ for all $\xi \in [\xi_{3j}, \xi_{3j+1}]$. As a consequence, when computing the numerical wave breaking function $\tau_{\Dx}:\R \rightarrow [0, \infty]$ via \eqref{eq:waveBreakingFunc} with initial data \eqref{eq:numLagrInitial}, it holds that $\tau_{\Dx}(\xi)\! = 0$ for $\xi \in [\xi_{3j}, \xi_{3j+1}]$. Furthermore, observe that $U_{\Dx, \xi}(0, \xi) \!= u_{\Dx, x}(0, y_{\Dx}(0, \xi))y_{\Dx, \xi}(0, \xi)$ for all $\xi \in (\xi_{3j+1}, \xi_{3j+3})\setminus\! \{\xi_{3j+2}\}$, and introduce
\begin{align*}
    \tau_{3j+ \frac{3}{2}} &= \begin{cases} -\frac{2}{Du_{2j} \mp q_{2j}}, & \text{if } Du_{2j} \mp q_{2j} < 0, \\ 
    \infty, & \text{otherwise}, \end{cases} 
    \\
    \tau_{3j+\frac{5}{2}} &=\begin{cases} -\frac{2}{Du_{2j} \pm q_{2j}}, & \text{if } Du_{2j} \pm q_{2j} < 0, \\ 
    \infty, & \text{otherwise}. \end{cases}
\end{align*} 
Then the wave breaking function $\tau_{\Dx}(\xi)$ for $\xi \in [\xi_{3j}, \xi_{3j+3})$ takes the form 
\begin{align}\label{eq:numBreakFunction}
    \tau_{\Delta x}(\xi) &= \begin{cases} 
    0, & \xi \in [\xi_{3j}, \xi_{3j +1}], \\
    \tau_{3j+\frac{3}{2}}, & \xi \in (\xi_{3j + 1}, \xi_{3j+2}], \\
    \tau_{3j + \frac{5}{2}}, & \xi \in (\xi_{3j+2}, \xi_{3j+3}). \end{cases}
\end{align}
To ease notation later on, we will denote by $\tau_{3j+\frac{1}{2}}$ the value attained by $\tau_{\Dx}$ on $[\xi_{3j}, \xi_{3j+1}]$. The reader is referred to \cite[Sec. 3.2.2]{AlphaAlgorithm} for more details about $\tau_{\Dx}$.

\subsubsection{Time evolution and numerical iteration procedure}\label{sec:numericalImp}
The next step is the implementation of the solution operator $S_t$ from Definition~\ref{def:SolOPLagr}. 

Note that for a fixed $\xi \in \R$ with $0<\tau(\xi) < \infty$, we have, due to \eqref{eq:int_ODE3},
\begin{align}\label{eq:energy_den}
	V_{\xi}(t, \xi) &= \begin{cases}
		V_{0, \xi}(\xi), & t < \tau(\xi), \\
		(1-\alpha(y(\tau(\xi), \xi)))V_{0, \xi}(\xi), & \tau(\xi) \leq t,
	\end{cases}
\end{align}
from which it is evident that the amount of energy to be removed depends on $y(\tau(\xi), \xi)$. 
However, the time evolution of $y(t, \xi)$, which is governed by \eqref{eq:int_ODE1}, depends on $V_\xi(t, \cdot)$ on all of $\R$. Consequently, $y(\tau(\xi), \xi)$ will depend on all wave breaking occurrences in $[0, \tau(\xi))$. Thus \eqref{eq:intLagrSystem} cannot be solved explicitly in general. As a consequence, we have to approximate the solution operator $S_t$ numerically. This is done by combining an iteration scheme, which draws inspiration from the one used in the proof of \cite[Lem. 2.3]{AlphaHS}, with a minimal time evolution criterion that reduces the number of times numerical integration is performed. The resulting approximation of $S_t$ is restricted to the time interval $[0, T]$. 

We start by introducing 
\begin{equation*}
    \zeta_{\Dx} := y_{\Dx} - \id, 
\end{equation*}
which belongs to $L^{\infty}(\R)$ in contrast to $y_{\Dx}$. Furthermore, set
\begin{equation*}
    \hat{X}_{\Dx} = (\zeta_{\Dx}, U_{\Dx}, V_{\Dx}, H_{\Dx}). 
\end{equation*}
From a numerical perspective, it is advantageous to compute the time evolution of $\hat{X}_{\Dx, \xi}$ rather than the one of $\hat{X}_{\Dx}$, due to possible drops in $V_{\Dx, \xi}$. However, a major difficulty now, which is not present when $\alpha$ is a constant, see \cite{AlphaAlgorithm}, is the lack of a closed ODE-system for $\hat{X}_{\Dx, \xi}$, since $V_{\Dx, \xi}$ depends on $y_{\Dx}$, cf.  \eqref{eq:energy_den}.

Note that $\hat{X}_{\Dx, \xi}(0)$, given by differentiating \eqref{eq:numLagrInitial}, is a piecewise constant function, whose possible discontinuities are located at the nodes $\{\xi_{j} \}_{j \in \mathbb{Z}}$. As a consequence, we can recover $\hat{X}_{\Dx, \xi}(0)$ from the sequence which associates to each $j \in \mathbb{Z}$ the value attained at the midpoint of the interval $[\xi_j, \xi_{j+1}]$, i.e.,
\begin{align} \label{eq:sec_lagr}
    \hat{X}_{j+\frac{1}{2}, \xi}(0) &= \begin{cases}
    \hat{X}_{\Dx, \xi} \left(0, \tfrac{1}{2}(\xi_{j} + \xi_{j+1}) \right) \!, & \text{if } \xi_j \neq \xi_{j+1}, \\ 
    (0, 0, 0, 0), & \text{otherwise}. 
    \end{cases}
\end{align}
Furthermore, introduce 
\begin{equation}
    \tau_{j+\frac{1}{2}} = \tau_{\Dx} \left( \tfrac{1}{2} (\xi_j + \xi_{j+1} ) \right) \!,
    \label{eq:sec_breaking}
\end{equation}
and let $\{\tilde{\tau}_{j} \}_{j \in \mathbb{N}}$ be the sequence of distinct breaking times sorted in increasing order, i.e., $\tilde{\tau}_{j} < \tilde{\tau}_{j+1}$. In addition, if $\tilde{\tau}_1 = 0$, we relabel the sequence by reducing the indices by one. Otherwise, augment the sequence by the element $\tilde{\tau}_0 = 0$. Either way, the resulting sequence, which might be finite, is denoted by  $\{\hat{\tau}_j \}_{j \in \mathcal{J}}$, where $\mathcal{J} \subset \mathbb{N}_0 \!= \mathbb{N} \cup \{0\}$ and $\hat{\tau}_0 =0$. 

Next, we extract from $\{\hat{\tau}_j \}_{j \in \mathcal{J}}$ a finite sequence of breaking times $\{\tau_k^*\}_{k=0}^{N}$, which constitutes a non-uniform discretization of $[0,T]$. To this end, we introduce a temporal parameter $\Dt>0$, which eventually will be related to $\Delta x$ in order to obtain an upper bound on the number of iterations needed, see \eqref{eq:dt}.

\vspace{-0.25cm}
\begin{algorithm}
\caption{Extraction of the finite time sequence $\{\tau_k^*\}$.}
\label{alg:pseudocode}
\begin{algorithmic}
\STATE Let $\tau_0^* = 0$, $m=0$ and $k=0$.
\STATE \textbf{if} $\hat{\tau}_{m+1} \geq T$ \textbf{then} let $\tau_{k+1}^* = T$ and \textbf{stop}. \vspace{-0.05cm}
\IF{$\hat{\tau}_{m+2} - \tau_{k}^* \leq \Delta t$ \textbf{and} $\hat{\tau}_{m+2} <\! T$}  
        \STATE Let $\tau_{k+1}^* = \max\limits_{j \geq m+2} \left\{\hat{\tau}_j: \tau_{k}^*\! < \! \! \hat{\tau}_j \leq \min\{\tau_{k}^* + \Delta t, T\} \right\}=\hat{\tau}_{m+l}$. 
        \STATE $m \mapsto m+l$.
        \STATE $k \mapsto k+1$.
        \STATE \textbf{if} $\tau_{k}^* = T$  \textbf{then} \textbf{stop}. 
\ELSE \vspace{-0.05cm}
        \STATE $\tau_{k+1}^* = \hat{\tau}_{m+1}$. 
        \STATE $m \mapsto m+1$. 
        \STATE $k \mapsto k+1$. 
       \ENDIF{}
\end{algorithmic}
\end{algorithm}
\vspace{-0.25cm}

We now show that the resulting sequence $\{\tau_k^*\}$ is finite and constitutes a non-uniform discretization of $[0,T]$, with local mesh sizes 
\begin{equation}\label{eq:temporalMesh}
	\Dt_k = \tau_{k+1}^* - \tau_k^*.
\end{equation}
 
\begin{prop}\label{lem:TimestepBound}
The sequence $\{\tau_k^*\}$, extracted according to Algorithm~\ref{alg:pseudocode}, has $N+1$ elements and 
\begin{equation}\label{eq:TimestepBound}
		N< 2\!\left(\frac{T}{\Dt} +1\right)\!.
	\end{equation}
\end{prop}

\begin{proof}
Suppose that the strictly increasing sequence $\{\tau_k^*\}$, which is bounded from above by $T$, is not finite. Pick $k\in \mathbb{N}_0$, then there exists $m\in \mathbb{N}_0$ such that $\tau_k^*=\hat \tau_m<\hat\tau_{m+1}\leq \tau_{k+1}^*$ and one of the following scenarios plays out:
\begin{itemize}
	\item If $\tau_{k+1}^* = \hat{\tau}_{m+1}$, which is only possible if $\hat{\tau}_{m+2} - \tau_k^* > \Dt$, then
	\begin{equation*}
	\tau_{k+2}^* - \tau_k^*\geq  \hat{\tau}_{m+2}- \tau_k^*> \Dt.
	\end{equation*} 
	\item If $\tau_{k+1}^*=\hat{\tau}_{m+l}$ for $l\geq 2$, which is only possible if $\hat \tau_{m+2}- \tau_k^*\leq \Dt$, then the maximum in Algorithm~\ref{alg:pseudocode} ensures that 
	\begin{equation*}
		\tau_{k+2}^* -\tau_k^* \geq \hat{\tau}_{m+l+1} - \hat{\tau}_k^* > \Dt. 
	\end{equation*}
	\end{itemize}
Thus, for any $N\in \mathbb{N}_0$ we have 
\begin{align*}
	(N-1)\Dt = \sum_{k=0}^{N-2}\Dt < \sum_{k=0}^{N-2} \left(\tau_{k+2}^* - \tau_k^*\right) = \sum_{k=0}^{N-2} \Dt_{k+1} + \sum_{k=0}^{N-2}\Dt_k \leq 2T,
\end{align*}
which is impossible for $T>0$ and $\Dt>0$ fixed. As a consequence, the sequence $\{\tau_k^*\}$ has at most $N+1$ elements and 
\begin{equation}\label{eq:Claim}
	\tau_{k+2}^* - \tau_k^* > \Dt, \qquad  \text{ for }k \in \{0, \dots, N-3\},
\end{equation}
which implies that $(N-2)\Dt\leq 2T$.
\end{proof}

With the help of the sequence $\{\tau_k^*\}_{k=0}^N$,  we can approximate the solution to \eqref{eq:intLagrSystem} on $(\tau_k^*, \tau_{k+1}^*]$ by using an iteration scheme, where the number of iterations needed is related to the accuracy of the approximation, and hence dependent on the choice of $\Dt$ and $\Dx$.

This iteration scheme is defined inductively as follows. 
Assume that the sequence $\{\hat X_{j+\frac12, \xi}(t)\}_{j\in \mathbb{Z}}$ has already been computed for $t\in [0, \tau_k^*]$ and introduce the sequence $\{\hat{X}_{j+\frac{1}{2}, \xi}^n(t)\}_{j \in \mathbb{Z}, n\in \mathbb{N}} = \{(\zeta_{j+\frac{1}{2}, \xi}^n, U_{j+\frac{1}{2}, \xi}^n, V_{j+\frac{1}{2}, \xi}^n, H_{j+\frac{1}{2}, \xi}^n )(t) \}_{j \in \mathbb{Z}, n\in \mathbb{N}}$ for $t\in (\tau_k^*, \tau_{k+1}^*]$, whose time evolution is governed by 
\begin{subequations}\label{eq:numLagrSystemIter}
\begin{align}
    \zeta_{j+\frac{1}{2}, t\xi}^{n}(t) &= U_{j+\frac{1}{2}, \xi}^{n}(t), 
    \label{eq:numLagr1it}
    \\ 
    U_{j+\frac{1}{2}, t \xi}^{n}(t) &= \frac{1}{2}V_{j+\frac{1}{2}, \xi}^{n}(t), \label{eq:numLagr2it}
    \\ 
    V_{j+\frac{1}{2}, \xi}^{n}(t) &= \left(1 - \beta_{j+\frac{1}{2}}^{n}(t) \chi_{\{s:s \geq \tau_{j+\frac{1}{2}} > \tau_{k}^* \}} (t) \right) V_{j+\frac{1}{2}, \xi}(\tau_k^*), \label{eq:numLagr3it}
    \\
    H_{j+\frac{1}{2}, t\xi}^{n}(t) &= 0, \label{eq:numLagr4it}
\end{align}
\end{subequations}
with $\hat{X}_{j+\frac{1}{2}, \xi}^{n}(\tau_k^*) = \hat{X}_{j+\frac{1}{2}, \xi}(\tau_k^*)$ for all $n \geq 1$, where  \vspace{-0.15cm}
\begin{equation}\label{eq:beta_one}
    \beta_{j+\frac{1}{2}}^1(t) = 0,
\end{equation}
and for $n \geq 2$, \vspace{-0.15cm}
\begin{align}\label{eq:beta}
    \beta_{j+\frac{1}{2}}^{n}(t) &= \begin{cases}
        0, & \tau_{j+\frac{1}{2}} \notin (\tau_k^*, \tau_{k+1}^*] , \\
        \alpha \!\left(y_{j}^{n-1}(\tau_{k+1}^*) \right) \!, & \tau_{j+\frac{1}{2}} \in (\tau_k^*, \tau_{k+1}^*]. 
    \end{cases}
\end{align}
A closer look at the above system reveals that it depends on the sequence $\{y^n_{j}(\tau_{k+1}^*)\}_{j\in \mathbb{Z}}$, which can, keeping \eqref{eq:intLagrSystem} in mind, be computed from $\{ \hat X^n_{j+\frac12,\xi}\}_{j\in \mathbb{Z}}$. This will be discussed in depth after examining how to solve \eqref{eq:numLagrSystemIter}.

For $k\in \{0, \dots, N-1\}$ introduce the set $\mathcal{I}_{k+\frac{1}{2}} \subset \mathbb{Z}$ by
\begin{equation}\label{eq:breakingIndices}
	\mathcal{I}_{k+\frac{1}{2}} := \left\{j \in \mathbb{Z}: \tau_{j+\frac{1}{2}} \in (\tau_k^*, \tau_{k+1}^*] \right\}\!.
\end{equation}
If $j \notin \mathcal{I}_{k+\frac{1}{2}}$ and $t \in [\tau_k^*, \tau_{k+1}^*]$, then the solution to \eqref{eq:numLagrSystemIter} is given by
\begin{align}\label{eq:exact_num}
	\zeta_{j+\frac{1}{2}, \xi}^{n}(t) &= \zeta_{j+\frac{1}{2}, \xi}(\tau_k^*) + U_{j+\frac{1}{2}, \xi}(\tau_k^*)(t-\tau_k^*) + \frac{1}{4}V_{j+\frac{1}{2}, \xi}(\tau_k^*)(t-\tau_k^*)^2\!, \nonumber \\ 
	U_{j+\frac{1}{2}, \xi}^{n}(t) &= U_{j+\frac{1}{2}, \xi}(\tau_k^*) + \frac{1}{2}V_{j+\frac{1}{2}, \xi}(\tau_k^*)(t-\tau_k^*), \nonumber \\
	V_{j + \frac{1}{2}, \xi}^{n}(t) &= V_{j+\frac{1}{2}, \xi}(\tau_k^*). 
\end{align}
This will also be the solution when $j \in \mathcal{I}_{k+\frac{1}{2}}$ prior to wave breaking, i.e., for $t \in [\tau_k^*, \tau_{j+\frac{1}{2}})$. In addition, 
\begin{equation}\label{eq:num_timecont}
	\lim_{t \uparrow \tau_{j+\frac{1}{2}}} \zeta_{j+\frac{1}{2}, \xi}^{n}(t)\! = \zeta^{n}_{j+\frac{1}{2}, \xi}(\tau_{j+\frac{1}{2}})\! = -1 \hspace{0.4cm} \text{and} \hspace{0.4cm} \lim_{t \uparrow \tau_{j+\frac{1}{2}}} U_{j+\frac{1}{2}, \xi}^{n}(t)\! = U^{n}_{j+\frac{1}{2}, \xi}(\tau_{j+\frac{1}{2}})\! = 0,
\end{equation}
since $\zeta_{j+\frac{1}{2}, \xi}^{n}(t)$ and $U_{j+\frac{1}{2}, \xi}^{n}(t)$ are continuous with respect to time. Therefore the solution past wave breaking can also be computed using \eqref{eq:numLagrSystemIter}, which yields 
\begin{align}\label{eq:numExactIteration}
	\zeta_{j+\frac{1}{2}, \xi}^{n}(t) &= \begin{cases}
		\zeta_{j+\frac{1}{2}, \xi}(\tau_k^*) + U_{j+\frac{1}{2}, \xi}(\tau_k^*)(t-\tau_{k}^*) \\ \quad + \frac{1}{4}V_{j+\frac{1}{2}, \xi}(\tau_k^*)(t-\tau_k^*)^2, & \tau_k^* \leq t < \tau_{j+\frac{1}{2}}, \nonumber \\
		-1 + \frac{1}{4}\!\left(1-\alpha(y^{n-1}_{j}(\tau_{k+1}^*))\right)\!V_{j+\frac{1}{2}, \xi}(\tau_k^*)(t-\tau_{j+\frac{1}{2}})^2, & \tau_k^* <\tau_{j+\frac{1}{2}} \leq t,
		\end{cases} \nonumber \\
	U_{j+\frac{1}{2}, \xi}^{n}(t) &= \begin{cases}
		U_{j+\frac{1}{2}, \xi}(\tau_k^*) + \frac{1}{2}V_{j+\frac{1}{2}, \xi}(\tau_k^*)(t-\tau_k^*), & \tau_k^* \leq t < \tau_{j+\frac{1}{2}}, \\
		\frac{1}{2}\!\left(1-\alpha(y^{n-1}_{j}(\tau_{k+1}^*))\right)\!V_{j+\frac{1}{2}, \xi}(\tau_k^*)(t-\tau_{j+\frac{1}{2}}), & \tau_k^* <\tau_{j+\frac{1}{2}} \leq t, 
		\end{cases} \nonumber \\ 
		V_{j+\frac{1}{2}, \xi}^{n}(t) &= \begin{cases}
		V_{j+\frac{1}{2}, \xi}(\tau_k^*), & \tau_k^* \leq t < \tau_{j+\frac{1}{2}}, \\
		\left(1-\alpha(y_{j}^{n-1}(\tau_{k+1}^*))\right)\!V_{j+\frac{1}{2}, \xi}(\tau_k^*), & \tau_k^* < \tau_{j+\frac{1}{2}} \leq t. 
		\end{cases} \raisetag{-42.5pt}
\end{align}
\vspace{0.25cm}

Finally, to proceed from one iterate to another, i.e., from $\{\hat X^{n}_{j+\frac12, \xi}(t)\}_{j\in\mathbb{Z}}$ to $\{\hat X^{n+1}_{j+\frac12,\xi}(t)\}_{j\in \mathbb{Z}}$, one has to compute $\{y^{n}_j(\tau_{k+1}^*)\}_{j\in \mathbb{Z}}$ from the sequence $\{\hat X_{j+\frac12,\xi}^{n}(t)\}_{j\in \mathbb{Z}}$. Since $\{\hat{X}_j^{n}(t)\}_{j \in \mathbb{Z}}\! = \{(\zeta_j^{n}, U_j^{n}, V_j^{n}, H_j^{n})(t)\}_{j \in \mathbb{Z}}$ is an approximation of the exact solution to \eqref{eq:intLagrSystem}--\eqref{eq:waveBreakingFunc}, its asymptotic behavior as $j \!\rightarrow\! \pm \infty$ should change in accordance with the asymptotic behavior of \eqref{eq:intLagrSystem}.

Motivated by \eqref{eq:intLagrSystem} and \eqref{eq:numLagrSystemIter} it immediately follows that
\begin{equation*}
	H_j^{n}(t) = H_j(\tau_k^*) = H_j(0),
\end{equation*}
while the computation of the remaining sequences is a bit more involved. 

Introduce the abbreviation \small $f_{\pm \infty}(t) = \lim_{j \rightarrow \pm \infty}f_j(t)$\normalsize. Then \eqref{eq:int_ODE3} implies 
\begin{equation}\label{eq:leftAsymptoteV}
	V^{n}_{-\infty}(t) = 0 \quad \text{ for all } t \in [\tau_k^*, \tau_{k+1}^*],
\end{equation}
which after being combined with \eqref{eq:exact_num} and \eqref{eq:numExactIteration}, results in the following recursive formula for any $t \in [\tau_k^*, \tau_{k+1}^*]$,  
\begin{align}\label{eq:recursiveV}
	V_j^{n}(t) &= \sum_{m=-\infty}^{j-1}V^{n}_{m+\frac{1}{2}, \xi}(t)\!\left(\xi_{m+1} - \xi_m\right) = V_{j-1}^{n}(t) + V_{j-\frac{1}{2}, \xi}^{n}(t) \left(\xi_{j} - \xi_{j-1} \right)\! \nonumber
	\\ &= V_j(\tau_k^*) -\sum_{m =-\infty}^{j-1} \beta_{m+\frac{1}{2}}^{n}(t)\chi_{\{s:s\geq \tau_{m+\frac{1}{2}}> \tau_k^*\}}(t)V_{m+\frac{1}{2}, \xi}(\tau_k^*)(\xi_{m+1}-\xi_m). 
\end{align}

In order to compute $\{U_j^{n}(t)\}_{j \in \mathbb{Z}}$ for $t \in [\tau_k^*, \tau_{k+1}^*]$, we start by calculating its left asymptote. Combining \eqref{eq:intLagrSystem} and \eqref{eq:numLagrSystemIter} implies 
\begin{align}
	U_{-\infty}^{n}(t) &= U_{-\infty}(\tau_k^*) - \frac{1}{4}V_{\infty}(\tau_k^*)(t-\tau_k^*) \nonumber
	\\ & \quad + \frac{1}{4}\sum_{m=-\infty}^{\infty}\beta_{m+\frac{1}{2}}^{n}(t) \chi_{\{s:s \geq \tau_{m+\frac{1}{2}} > \tau_k^*\}}(t)V_{m+\frac{1}{2}, \xi}(\tau_k^*)(\xi_{m+1} - \xi_m)(t-\tau_{m+\frac{1}{2}}), \raisetag{-25pt}
\end{align}
and hence, we have for all $j \in \mathbb{Z}$ and $t \in [\tau_k^*, \tau_{k+1}^*]$, 
\begin{equation}\label{eq:recursiveU}
	U_{j}^{n}(t) = U_{-\infty}^{n}(t) + \sum_{m=-\infty}^{j-1}U_{m+\frac{1}{2}, \xi}^{n}(t) \left(\xi_{m+1} - \xi_m\right) = U_{j-1}^{n}(t) + U_{j-\frac{1}{2}, \xi}^{n}(t)(\xi_j - \xi_{j-1}). 
\end{equation}
Following the same procedure for $\{\zeta_j^{n}(t)\}_{j \in \mathbb{Z}}$ where $t \in [\tau_k^*, \tau_{k+1}^*]$, i.e., using \eqref{eq:intLagrSystem} and \eqref{eq:numLagrSystemIter} once more, yields 
\begin{align}\label{eq:leftAsymptoteZeta}
	\zeta_{-\infty}^{n}(t) &= \zeta_{-\infty}(\tau_k^*) + U_{-\infty}(\tau_k^*)(t-\tau_k^*) - \frac{1}{8}V_{\infty}(\tau_k^*)(t-\tau_k^*)^2 \nonumber
	\\ & + \frac{1}{8}\sum_{m=-\infty}^{\infty}\beta_{m+\frac{1}{2}}^{n}(t) \chi_{\{s:s \geq \tau_{m+\frac{1}{2}} > \tau_k^*\}}(t)V_{m+\frac{1}{2}, \xi}(\tau_k^*)(\xi_{m+1}-\xi_m)(t-\tau_{m+\frac{1}{2}})^2, \raisetag{-25pt}
\end{align}
and for all $j \in \mathbb{Z}$ and $t \in [\tau_k^*, \tau_{k+1}^*]$, 
\begin{equation}\label{eq:recursiveZeta}
	\zeta_j^{n}(t) = \zeta_{-\infty}^{n}(t) + \sum_{m=-\infty}^{j-1}\zeta_{m+\frac{1}{2}, \xi}^{n}(t)\left(\xi_{m+1}-\xi_m\right) = \zeta_{j-1}^{n}(t) + \zeta_{j-\frac{1}{2}, \xi}^{n}(t)\left(\xi_j - \xi_{j-1}\right)\!.
\end{equation}
At last, after using $y_j^{n}(t) = \zeta_j^{n}(t) + \xi_j$, one recovers $\{y_j^{n}(t) \}_{j \in \mathbb{Z}}$.

Having a closer look at the iteration given by \eqref{eq:numLagrSystemIter}--\eqref{eq:beta} reveals that it suffices to compute $\{y_j^{n}(\tau_{k+1}^*)\}_{j \in \mathbb{Z}}$ in order to proceed from iterate $n$ to iterate $n+1$, meaning that we only need to use \eqref{eq:leftAsymptoteZeta} and \eqref{eq:recursiveZeta} once for every iterate. 
This reduces the computational complexity substantially compared to the case where one would compute $\{y_j^{n}(t)\}_{j \in \mathbb{Z}}$ at each wave breaking time $t=\tau_{m+\frac{1}{2}}$ for $m \in \mathcal{I}_{k+\frac{1}{2}}$. However, this comes at the cost of introducing a local error, which we analyze in Section~\ref{sec:Convergence}  

\begin{remark}\label{rem:evolveExact}
If no wave breaking occurs in $(\tau_k^*, \tau_{k+1}^*)$, the iteration scheme terminates after at most $2$ iterations. In particular, if no wave breaking takes place at $\tau_{k+1}^*$, which is only possible if $\tau_{k+1}^*=T$, the scheme terminates with $n=1$, and  if $\tau_{k+1}^*=\tau_{j+\frac12}$ for some $j\in \mathbb{Z}$, then $n=2$. Furthermore, the iteration scheme is redundant in the latter case as $\{\hat X^n_{j+\frac12,\xi}(t)\}$ for $t \in [\tau_k^*, \tau_{k+1}^*)$ coincides for all $n\in \mathbb{N}$ with the unique solution $\{\hat X_{j+\frac12,\xi}(t)\}$ to 
	\begin{align}\label{eq:numLagrSystem}
	    \zeta_{j+\frac{1}{2}, t\xi}(t) = U_{j+\frac{1}{2}, \xi}(t), \hspace{0.6cm} U_{j+\frac{1}{2}, t \xi}(t) = \frac{1}{2}V_{j+\frac{1}{2}, \xi}(t), \hspace{0.6cm}  V_{j+\frac{1}{2}, \xi}(t) = V_{j+\frac{1}{2}, \xi}(\tau_{k}^*), 
 \end{align}
with initial data $\{\hat{X}_{j+\frac{1}{2}, \xi}(\tau_k^*)\}_{j \in \mathbb{Z}}$. Hence, we have $\hat X_{j}(t)=\hat X_j^n(t)$ for all $t\in [\tau_k^*, \tau_{k+1}^*)$. Since  $\{(y_j^n, U_j^n)\}_{j \in \mathbb{Z}}$ is continuous with respect to time,  one can thereafter compute 
	\begin{align}\label{eq:numBreak}
	V_{j+\frac{1}{2}, \xi}(\tau_{k+1}^*) &= \begin{cases}
	V_{j+\frac{1}{2}, \xi}(\tau_k^*), & j \notin \mathcal{I}_{k+\frac{1}{2}}, \\
	(1-\alpha(y_j(\tau_{k+1}^*)))V_{j+\frac{1}{2}, \xi}(\tau_k^*), & j \in \mathcal{I}_{k+\frac{1}{2}}.
	\end{cases} 
\end{align}	
	Numerically, we therefore evolve according to \eqref{eq:numLagrSystem} and \eqref{eq:numBreak} whenever no wave breaking occurs within $(\tau_k^*, \tau_{k+1}^*)$, because this is slightly more efficient as we only have to compute $\{y_j(\tau_{k+1}^*)\}$ once. Moreover, no local error is introduced. 
\end{remark}   

The last issue we have to address is the stopping of the iteration scheme. To this end, let $\ell^{\infty}$ denote the usual sequence space equipped with the norm
\begin{equation*}
	\|\{a_j\}_{j \in \mathbb{Z}}\|_{\ell^{\infty}} := \sup_{j \in \mathbb{Z}}|a_j|. 
\end{equation*}

Closely inspecting \eqref{eq:numExactIteration} reveals that the difference between two successive elements in the iteration scheme depends heavily on $\| \{y_j^{n+1}(\tau_{k+1}^*)\}-\{y_j^n(\tau_{k+1}^*)\}\|_{\ell^{\infty}}$, which serves as a basis for determining a limiting criteria for the number of iterations needed.

\begin{prop}\label{prop:termination}
	Assume $\tau_{k}^* = \hat{\tau}_m$ and $\tau_{k+1}^*=\hat{\tau}_{m+l}$ for $m\! \! \in \mathbb{N}_0$ and $l \geq 2$. Then for every $\varepsilon > 0$ there is some $M_{\mathrm{it}}^k \geq 2$ such that
\begin{equation}\label{eq:termCond}
	\sup_{t \in [\tau_k^*, \tau_{k+1}^*]} \|\{y_j^{M_{\mathrm{it}}^k}(t)\} - \{y_j^{M_{\mathrm{it}}^k-1}(t)\}\|_{\ell^{\infty}} \leq \varepsilon. 
\end{equation}
\end{prop}

\begin{proof}
Recall that $\{\hat{X}_{j}^{n}(\tau_k^*)\} = \{\hat{X}_{j}(\tau_k^*)\}$ for all $n \geq 1$. Therefore,  after combining \eqref{eq:numExactIteration} with \eqref{eq:recursiveV}--\eqref{eq:recursiveZeta}, it holds, for $t \in [\tau_k^*, \tau_{k+1}^*]$ and $j\in\mathbb{Z}$ that 
\begin{subequations}\label{eq:ftc}
\begin{align}
	\left|y_{j}^{n+1}(t) - y_{j}^n(t) \right|\!  &= \left| \zeta_{j}^{n+1}(t) - \zeta_{j}^n(t) \right| 
	\leq \int_{\tau_k*}^{t} \left|U_{j}^{n+1}(s) - U_{j}^n(s) \right|ds, \label{eq:ftc_y} \\
	\left| U_{j}^{n+1}(t) - U_{j}^n(t) \right| &\leq \frac{1}{4}\sum_{m\in \mathbb{Z}}\int_{\tau_k^*}^t \left|V_{m+\frac12, \xi}^{n+1}(s) - V_{m+\frac12, \xi}^{n}(s) \right|ds(\xi_{m+1}-\xi_m)\label{eq:ftc_U}.
\end{align}
\end{subequations}
Moreover, \eqref{eq:exact_num} and \eqref{eq:numExactIteration} yield, for $n\geq 2$,
\begin{align}\label{eq:iter_Vxis}
	V_{j+\frac12, \xi}^{n+1}(t) - V_{j+\frac12, \xi}^{n}(t) &= \begin{cases}
	\!0, & j \notin \mathcal{I}_{k+\frac{1}{2}}, \\
	\!\!\left(\alpha(y_{j}^{n-1}(\tau_{k+1}^*)) - \alpha (y_j^{n}(\tau_{k+1}^*) ) \right)
	\\ \quad \times V_{j+\frac12, \xi}(\tau_k^*) \chi_{\{q: q \geq \tau_{j+\frac{1}{2}} > \tau_{k}^*\}}(t), & j \in \mathcal{I}_{k+\frac{1}{2}}.  \raisetag{-50pt}
	\end{cases}
\end{align}
By assumption there are $(l-1) \geq 1$ distinct breaking times in $(\tau_k^*, \tau_{k+1}^*)$ and by construction $\tau_{k+1}^* - \tau_k^* \leq \Dt$. Thus, we obtain for $n \geq 2$,
\begin{align}\label{eq:general_recursive}
	\left| y_{j}^{n+1}(t) - y_{j}^n(t) \right| &\leq  \frac{1}{4}\|\alpha ' \|_{\infty}  \sum_{m \in \mathbb{Z}}\int_{\tau_k^*}^t \int_{\tau_k^*}^s \left| y_{m}^{n}(\tau_{k+1}^*) - y_{m}^{n-1}(\tau_{k+1}^*) \right| \nonumber
	\\ & \qquad \qquad \qquad \quad  \times  \chi_{\{q: q \geq \tau_{m+\frac{1}{2}} >\tau_k^*\}}(r)V_{m+\frac12, \xi}(\tau_k^*)   dr ds (\xi_{m+1}-\xi_m)\nonumber
	\\ & \leq\frac{1}{8} \|\alpha ' \|_{\infty} V_{ \infty}(\tau_k^*) \Dt^2 \sup_{s \in [\tau_k^*, \tau_{k+1}^*]} \|\{y_{j}^{n}(s)\} - \{y_{j}^{n-1}(s)\} \|_{\ell^\infty} \nonumber \\ 
	&\leq \gamma_{\Dt} \!\sup_{s \in [\tau_k^*, \tau_{k+1}^*]} \|\{y_{j}^{n}(s)\} - \{y_{j}^{n-1}(s)\} \|_{\ell^\infty},
\end{align}
where we introduced
\begin{equation}\label{eq:gamma_dt}
	\gamma_{\Dt} := \frac{1}{8} \|\alpha ' \|_{\infty}G_{\infty}\Dt^2,
\end{equation}
and used that the total energy $V_\infty(t)$ is nonincreasing and bounded from above by $V_\infty(0)=F_\infty(0)=G_\infty$. Furthermore, due to \eqref{eq:beta_one}, we have 
\begin{align}\label{eq:iter_initial}
	\bigl| y_{j}^{2}(t) &- y_{j}^{1}(t) \bigr| \nonumber
	\\ &\leq \frac{1}{4} \sum_{m \in \mathbb{Z}} \int_{\tau_k^*}^t \int_{\tau_k^*}^s \alpha(y_{m}^1(\tau_{k+1}^*))\chi_{\{q: q \geq \tau_{m+\frac{1}{2}} >\tau_k^*\}}(r)V_{m+\frac12, \xi}(\tau_k^*) dr ds (\xi_{m+1}-\xi_m)\nonumber
	\\ & \leq \frac{1}{8} G_{\infty} \Dt^2\!. \raisetag{-20pt}
\end{align}
Thus, it follows by \eqref{eq:general_recursive} and \eqref{eq:iter_initial}, that for $n\geq1$ 
\begin{equation}\label{eq:iter_y1}
	\sup_{t \in [\tau_k^*, \tau_{k+1}^*]} \| \{y_{j}^{n+1}(t)\} - \{y_{j}^{n}(t)\} \|_{\ell^\infty} \leq \frac{1}{8} G_{\infty}\Dt^2 \gamma_{\Dt}^{n-1}\!.
\end{equation}
Provided that we choose $\Dt$ small enough, we can ensure that $\gamma_{\Dt} < 1$. As a consequence, the right side of \eqref{eq:iter_y1} can be made arbitrarily small by choosing $n$ large enough. Thus, to every $\varepsilon > 0$, there exists an integer $n=M_{\mathrm{it}}^k-1$ such that \eqref{eq:iter_y1} is less than or equal to $\varepsilon$. 
\end{proof}

Finally, we can put an upper bound on the maximum number of iterations needed. Enforcing from now on that $\Dt$ satisfies the upper bound 
\begin{equation}\label{eq:dt}
    \Delta t\leq   \sqrt{  \frac{8 \Delta x}{\|\alpha ' \|_{\infty} G_{\infty}} },
\end{equation}
yields $\gamma_{\Dt}\! \leq\! \Dx$ and \eqref{eq:iter_y1} turns into  
\begin{equation}\label{eq:iter_y2}
	\sup_{t \in [\tau_k^*, \tau_{k+1}^*]} \| \{y_{j}^{n+1}(t)\} - \{y_{j}^{n}(t)\} \|_{\ell^\infty} \leq\frac18 G_{\infty} \Dt^2 \Dx^{n-1} \!, 
\end{equation}
for all $n\geq 1$. Choosing $n=2$ corresponds to carrying out at most $3$ iterations, cf. Proposition~\ref{prop:termination} and introduces an error of order $\mathcal{O}(\Dx^2)$. Therefore, we define, for $t\in (\tau_k^*, \tau_{k+1}^*]$, the sequence $\{X_j(t)\}_{j \in \mathbb{Z}}$ as 
\begin{equation}\label{iterseq}
\{X_j(t)\}_{j \in \mathbb{Z}} =\{X_j^{\mathrm{M}_{\mathrm{it}}^{k}}(t)\}_{j\in \mathbb{Z}}= \{(y_j^{\mathrm{M}_{\mathrm{it}}^{k}}, U_j^{\mathrm{M}_{\mathrm{it}}^{k}}, V_j^{\mathrm{M}_{\mathrm{it}}^{k}}, H_j^{\mathrm{M}_{\mathrm{it}}^{k}})(t)\}_{j \in \mathbb{Z}},
\end{equation}
which satisfies \eqref{eq:termCond} with $\varepsilon=\frac{1}{8} G_{\infty}\Dt^2 \Dx$, and 
where, taking into account Remark~\ref{rem:evolveExact} and Proposition~\ref{prop:termination},
\begin{equation*}
\mathrm{M}_{\mathrm{it}}^{k}=\begin{cases} 1, \quad \text{ if } [\tau_k^*,T] \subset [\hat\tau_m, \hat\tau_{m+1}),\\
 2, \quad  \text{ if } [\tau_k^*, \tau_{k+1}^*]=[\hat \tau_m, \hat \tau_{m+1}],\\
 2 \text{ or } 3, \quad \text{ otherwise}.
 \end{cases}
 \end{equation*}

Finally, we can recover the numerical Lagrangian solution $$X_{\Dx}(t,\xi) = (y_{\Dx}, U_{\Dx}, V_{\Dx}, H_{\Dx})(t,\xi) \quad \text{ for }(t,x)\in [0,T]\times\R$$ by linear interpolation based on $\{X_j(t) \}_{j \in \mathbb{Z}}=\{X_{\Dx}(t, \xi_j)\}_{j\in \mathbb{Z}}$. In symbolic notation, we represent the whole procedure via the numerical solution operator $S_{\Dx, t}$ defined by 
\begin{equation} \label{eq:numSolOP}
    S_{\Dx, t}(X_{\Dx}(0)) := X_{\Dx}(t). 
  \end{equation}

\begin{remark}\label{rem:Invariance}
	The iteration scheme \eqref{eq:numLagrSystemIter}--\eqref{eq:beta} is chosen in such a way that $S_{\Dx, t}$ preserves Definition~\ref{def:LagSet}~\ref{def:importantRelation} and that the breaking times given by $\tau_{\Dx}$ are respected. As a consequence, one can follow the argument in \cite[Lem. 2.3]{AlphaHS} to show that for any $X \in \F^{\alpha}$, which is piecewise linear with nodes located at $\{\xi_j\}_{j \in \mathbb{Z}}$, one has 
	\begin{equation*}
		S_{\Dx, t}(X) \in \F^{\alpha} \quad \text{ for all } t \in [0,T]. 
	\end{equation*}
\end{remark}

Before examining the implementation of the mapping $M$, we consider an example illustrating the possible computational benefit of evolving between successive times in $\{\tau_k^*\}_{k=0}^N$, rather than evolving between successive times in $\{\hat{\tau}_j\}_{j \in \mathcal{J}}$.   
  
 \begin{example}[Cusp data]\label{ex:cuspData}
 Consider the following initial data 
 \begin{align*}
	u(x) &= \begin{cases}
	1, &x < -1, \\
	\left|x\right|^{\frac{2}{3}}\!, & -1\leq x \leq1, \\
	1, & 1 < x,
	\end{cases} \\ 
	F(x) &= \begin{cases}
	0, & x < -1, \\
	\frac{4}{3} \left(1 + \mathrm{sgn}(x)|x|^{\frac{1}{3}}\right)\!, & -1 \leq x \leq 1, \\
	\frac{8}{3}, & 1 < x, \end{cases} \\ 
	\alpha(x) &= \begin{cases}
	\beta, & x < -1, \\
	\beta|x|, & -1 \leq x <0, \\
	0, & 0 \leq x, \end{cases}
\end{align*}
where $\beta \in [0, 1)$. Computing $u_{x}(x)$ for $x \in (-1, 0)$ reveals that the exact solution experiences wave breaking continuously over the time interval $[0, 3]$. 

Let $x_j = -1 + j\Dx$ denote a uniform discretization of $[-1, 1]$ with $\Dx =10^{-k}$ for $k \in \mathbb{N}$. Then there are $10^k$ grid cells covering $[-1, 0]$. Furthermore, $u$ is strictly decreasing on $[-1, 0)$, implying that $Du_{2j} < 0$ for all $j \in \{0, 1, \hdots,  5\cdot 10^{k-1}-1\}$.  Consequently, it follows, after differentiating \eqref{eq:Proj_u} and using the nonnegativity of $q_{2j}$, that at least $5\cdot 10^{k-1} $ linear segments on $[-1, 0]$ have negative slope. However, if we can show that
\begin{equation*}
	2 (Du_{2j})^2 > DF_{\mathrm{ac}, 2j}
\end{equation*}
holds, this will imply that $Du_{2j} + q_{2j} < 0,$ and as a consequence, all the $10^k$ linear segments of $u_{\Dx}$ over $[-1, 0]$ have negative slope. Indeed, we have  
\begin{align*}
	Du_{2j} &= \frac{1}{2\Dx} \left(|x_{2j+2}|^{\frac{2}{3}} - |x_{2j}|^{\frac{2}{3}} \right)
	\\ &= \frac{1}{2\Dx} \left(|x_{2j+2}|^{\frac{1}{3}} - |x_{2j}|^{\frac{1}{3}} \right) \left(|x_{2j+2}|^{\frac{1}{3}} + |x_{2j}|^{\frac{1}{3}} \right)\!, \\
	DF_{\mathrm{ac}, 2j} &= \frac{2}{3\Dx} \left(|x_{2j}|^{\frac{1}{3}} - |x_{2j+2}|^{\frac{1}{3}}\right)\!,
\end{align*}
and moreover, $2\Dx = x_{2j+2} - x_{2j} = (|x_{2j}|^{\frac{1}{3}})^3 -  (|x_{2j+2}|^{\frac{1}{3}})^3$, which combined with $a^3 - b^3 = (a-b)(a^2 +ab +b^2)$ yields
\begin{align*}
	2(Du_{2j})^2 &= \frac{1}{\Dx} \frac{ \left(|x_{2j+2}|^{\frac{1}{3}} - |x_{2j}|^{\frac{1}{3}} \right)^2 \left(|x_{2j+2}|^{\frac{1}{3}} + |x_{2j}|^{\frac{1}{3}} \right)^2}{\left(|x_{2j}|^{\frac{1}{3}} - |x_{2j+2}|^{\frac{1}{3}}\right) \left(|x_{2j}|^{\frac{2}{3}} +|x_{2j}|^{\frac{1}{3}} |x_{2j+2}|^{\frac{1}{3}} + |x_{2j+2}|^{\frac{2}{3}}\right)}
	\\ & \geq \frac{1}{\Dx} \frac{\left(|x_{2j}|^{\frac{1}{3}} - |x_{2j+2}|^{\frac{1}{3}} \right)\left(|x_{2j+2}|^{\frac{1}{3}} + |x_{2j}|^{\frac{1}{3}} \right)^2}{\left( |x_{2j}|^{\frac{2}{3}} + 2|x_{2j}|^{\frac{1}{3}}|x_{2j+2}|^{\frac{1}{3}} +  |x_{2j+2}|^{\frac{2}{3}}\right)}
	\\ &= \frac{1}{\Dx} (|x_{2j}|^{\frac{1}{3}} - |x_{2j+2}|^{\frac{1}{3}}) = \frac{3}{2}DF_{\mathrm{ac}, 2j} > DF_{\mathrm{ac}, 2j}. 
\end{align*}

Assuming evenly distributed breaking times over $[0, 3]$, yields the following approximate internal distance between two successive breaking times
\begin{equation*}
	\hat{\tau}_{j+1} - \hat{\tau}_j \approx \frac{3}{10^k}\!.
\end{equation*}
Suppose we have equality in \eqref{eq:dt}, then  
\begin{equation*}
	\Dt = \sqrt{\frac{3}{\beta}}10^{-\frac{k}{2}}, 
\end{equation*}
which implies that each subinterval of $[0, 3]$ of length $\Dt$ contains on average
\begin{equation*}
	M \approx \frac{10^k \Dt }{3}= \frac{10^{\frac{k}{2}}}{\sqrt{3\beta}}, 
\end{equation*}
distinct breaking times. Consequently, the length of the sequence $\{\tau_k^*\}_{k=0}^N$ is roughly
\begin{equation*}
	N \approx \frac{3}{\Dt} = \sqrt{3\beta}10^{\frac{k}{2}}.
\end{equation*}
When $\Dx=10^{-4}$ and $\beta=\frac{19}{20}$, $\{\hat{\tau}_j \}_{j \in \mathcal{J}}$ has length equal to approximately $10^4$, while $N \approx 165$. Furthermore, $\Dt \approx 1.78 \cdot 10^{-2}$ and each interval $[\tau_k^*, \tau_{k+1}^*]$ contains on average around $59$ distinct breaking times. Numerical simulations conducted in Section~\ref{sec:NumericalExperiments} support these observations.
\end{example}

\subsubsection{Implementation of $M$}
Finally, to recover the numerical solution in Eulerian coordinates we apply the mapping $M$ from Definition~\ref{def:MapM} to $X_{\Dx}(t)$, i.e., 
\begin{equation*}
    \left (u_{\Delta x}, \mu_{\Delta x}, \nu_{\Delta x}\right)(t) = M\!\left( X_{\Dx}(t) \right)\!.
\end{equation*}
As the components of $X_{\Dx}(t)$ are continuous and piecewise linear, $u_{\Dx}(t)$ will be continuous and piecewise linear, while $F_{\Dx}(t)$ and $G_{\Dx}(t)$ are increasing, piecewise linear and left-continuous, with $F_{\Dx}(t) \neq G_{\Dx}(t)$ for $t > 0$ in general. Furthermore, the nodes of $(u_{\Dx}, F_{\Dx}, G_{\Dx})(t)$ are located at the points $\{y_{\Dx}(t, \xi_j) \}_{j \in \mathbb{Z}}=\{y_j(t)\}_{j \in \mathbb{Z}}$. Thus, in principle, applying $M$ amounts to applying a piecewise linear reconstruction based on the Eulerian gridpoints $\{y_{j}(t) \}_{j \in \mathbb{Z}}$. For details about this reconstruction, the interested reader is referred to \cite[Sec. 3.2.4]{AlphaAlgorithm}. 

To summarize, we define the numerical solution as follows.
\begin{definition}\label{def:numSol}
    For any $(u_0, \mu_0,\nu_0)\in \mathcal{D}^\alpha_0$, the numerical $\alpha$-dissipative solution at time $t\in [0,T]$ is given by 
    \begin{align*}
        (u_{\Delta x}, \mu_{\Delta x}, \nu_{\Delta x})(t) &:= T_{\Delta x, t} \circ P_{\Dx} \left((u_0, \mu_0, \nu_0)\right) \\ & = M \circ S_{\Delta x, t} \circ L \circ P_{\Dx} \left( (u_0, \mu_0, \nu_0)\right)\!.
    \end{align*}
\end{definition}

\section{Convergence of the numerical method}\label{sec:ConvergenceMethod}

In order to prove convergence of the presented numerical method, we will follow the same steps as in the last section. That is, we will start by showing that convergence of the initial data in Eulerian coordinates leads to convergence initially in Lagrangian coordinates. Thereafter, it is shown that the numerical solution converges in Lagrangian coordinates for all $t \in [0, T]$, i.e., $X_{\Dx}(t)\to X(t)$  as $\Dx\to 0$, before we establish that this convergence carries over to Eulerian coordinates.

\subsection{Convergence of the numerical initial data}\label{sec:PrelimConv}
We start this section by recalling the convergence results for the projection operator $P_{\Dx}$ from \cite{AlphaAlgorithm}. Then we proceed by showing that this implies convergence of the initial data in Lagrangian coordinates.

For $P_{\Dx}$ the following result has been proven in \cite[Prop. 4.1 and Lem. 4.2]{AlphaAlgorithm}.
 
\begin{prop}\label{prop:ProjectionEul}
    For any $(u,\mu, \nu) \in \D_0^{\alpha}$, let $(u_{\Delta x}, F_{\Delta x}, G_{\Delta x})=P_{\Delta x} \left((u,F, G) \right)$, then
 \begin{subequations}
    \begin{align}
        \|u - u_{\Delta x}\|_{\infty} & \leq \left(1 + \sqrt{2} \right)\!\sqrt{F_{\mathrm{ac}, \infty}} \Delta x^{\frac{1}{2}}, \label{eq:proj_uinfty}\\
        \|u - u_{\Delta x}\|_{2} & \leq \sqrt{2}\left(1+\sqrt{2}\right)\!\sqrt{F_{\mathrm{ac}, \infty}} \Delta x, \label{eq:proj_uL2}  \\ 
        \| F - F_{\Delta x}\|_p &\leq 2 F_{\infty}\Delta x^{\frac{1}{p}},  \quad \text { for } p=1,2, \label{eq:proj_FLp} \\
        \|G - G_{\Dx} \|_p & \leq 2 G_{\infty}\Dx^{\frac{1}{p}}, \quad \text { for } p=1,2. \label{eq:proj_GLp} 
       \end{align}
        \end{subequations}
       Moreover, it holds, as $\Dx \rightarrow 0$, that 
       \begin{equation}
          u_{\Dx, x}  \rightarrow u_x \text{ in } L^2(\R). \label{eq:proj_uxL2}
           \end{equation}
\end{prop}
The convergence in Proposition~\ref{prop:ProjectionEul} is transported to Lagrangian coordinates via the mapping $L$, from Definition~\ref{def:MapL}, in the following way. 

\begin{lemma}\label{lem:intLagr}
Given $(u, \mu, \nu) \in \D_0^{\alpha}$, let $X = \left(y, U, V, H \right) = L \left( (u, \mu, \nu) \right)$ and $X_{\Delta x}= (y_{\Delta x}, U_{\Delta x}, V_{\Delta x}, H_{\Delta x}) = L \circ P_{\Delta x} \left( (u, \mu, \nu) \right)$, then we have 
    \begin{subequations}
    \begin{align}
        \|y - y_{\Delta x} \|_{\infty} &\leq 2\Delta x, \label{eq:initial_approx_y}\\
        \|U - U_{\Delta x} \|_{\infty} & \leq \left(1 + 2\sqrt{2} \right)\!\sqrt{F_{\mathrm{ac}, \infty}}\Delta x^{\frac{1}{2}},\label{eq:initial_approx_U} \\
        \|H - H_{\Delta x} \|_{\infty} & \leq 2 \Delta x. \label{eq:initial_approx_H}
     \end{align}
     \end{subequations}
     Furthermore, as $\Dx \rightarrow 0$, 
     \begin{subequations}
     \begin{align}
        y_{\Delta x, \xi} & \rightarrow y_{\xi} \text{ in } L^1(\R) \cap L^2(\R),  \label{eq:convergence_y_xi}\\
             U_{\Delta x, \xi} &\rightarrow U_{\xi} \text{ in } L^2(\R), \label{eq:convergence_U_xi} \\ 
            H_{\Delta x, \xi} &\rightarrow H_{\xi} \text{ in } L^1(\R) \cap L^2(\R). \label{eq:convergence_H_xi}
    \end{align}
    \end{subequations}
\end{lemma} 
For a proof the interested reader is referred to \cite[Lem. 4.4--4.5]{AlphaAlgorithm} and \cite[Sec. 5]{EquivalenceEulerLagrange}. 

In order to obtain convergence with respect to the metric $d$ from Definition~\ref{def:metric}, it remains to show the following result.

\begin{prop}\label{prop:conv_gUHxi}
Given $(u, \mu, \nu) \in \D_0^{\alpha}$, let $X= (y, U, V, H) = L \left((u, \mu, \nu) \right)$ and $X_{\Delta x} = \left(y_{\Delta x}, U_{\Delta x}, V_{\Delta x}, H_{\Delta x} \right) = L \circ P_{\Delta x} \left((u, \mu, \nu ) \right)$, then, as $\Dx \rightarrow 0$,  
\begin{subequations}
\begin{align}
    g_1(X_{\Delta x}) &\rightarrow g_1(X) \text{ in } L^1(\R) \cap L^2(\R) \label{eq:conv_g}, 
    \\
    g_2(X_{\Delta x}) & \rightarrow g_2(X) \text{ in } L^2(\R) \label{eq:conv_g2},
    \\ 
    g_3(X_{\Delta x}) & \rightarrow g_3(X) \text{ in } L^2(\R) \label{eq:conv_g3},  \\ 
    U_{\Dx}H_{\Dx, \xi} & \rightarrow U H_{\xi} \text{ in } L^2(\R) \label{eq:conv_UHxi}. 
\end{align}
\end{subequations}
\end{prop}

\begin{proof}
The proof of \eqref{eq:conv_g}, which we do not present here, is a slight modification of the one of \cite[Prop 4.7]{AlphaAlgorithm}, and relies on the following two inequalities, which hold for any $\alpha\in W^{1,\infty}(\mathbb{R},[0,1))$,
\begin{equation*}
	\bigl | \alpha(x)\bigl(2- \alpha(x) \bigr) - \alpha(z)\bigl(2 -\alpha(z)\bigr) \! \! \bigr | \leq
	\left| \int_{\alpha(z)}^{\alpha(x)}2(1-s)ds \right| \leq 2 \|\alpha ' \|_{\infty} |x-z|,
\end{equation*}
and
\begin{equation*}
	\alpha(x)(2-\alpha(x)) \leq 1. 
\end{equation*}

Next, recall $g_3(X)$ given by \eqref{eq:g_3}, and introduce 
\begin{equation}\label{eq:notation}
	\Omega_{k, m} = \Omega_k(X) \cap \Omega_m(X_{\Delta x}), \quad \text{for } k, m \in \{c, d \}.
\end{equation}
Then, it follows that
\begin{subequations}
\begin{align}
	\|g_3(X) - g_3(X_{\Dx}) \|_2^2 &= \|\alpha'\|_{\infty}^2\int_{\Omega_{c, d}}\! U_{\Dx}^2 U_{\Dx, \xi}^2(\xi)d\xi \label{eq:est_g31}
	\\ & \quad +  \|\alpha'\|_{\infty}^2\int_{\Omega_{d, c}} \!U^2 U_{\xi}^2(\xi)d\xi \label{eq:est_g32}
	\\ & \quad +  \|\alpha'\|_{\infty}^2\int_{\Omega_{d, d}} \!(UU_{\xi} - U_{\Dx}U_{\Dx, \xi} )^2 (\xi)d\xi \label{eq:est_g33}. 
\end{align}
\end{subequations} 
The two terms \eqref{eq:est_g31} and \eqref{eq:est_g32} have a similar structure and thus we only derive an estimate for \eqref{eq:est_g31}.  Combining \eqref{eq:L_eq2} and \eqref{eq:proj_uinfty} yields
\begin{equation}\label{eq:estUDx}
 \|U_{\Dx} \|_{\infty}\! \leq \|u_{\Dx} \|_{\infty}\leq \|u\|_\infty + \|u_{\Dx}-u\|_\infty\leq \|u\|_\infty +\left(1 + \sqrt{2} \right)\!\sqrt{F_{\mathrm{ac}, \infty}} \sqrt{\Delta x}.
 \end{equation}
Furthermore, $U_{\xi}(\xi) \geq 0$, while $U_{\Dx, \xi}(\xi) <  0$ for all $\xi\in \Omega_{c, d}$ and hence 
\begin{align}\nonumber
	\int_{\Omega_{c, d}}\!\!\! \!U_{\Dx}^2 U_{\Dx, \xi}^2(\xi)d\xi  & \leq \|U_{\Dx}\|_{\infty}^2 \int_{\Omega_{c, d}}\! \!\!(U_{\Dx, \xi} - U_{\xi} )^2 (\xi)d\xi \\ 
	& \label{eq:est_g3First} \leq\! \|U_{\Dx}\|_{\infty}^2 \|U_{\Dx, \xi} - U_{\xi}\|_2^2.
\end{align} 

For the term \eqref{eq:est_g33}, we use $(a+b)^2 \leq 2(a^2+b^2)$ for any $a, b \in \R$. In particular, choosing $a=U(U_{\xi} - U_{\Dx, \xi})$ and $b=U_{\Dx, \xi}(U - U_{\Dx})$ we have
\begin{align}\label{eq:est_g3Final}
	\int_{\Omega_{d, d}} \!(UU_{\xi} - U_{\Dx}U_{\Dx, \xi})^2(\xi)d\xi  &\leq 2 \int_{\Omega_{d, d}} U^2 (U_{\xi} - U_{\Dx, \xi})^2(\xi) d\xi \nonumber
	\\ & \quad + 2\int_{\Omega_{d, d}} U_{\Dx, \xi}^2 (U - U_{\Dx})^2 (\xi)d\xi \nonumber
	\\ & \leq 2 \|U\|_{\infty}^2 \|U_{\xi} - U_{\Dx, \xi} \|_2^2 + 2F_{\infty} \|U - U_{\Dx}\|_{\infty}^2, 
\end{align}
where we in the last step used that Definition~\ref{def:LagSet}~\ref{def:importantRelation} and $0\leq y_{\Dx, \xi} \leq 1$, which holds by Definition~\ref{def:MapL}, imply 
\begin{equation*}
	\int_{\Omega_{d, d}} \l U_{\Dx, \xi}^2(\xi)d\xi \leq \int_{\R}y_{\Dx, \xi}V_{\Dx, \xi}(\xi)d\xi \leq F_{\infty}.
\end{equation*}
Combining \eqref{eq:est_g3First}--\eqref{eq:est_g3Final} with Lemma~\ref{lem:intLagr} finishes the proof of \eqref{eq:conv_g3}. Next, observe that $H_\infty= H_{\Dx, \infty}$, and \eqref{eq:conv_g2} is therefore an immediate consequence of the above estimates. 

In order to prove \eqref{eq:conv_UHxi}, we combine Lemma~\ref{lem:intLagr} with $0\leq H_{\Delta x, \xi}\leq 1$, which holds by Definition~\ref{def:MapL}, and 
\begin{align*}
	\int_{\R}\left(UH_{\xi} - U_{\Dx}H_{\Dx, \xi} \right)^2(\xi)d\xi & \leq 2\int_{\R}U^2(H_{\xi} - H_{\Dx, \xi})^2(\xi)d\xi \\
	& \quad + 2 \int_{\R}H_{\Dx, \xi}^2 (U- U_{\Dx})^2(\xi)d\xi. \qedhere
\end{align*}
\end{proof}

After combining Definition~\ref{def:metric}, Lemma~\ref{lem:intLagr}, and Lemma~\ref{prop:conv_gUHxi} we end up with the following result, which establishes the convergence of the numerical initial data in Lagrangian coordinates. 

\begin{corollary}\label{cor:convini}
Given $(u, \mu, \nu) \in \D_0^{\alpha}$, let $X= (y, U, V, H) = L \left((u, \mu, \nu) \right)$ and $X_{\Delta x} = \left(y_{\Delta x}, U_{\Delta x}, V_{\Delta x}, H_{\Delta x} \right) = L \circ P_{\Delta x} \left((u, \mu, \nu ) \right)$, then
\begin{equation*}
d(X,X_{\Dx})\to 0 \quad \text{ as }\quad \Dx\to 0.
\end{equation*}
\end{corollary}

\subsection{Convergence in Lagrangian coordinates}\label{sec:Convergence}
The aim of this section is to prove that $X_{\Dx}(t)\to X(t)$ as $\Dx\to 0$ for all $t\in [0,T]$. Unlike the situation in \cite{AlphaAlgorithm} where $\alpha$ is a constant, and as highlighted in Section~\ref{sec:numericalImp}, we have to approximate the operator $S_t$ by the numerical solution operator $S_{\Dx,t}$. In particular, $S_{\Dx, t}$ involves an iteration scheme, which stops after a finite number of iterations, and is based on approximating the amount of energy to be removed at wave breaking. A consequence of these two {\it inaccuracies} is that Theorem~\ref{thm:Lipschitz} is not sufficient for establishing convergence.

To compare $X(t)=S_t(X)$ and $X_{\Dx}(t)=S_{\Dx, t}(X_{\Dx})$, where $S_t$ and $S_{\Dx, t}$ are defined in Definition~\ref{def:SolOPLagr} and \eqref{eq:numSolOP}, respectively, introduce
\begin{equation*}
	\widehat{X}_{\Dx}(t) := S_t(X_{\Dx}).
\end{equation*}
By the triangle inequality and Theorem~\ref{thm:Lipschitz}, we then obtain
\begin{align}\label{eq:triangle}
	d(X(t), X_{\Dx}(t)) &\leq d\!\left(X(t), \widehat{X}_{\Dx}(t) \right) + d\!\left(\widehat{X}_{\Dx}(t), X_{\Dx}(t) \right)\! \nonumber
	\\ & \leq  C(t)e^{D(t)t}d(X, X_{\Dx}) + d(\widehat{X}_{\Dx}(t), X_{\Dx}(t)).
\end{align}
Corollary~\ref{cor:convini} implies that the first term on the right-hand side tends to $0$ as $\Dx \to 0$.  For the second term, on the other hand, no analogue to Theorem~\ref{thm:Lipschitz} is available. However, we can take advantage of the fact that $\widehat{X}_{\Dx}(t)$ and $X_{\Dx}(t)$ share the same initial data, which by \eqref{eq:waveBreakingFunc} and Remark~\ref{rem:Invariance} imply $\hat{\tau}_{\Dx} = \tau_{\Dx}$. This allows us to replace $d$ by a much simpler function $d_s$, which is inspired by \cite[Sec. 3]{NewestLipschitzMetric} and which allows us to argue inductively with respect to the sequence $\{\tau_k^*\}_{k=0}^N$ as we will see.

For any $(X,\hat X)\! \in \! \F^{\alpha} \times \F^{\alpha}$ with $\tau=\hat \tau$, which in turn implies $\Omega_j(X)=\Omega_j(\hat X)$ for $j\in\{c, d\}$, define 
\begin{equation}\label{eq:gg}
	g(X,\hat X)(\xi):= \begin{cases} \vert V_\xi-\hat V_\xi\vert(\xi) , & \xi \in \Omega_c(X),\\
	\vert V_\xi-\hat V_\xi\vert(\xi) + \|\alpha'\|_\infty \min(V_\xi(\xi), \hat V_\xi(\xi)) &\\
	 \qquad \qquad \times \!\left(\vert y-\hat y\vert (\xi)+ \vert U-\hat U\vert(\xi) \right)\!, & \xi \in \Omega_d(X),
\end{cases}
\end{equation}
and let 
\begin{align}\label{eq:ds}
	d_s(X, \hat X)&:=\|y-\hat y\|_\infty+ \|U-\hat U\|_\infty\\ \nonumber
	& \quad + \|y_\xi-\hat y_\xi\|_2+ \|U_\xi-\hat U_\xi\|_2
 	+ \|g(X, \hat X)\|_2.
\end{align}

To begin with, we establish a connection between $d$ and $d_s$ which forms the basis for the remainder of this section.

\begin{lemma}\label{lem:ConstantCs}
	Let $X =L\circ P_{\Dx}((u,\mu,\nu)) \in \F^{\alpha, 0}_0$ for some $(u,\mu,\nu)\in \D_0^\alpha$. Moreover, assume that $\Dx\leq 1$  and let $\widehat{X}_{\Dx}(t) = S_{t}(X)$ and $X_{\Dx}(t)=S_{\Dx, t} (X)$ for $t\in [0,T]$. Then 
	\begin{equation}\label{eq:equivalent}
	 d\!\left(\widehat{X}_{\Dx}(t), X_{\Dx}(t) \right)\leq C_sd_s\!\left(\widehat{X}_{\Dx}(t), X_{\Dx}(t) \right)\!,
	\end{equation}
	where
	\begin{equation}\label{eq:constantCs}
		C_s = 2 + \|\alpha'\|_{\infty} \bigg(\|u\|_{\infty} +  \left(2+ \sqrt{2} + e^{\frac{1}{4}T} \right)\sqrt{G_{\infty}} + \bigg(1 + \frac{1}{4}T \bigg)G_{\infty}\bigg).
	\end{equation}
\end{lemma}

\begin{proof}
To ease the readability we drop the subindex $\Dx$. 

Since $\widehat X(0)=X=X(0)$, it follows by \eqref{eq:waveBreakingFunc} and \eqref{eq:numBreakFunction} that $\widehat{\tau}= \tau$, and hence $\Omega_j(\widehat{X}(t)) = \Omega_j(X(t))$ for $j\in \{c,d\}$. Thus, by \eqref{eq:g} and \eqref{eq:gg} we have 
\begin{equation*}
|g_1(\widehat{X}(t))(\xi)- g_1(X(t))(\xi)| = g(\widehat{X}(t),X(t))(\xi)\quad \text{ for }\xi \in \Omega_c(X(t)),
\end{equation*}
whereas for $\xi \in \Omega_d(X(t))$ it holds that
\begin{align*}
	|g_1(\widehat{X}(t))(\xi)- g_1(X(t))(\xi)| &\leq |\left(1-\alpha(\widehat{y}(t, \xi))\right)(\widehat{V}_{\xi} - V_{\xi})(t, \xi) |
	\\ & \qquad + \left|\alpha(\widehat{y}(t, \xi)) - \alpha(y(t, \xi))\right|\!V_{\xi}(t, \xi)
	\\ & \leq |\widehat{V}_{\xi} - V_{\xi}|\!(t, \xi) + \|\alpha'\|_{\infty}V_{\xi}(t, \xi)\!\left|\widehat{y} - y\right|\!(t, \xi).
\end{align*}
Note that one can argue similarly with the roles of $\widehat{X}(t)$ and $X(t)$ reversed, and hence 
\begin{equation*}
|g_1(\widehat{X}(t))(\xi)- g_1(X(t))(\xi)| \leq g(\widehat{X}(t),X(t))(\xi)\quad \text{ for } \xi \in \Omega_d(X(t)).
\end{equation*}
Thus, one obtains 
\begin{align*}
	\|g_1(\widehat{X}(t)) + \widehat y_{\xi}(t)&- g_1(X(t)) -y_{\xi}(t)\|_2 \\
& \leq \|g_1(\widehat{X}(t))- g_1(X(t)) \|_2 + \|\widehat y_{\xi}(t)-y_{\xi}(t)\|_2\\
& \leq \|g(\widehat{X}(t),X(t))\|_2+  \|\widehat y_{\xi}(t)-y_{\xi}(t)\|_2.
\end{align*}
 Furthermore, by \eqref{eq:g_2} and $G_\infty=H_{\infty}(t)=\widehat H_{\infty}(t)$ for all $t\in [0,T]$, we have 
	\begin{equation*}
		\|g_2(\widehat{X}(t)) - g_2(X(t))\|_2 \leq \|\alpha'\|_{\infty}G_{\infty}\|\widehat{U}_{\xi}(t) - U_{\xi}(t)\|_2,
	\end{equation*}
and, by \eqref{eq:g_3}, 
	\begin{align}\label{eq:bound_g3}
		\|g_3(\widehat{X}(t)) - g_3(X(t))\|_2 &\leq \|\alpha'\|_{\infty}\|\widehat{U}_{\xi}(t)\|_2 \|\widehat{U}(t) - U(t)\|_{\infty} \nonumber
		\\ & \quad +\|\alpha'\|_{\infty}\|U(t)\|_{\infty} \|\widehat{U}_{\xi}(t) - U_{\xi}(t)\|_2.
	\end{align}
	Moreover, since $\widehat{H}_{\xi}(t) = H_{\xi}(t) = H_{\xi}(0)$ and $0\leq H_{\xi}(0)\leq 1$, we immediately obtain\vspace{-0.07cm}
	\begin{align*}
		\|\alpha'\|_{\infty} \|\widehat{U}\widehat{H}_{\xi}(t) - UH_{\xi}(t)\|_2 &\leq \|\alpha'\|_{\infty} \|H_{\xi}(0)\|_2 \|\widehat{U}(t) - U(t)\|_{\infty}\\
		& \leq  \|\alpha'\|_{\infty} \sqrt{G_\infty} \|\widehat{U}(t) - U(t)\|_{\infty},	\end{align*}
and by Definition~\ref{def:metric}, we end up with 
\begin{align*}
	 d(\widehat{X}(t), X(t)) &\leq \! \left(2+\|\alpha'\|_\infty\Big(\sqrt{G_\infty}+ G_\infty +  \|\hat U_{\xi}(t)\|_2+ \|U(t)\|_\infty\Big)\right) d_s(\widehat{X}(t), X(t)).
\end{align*}

It remains to bound  $\|\widehat{U}_{\xi}(t)\|_2$ and $\|U(t)\|_{\infty}$. Following the proof of \cite[(2.15)]{AlphaHS}, we find, as $\widehat X(0)\in \F_0^{\alpha,0}$,
\begin{equation}\label{eq:upper_bound_deriv}
	(\widehat{y}_{\xi}+ \widehat{H}_{\xi})(t, \xi) \leq e^{\frac{1}{2}t}\left(\widehat y_{\xi}(0, \xi) + \widehat H_{\xi}(0, \xi) \right)= e^{\frac12 t}\!,
\end{equation}
which together with Definition~\ref{def:LagSet}~\ref{def:importantRelation} implies 
\begin{equation}\label{eq:bound_deriv}
	\|\widehat{U}_{\xi}(t) \|_2 \leq e^{\frac{1}{4}T}\sqrt{G_{\infty}}.
\end{equation}

To estimate $\|U(t)\|_{\infty}$ for any $t \in [\tau_k^*, \tau_{k+1}^*]$,  recall that $X(t)$ is defined as the linear interpolant of $\{X_j(t)\}_{j\in \mathbb{Z}}=\{X_j^{M_{\mathrm{it}}}(t)\}_{j \in \mathbb{Z}}$, cf. \eqref{eq:numSolOP}, which combined with Definition~\ref{def:EulSet}~\ref{def:EulSet2},   \eqref{eq:numLagrSystemIter}, \eqref{eq:recursiveU}, \eqref{eq:estUDx}, and induction implies
\begin{align}\label{eq:upper_boundU}
	\left|U(t, \xi) \right| &\leq \left|U(\tau_k^*, \xi) \right| + \frac14 \int_{\tau_k^*}^t V_\infty (s)ds\\ \nonumber
	 &\leq |U(\tau_k^*, \xi)| + \frac{1}{4}V_{\infty}(\tau_k^*)(t-\tau_k^*) \\ \nonumber
	 & \leq |U(0,\xi)|+ \frac14 V_\infty(0)t
	\\ \nonumber&\leq \|u\|_\infty +\left(1 + \sqrt{2} \right)\!\sqrt{G_\infty} \sqrt{\Delta x} + \frac{1}{4} G_{\infty}T. \qedhere
\end{align}
\end{proof}

Up to this point, we have shown for any $t \in [0, T]$, cf. \eqref{eq:triangle} and Lemma~\ref{lem:ConstantCs}, that
\begin{equation*}
	d(X(t), X_{\Dx}(t)) \leq C(t)e^{D(t)t}d(X, X_{\Dx}) + C_sd_s(\widehat{X}_{\Dx}(t), X_{\Dx}(t)),
\end{equation*}
with the first term tending to $0$ as $\Dx \rightarrow 0$. Hence it remains to prove 
\begin{equation}\label{claim:2}
	d_s(\widehat X_{\Dx}(t), X_{\Dx}(t))\to 0 \quad \text{ as } \Dx\to 0.
\end{equation}

To this end, note that for $t= \tau_k^*+h \in [\tau_k^*, \tau_{k+1}^*]$, we can apply the triangle inequality and write
 \begin{align}\nonumber
 	d_s(\widehat{X}_{\Dx}(t), X_{\Dx}(t)) &\leq d_s(S_h(\widehat{X}_{\Dx}(\tau_k^*)), S_h(X_{\Dx}(\tau_k^*))) 
	\\ & \qquad + d_s(S_h(X_{\Dx}(\tau_k^*)), S_{\Dx, h}(X_{\Dx}(\tau_k^*))).\label{eq:recursiveBeg}
\end{align}
Here the first term on the right-hand side measures how the error between $\hat X_{\Delta x}(\tau_k^*)$ and $X_{\Dx}(\tau_k^*)$ evolves with respect to time, while the second term measures the error introduced by terminating the iteration scheme after finitely many iterations and approximating the amount of energy to be removed.

\begin{remark}
	As observed in \cite[Sec. 3]{NewestLipschitzMetric}, $g(X, \hat{X})$ does not satisfy the triangle inequality in general, due to the minimum appearing in \eqref{eq:gg}. However, if we are given $X$, $\hat X$, and $\tilde X$ such that $\tau=\hat\tau=\tilde\tau$ and $\hat V_\xi(\xi)= \tilde V_{\xi}(\xi)$ for all $ \xi \in \Omega_{d}(X)$, then 
	\begin{equation*}
	g(X, \hat X)\leq g(X, \tilde X)+ g(\tilde X, \hat X),
	\end{equation*}
which implies \eqref{eq:recursiveBeg}.
\end{remark}

The two terms on the right-hand side of \eqref{eq:recursiveBeg} are analyzed separately, and we start by deriving an estimate for the first term which allows us to apply Gronwall's inequality.

\begin{lemma}\label{lem:GronwallIntegrated}
Let $X=L\circ P_{\Dx}((u,\mu,\nu))\in \F_0^{\alpha,0}$ for some $(u,\mu,\nu)\in \D_0^\alpha$ and recall that $\widehat X_{\Dx}(t)=S_t(X)$. Moreover, introduce $\tilde X_{\Dx}(t)= S_h(X_{\Dx}(\tau_k^*))$ for $t=\tau_k^*+h$, with $0\leq h\leq \Dt_k$ and $X_{\Dx}(\tau_k^*)=S_{\Dx,\tau_k^*}(X)$. Then 
\begin{equation*}
	d_s(\widehat X_{\Dx}(t), \tilde X_{\Dx}(t))\leq d_s(\widehat X_{\Dx}(\tau_k^*), X_{\Dx}	(\tau_k^*))+ \lambda \int_{\tau_k^*}^{t} d_s(\widehat X_{\Dx}(r), \tilde X_{\Dx}(r))dr,
\end{equation*}
where 
\begin{equation*}
	\lambda=1+\Big(1+\|\alpha'\|_\infty+ \frac12T\Big)\sqrt{G_\infty}+ \|\alpha'\|_\infty \Big(1+\frac12 T \Big)G_\infty.
\end{equation*}
\end{lemma}

\begin{proof}
To ease the readability we drop the subindex $\Dx$. 

Since $\widehat X(0)=X=X(0)$, it follows that $\hat \tau= \tilde \tau=\tau$, which in turn implies $\Omega_j(\widehat X(t))=\Omega_j(\tilde X(t))$ for $j\in \{c,d\}$ and $t\in [\tau_k^*, \tau_{k+1}^*]$. As a consequence, $g(\widehat X(t), \tilde X(t))$ is well-defined and 
\begin{equation}\label{rel:Vxig}
\|\widehat{V}_{\xi}(t) - \tilde{V}_{\xi}(t)\|_2 \leq \|g(\widehat{X}(t), \tilde{X}(t))\|_2. 
\end{equation}
It therefore follows from \eqref{eq:intLagrSystem} that 
\begin{align}\nonumber
	 \| \widehat y(t)-\tilde y(t&)\|_\infty + \| \widehat y_\xi(t)-\tilde y_{\xi}(t)\|_2+ \|\widehat U_\xi(t)-	\tilde U_{ \xi}(t)\|_2\\ \nonumber
	& \leq \| \widehat y(\tau_k^*)-\tilde y(\tau_k^*)\|_\infty+ \| \widehat y_\xi(\tau_k^*)-\tilde 	y_{\xi}(\tau_k^*)\|_2+ \|\widehat U_\xi(\tau_k^*)-\tilde U_{ \xi}(\tau_k^*)\|_2\\ \nonumber 
	& \quad + \int_{\tau_k^*}^t \! \!\left( \|\widehat U(s)-\tilde U(s)\|_\infty + \|\widehat U_\xi(s)-\tilde U_{\xi}(s)\|_2+ \frac12 \|g(\widehat X(s), \tilde X(s))\|_2 \! \right)\!ds\\ \nonumber
	& = \| \widehat y(\tau_k^*)- y (\tau_k^*)\|_\infty+ \| \widehat y_\xi(\tau_k^*)-y_{\xi}(\tau_k^*)\|_2+ \|\widehat U_\xi(\tau_k^*)-U_{ \xi}(\tau_k^*)\|_2\\ \label{Gron1}
	& \quad  + \int_{\tau_k^*}^t \! \!\left( \|\widehat U(s)-\tilde U(s)\|_\infty + \|\widehat U_\xi(s)-\tilde U_{\xi}(s)\|_2+ \frac12 \|g(\widehat X(s), \tilde X(s))\|_2 \! \right)\!ds. \raisetag{-20pt}
\end{align}
Furthermore, \eqref{eq:int_ODE2} implies 
\begin{align}\label{eq:bound_Udiff}
	\|\widehat U(t)-\tilde U(t)\|_\infty &\leq \|\widehat U(\tau_k^*)-\tilde U(\tau_k^*)\|_\infty+\frac14  	\int_{\tau_k^*}^t \|\widehat V_\xi(s)-\tilde V_\xi(s)\|_1 ds \nonumber \\
	& = \|\widehat U(\tau_k^*)-U(\tau_k^*)\|_\infty +\frac14  \int_{\tau_k^*}^t \|\widehat V_\xi(s)-\tilde V_\xi(s)\|_1 ds.
\end{align}
Since $\widehat V_\xi(s, \xi)-\tilde V_\xi(s, \xi)=0$ for all $\xi \in \Omega_c(0)\cup\Omega_d(t)$ and, by \cite[Corr. 2.4]{AlphaHS}, 
\begin{equation}\label{eq:boundMeasure}
	\mathrm{meas}( \Omega_d(0) \cap \Omega_c(t))\leq \! \left(1+ \frac{1}{4}t^2\right) G_\infty\leq \Big(1+\frac12 t\Big)^{\! 2} G_\infty,
\end{equation}
we have, by the Cauchy-Schwarz inequality and \eqref{rel:Vxig},
\begin{equation}\label{L1V}
	\|\widehat V_\xi(s)-\tilde V_\xi(s)\|_1\leq \left(1+ \frac12T\right)\! \sqrt{G_\infty}\|g(\widehat X(s), \tilde X(s))\|_2,
\end{equation}
and hence  
\begin{align}\nonumber
	\|\widehat U(t)-\tilde U(t)\|_\infty 
& \leq  \|\widehat U(\tau_k^*)-U(\tau_k^*)\|_\infty \\
&  \qquad+ \frac14 \left(1+ \frac12T\right)\! \sqrt{G_\infty}\int_{\tau_k^*}^t \|g(\widehat X(s), \tilde X(s))\|_2 ds.  \label{Gron2}
\end{align}

Finally we can turn our attention towards $g(\widehat X(t), \tilde X(t))$, which is decreasing across wave breaking times, see \cite[Prop. 3.1]{NewestLipschitzMetric}. Let $\Omega_j(t) = \Omega_j(\tilde{X}(t))$ for $j \in \{c, d\}$. There are three scenarios to consider: If $\xi \in \Omega_c(\tau_k^*)$, then 
\begin{equation*}
g(\widehat X(t), \tilde X(t))(\xi) = g(\widehat X(\tau_k^*), \tilde X(\tau_k^*))(\xi),
\end{equation*}
 while for $\xi \in \Omega_d(\tau_k^*) \cap \Omega_c(t)$ we have, by \eqref{eq:intLagrSystem}, $\widehat{V}_{\xi}(\tau_k^*) = \tilde{V}_{\xi}(\tau_k^*)$, and by \eqref{eq:gg}, 
 \begin{align*}
	g(\widehat{X}(t), \tilde{X}(t))(\xi) &= |\alpha\left(\widehat{y}(\tau(\xi), \xi)\right)\widehat{V}_{\xi}(\tau_k^*, \xi)  - \alpha(\tilde{y}(\tau(\xi), \xi))\tilde{V}_{\xi}(\tau_k^*, \xi)|
	\\ & \leq \|\alpha'\|_{\infty} \tilde{V}_{\xi}(\tau_k^*, \xi)\!\left( \!|\widehat{y}(\tau_k^*, \xi) - \tilde{y}(\tau_k^*, \xi)| + \int_{\tau_k^*}^t |\widehat{U}(s, \xi) - \tilde{U}(s, \xi)|ds \! \right)
	\\ &\leq g(\widehat{X}(\tau_k^*), \tilde{X}(\tau_k^*))(\xi) + \|\alpha'\|_{\infty} \tilde{V}_{\xi}(\tau_k^*, \xi)\int_{\tau_k^*}^t \|\widehat{U}(s) - \tilde{U}(s)\|_{\infty}ds. 
\end{align*}
Using \eqref{eq:bound_Udiff} and \eqref{L1V} leads to the following estimate when $\xi \in \Omega_d(\tau_k^*) \cap \Omega_d(t)$
\begin{align*}
	g(\widehat{X}(t), \tilde{X}(t))(\xi) &\leq g(\widehat{X}(\tau_k^*), \tilde{X}(\tau_k^*))(\xi) 
	\\ & \quad + \|\alpha'\|_{\infty} \tilde{V}_{\xi}(\tau_k^*, \xi) \!\int_{\tau_k^*}^t \! \!\left(\|\widehat{U}(s) - \tilde{U}(s)\|_{\infty} + \frac{1}{4}\|\widehat{V}_{\xi}(s) - \tilde{V}_{\xi}(s)\|_1\! \right)\!ds\\
	& \leq g(\widehat{X}(\tau_k^*), \tilde{X}(\tau_k^*))(\xi)+  \|\alpha'\|_{\infty} \tilde{V}_{\xi}(\tau_k^*, \xi)\int_{\tau_k^*}^t \!\|\widehat{U}(s) - \tilde{U}(s)\|_{\infty}ds  \\
	& \quad +\frac{1}{4}\left(1+\frac12 T\right)\|\alpha'\|_{\infty}\sqrt{G_\infty} \tilde{V}_{\xi}(\tau_k^*, \xi)\int_{\tau_k^*}^t \|g(\hat X(s), \tilde X(s)\|_2 ds.
\end{align*}
As a consequence, \vspace{-0.2cm}
\begin{align}\nonumber
	\|g(\widehat X(t), \tilde X(t))\|_2& \leq \|g(\widehat X(\tau_k^*), X(\tau_k^*))\|_2 + \|\alpha'\|_\infty \sqrt{G_\infty} \int_{\tau_k^*}^t \|\widehat U(s)-\tilde U(s)\|_\infty ds\\ \label{Gron3}
	& \qquad + \frac14\left(1+\frac12 T\right)\|\alpha ' \|_\infty G_\infty\int_{\tau_k^*}^t \|g(\widehat X(s),\tilde X(s))\|_2 ds,
\end{align}
since $\max (\widehat V_\xi(s, \xi), \tilde V_\xi(s, \xi))\leq V_\xi(\xi)\leq 1$ for all $\xi \in \mathbb{R}$.
Combining \eqref{Gron1}, \eqref{Gron2}, and \eqref{Gron3} finishes the proof. 
\end{proof}

After applying Gronwall's inequality to Lemma~\ref{lem:GronwallIntegrated}, we end up with the following estimate.

\begin{corollary}\label{cor:GronwallIntegrated}
Let $X=L\circ P_{\Dx}((u,\mu,\nu))\in \F_0^{\alpha,0}$ for some $(u,\mu,\nu)\in \D_0^\alpha$ and recall that $\widehat X_{\Dx}(t)=S_t(X)$. Moreover, introduce $\tilde X_{\Dx}(t)= S_h(X_{\Dx}(\tau_k^*))$ for $t=\tau_k^*+h$, with $0\leq h\leq \Dt_k$ and $X_{\Dx}(\tau_k^*)=S_{\Dx,\tau_k^*}(X)$. Then 
\begin{equation*}
	d_s(\widehat X_{\Dx}(\tau_k^*+h), \tilde X_{\Dx}(\tau_k^*+h))\leq e^{\lambda h} d_s(\widehat 	X_{\Dx}(\tau_k^*), X_{\Dx}(\tau_k^*)),
\end{equation*}
where 
\begin{equation*}
\lambda=1+\Big(1+\|\alpha'\|_\infty+\frac12 T \Big)\sqrt{G_\infty}+ \|\alpha'\|_\infty \Big(1+\frac12 T \Big)G_\infty.
\end{equation*}
\end{corollary}

Looking back at \eqref{eq:recursiveBeg} and using the notation from Corollary~\ref{cor:GronwallIntegrated}, we have, for any $t \in [\tau_k^*, \tau_{k+1}^*]$, 
\begin{align*}
	d_s(\widehat X_{\Dx}(t),X_{\Dx}(t))& \leq d_s(\widehat X_{\Dx}(t), \tilde X_{\Dx}(t)) + d_s(\tilde X_{\Dx}(t), X_{\Dx}(t))\\
	& \quad \leq e^{\lambda (t-\tau_k^*)} d_s(\widehat X_{\Dx}(\tau_k^*), X_{\Dx}(\tau_k^*))+  d_s(\tilde X_{\Dx}(t), X_{\Dx}(t)),
\end{align*}
and hence by induction, since $\widehat X_{\Dx}(0) = X_{\Dx}(0)$, 
\begin{align}\label{eq:inductionArg}
d_s(\widehat X_{\Dx}(t)&, X_{\Dx}(t)) \nonumber
\\ & \leq e^{\lambda t} d_s(\widehat X_{\Dx}(0), X_{\Dx}(0)) + d_s(S_h(X_{\Dx}(\tau_k^*)), X_{\Dx}(t)) \nonumber\\
	& \qquad +\sum_{l=0}^{k-1} e^{\lambda (t-\tau_{l+1}^*)}d_s(S_{\Dt_l}(X_{\Dx}(\tau_{l}^*)), X_{\Dx}(\tau_{l+1}^*)) \nonumber\\
	& \leq e^{\lambda T}\!\left(\!d_s\big(S_h(X_{\Dx}(\tau_k^*)\big), X_{\Dx}(t))+\sum_{l=0}^{k-1} d_s\big(S_{\Dt_l}(X_{\Dx}(\tau_{l}^*)), X_{\Dx}(\tau_{l+1}^*)\big)\!\right)\!.
\end{align}
Thus, in view of \eqref{claim:2}, it remains to show that the above sum tends to $0$ as $\Dx\to 0$, which amounts to analyzing the aforementioned local errors introduced by the numerical solution operator $S_{\Dx, t}$. Since all the terms in the above sum have the same structure, we focus on estimating the first one.  

\begin{lemma}\label{lem:locerror}
Let $X=L\circ P_{\Dx}((u,\mu,\nu))\in \F_0^{\alpha,0}$ for some $(u,\mu,\nu)\in \D_0^\alpha$ and $\Delta x\leq 1$. Moreover, recall that $X_{\Dx}(t)=S_{\Dx,t}(X)$ and introduce $\tilde X_{\Dx}(t)= S_h(X_{\Dx}(\tau_k^*))$ for $t= \tau_k^*+h$ with $0\leq h\leq \Dt_k$. If $\Dt_k>\Dt$, then 
\begin{equation}\label{est:locerror1}
	d_s(\tilde X_{\Dx}(t), X_{\Dx}(t))=0,
\end{equation}
and if $\Dt_k\leq \Dt$, then 
\begin{align}\nonumber
	d_s(\tilde X_{\Dx}(t), X_{\Dx}(t))&\leq (1+\tilde C_s)(1+T)^2\big(1+\|\alpha'\|_{\infty}\sqrt{G_{\infty}}\big)G_\infty \Dt \Dt_k
	\\ \label{est:locerror2}& \qquad \qquad +2 \tilde C_s (1+T)^2\sqrt{\mathrm{meas}(B_k)}\Dt_k, 
\end{align}
where 
\begin{equation}\label{eq:tildeCs}
	\tilde C_s= \|\alpha'\|_\infty \Big(\|u\|_\infty +(1+\sqrt{2})\sqrt{G_\infty}+\frac14 G_\infty T\Big),
\end{equation}
and 
\begin{equation*}
	B_k=\{\xi \in \mathbb{R}: \tau_k^*<\tau(\xi)\leq\tau_{k+1}^*\}.
\end{equation*}
\end{lemma}

\begin{proof}
To simplify the notation we omit the subindex $\Dx$. 

If $\Dt_k>\Dt$, then Remark~\ref{rem:evolveExact} implies that $\tilde X(t)=X(t)$ for all $t\in [\tau_k^*, \tau_{k+1}^*]$, from which \eqref{est:locerror1} follows. 

For $\Dt_k\leq \Dt$,  we distinguish between two cases: 

{\it Case $1$:} If no wave breaking occurs in $(\tau_k^*, \tau_{k+1}^*)$, then it follows, as in the case where $\Dt_k > \Dt$, that $\tilde X(t)=X(t)$ for all $t\in [\tau_k^*, \tau_{k+1}^*]$,  hence $d_s(\tilde X(t), X(t))=0$. 

{\it Case  $2$:} If $\tau_k^*=\hat\tau_m$ and $\tau_{k+1}^*= \hat\tau_{m+l}$ for $m\in \mathbb{N}_0$ and $l\geq 2$, then 
\begin{equation}\label{cancel}
	\tilde X_{\xi}(t,\xi)= X_{\xi}(t,\xi) \quad \text{ for all } (t,\xi)\in [\tau_k^*, \tau_{k+1}^*]\times B_k^c.
\end{equation}
Furthermore, $\tilde X(t)$ and $X(t)$ are piecewise linear and continuous with nodes located at $\{\xi_j\}_{j\in \mathbb{Z}}$ for any $t\in [\tau_k^*, \tau_{k+1}^*]$ and we can therefore, recalling \eqref{iterseq}, associate to $\tilde X(t)$ and $X(t)$ the sequences $\{\tilde X_j(t)\}_{j\in\mathbb{Z}}= \{\tilde X(t,\xi_j)\}_{j\in \mathbb{Z}}$ and $\{X_j(t)\}_{j \in \mathbb{Z}} = \{ X_j^{M_{\mathrm{it}}^k}(t)\}_{j\in\mathbb{Z}}= \{ X^{M_{\mathrm{it}}^k}(t,\xi_j)\}_{j\in \mathbb{Z}}$,
respectively. Here $M_{\mathrm{it}}^k$ is chosen such that
\begin{equation}\label{eq:stopCond}
	\sup_{t \in [\tau_k^*, \tau_{k+1}^*]} \| \{y_j^{M_{\mathrm{it}}^k}(t)\} - \{y_j^{M_{\mathrm{it}}^k-1}(t)\} \|_{\ell^{\infty}} \leq \frac{1}{8}G_{\infty} \Dt^2 \Dx,
\end{equation}
and satisfies, see the discussion proceeding Proposition~\ref{prop:termination}, $2 \leq M_{\mathrm{it}}^k \leq 3$. Moreover, let $\{X_j^{M_{\mathrm{it}}^k-1}(t)\}_{j \in \mathbb{Z}}\!=\!\{X^{M_\mathrm{it}^{k}-1}(t,\xi_j)\}_{j \in \mathbb{Z}}$ and recall \eqref{eq:sec_lagr}.

Using \eqref{eq:intLagrSystem},  \eqref{eq:numLagrSystemIter}, and \eqref{cancel} we find 
\begin{align}
	\|\tilde y(t)-y(t)\|_\infty& \leq \|\{\tilde y_j(t)\}-\{y_j(t)\}\|_{\ell^\infty} \nonumber \\ \nonumber
	&  \leq \int_{0}^h \| \{\tilde U_j(\tau_k^*+s)\}-\{U_j(\tau_k^*+s)\}\|_{\ell^\infty} ds\\ \nonumber
	& \leq \frac14 \int_0^h \int_0^s \sum_{j=-\infty}^\infty \vert \tilde V_{j+\frac12, \xi}-V_{j+\frac12, \xi}\vert (\tau_k^*+r)(\xi_{j+1}-\xi_j) dr ds\\ \nonumber
	& \leq \frac18 \sum_{j=-\infty}^\infty V_{j+\frac12, \xi}(\tau_k^*)\chi_{\{m:\tau_k^*<\tau_{m+\frac12}<\tau_{k+1}^*\}}(j)(\xi_{j+1}-\xi_j) h^2\\ \nonumber
	& \leq \frac18 V_{\Dx, \infty}(\tau_k^*)\Dt_k \Dt\\ \label{est:yn}
	& \leq \frac 18 G_\infty \Dt_k\Dt.
\end{align}
Observe that  for $\xi \in B_k \cap [\xi_j, \xi_{j+1}]$, it holds by \eqref{eq:intLagrSystem} and \eqref{eq:numLagrSystemIter}--\eqref{eq:beta} that
\begin{align*}
	\tilde{V}_{j+\frac{1}{2}, \xi}(t) - V_{j + \frac{1}{2}, \xi}(t) &= \begin{cases}
	0, & t < \tau_{j+\frac{1}{2}}, \\
	\left(\alpha(y_j^{M_{\mathrm{it}}^k-1}(\tau_{k+1}^*))-\alpha(\tilde{y}_j(\tau_{j+\frac{1}{2}}))\right)\!V_{j+\frac{1}{2}, \xi}(\tau_k^*), & \tau_{j+\frac{1}{2}} \leq t. 
	\end{cases}
\end{align*}
Furthermore, \eqref{eq:upper_boundU} does not only hold for $U(t)$, but also for $\tilde U_j(t)$ and $U_j^{M_{\mathrm{it}}^k-1}(t)$ for any $j \in \mathbb{Z}$, which implies, for all $j \in \mathbb{Z}$ with $\tau_{j+\frac12}\in (\tau_k^*, \tau_{k+1}^*]$ that 
\begin{align}\label{error:ds:alpha}
	\vert \tilde y_j(\tau_{j+\frac{1}{2}})- y_j^{M_{\mathrm{it}}^k-1}(\tau_{k+1}^*)\vert &= \left|\int_{\tau_k^*}^{\tau_{j+\frac{1}{2}}}\tilde{U}_j(s)ds - \int_{\tau_k^*}^{\tau_{k+1}^*}U_j^{M_{\mathrm{it}}^k-1}(s)ds \right| \nonumber
	\\ & \leq 2\Big(\|u\|_\infty +\left(1 + \sqrt{2} \right)\!\sqrt{G_\infty} + \frac{1}{4} G_{\infty}T \Big) \Dt_k,
\end{align} 
since $\Dx\leq 1$. Thus, 
\begin{align}\label{est:Un}
\|\tilde U(t)-U(t)\|_\infty& \leq \|\{\tilde U_j(t)\}-\{U_j(t)\}\|_{\ell^\infty} \nonumber \\ \nonumber
	& \leq \frac14 \int_0^h \sum_{j=-\infty}^\infty \vert \tilde V_{j+\frac12, \xi}(\tau_k^*+s)-V_{j+\frac12, \xi}(\tau_k^*+s)\vert(\xi_{j+1}-\xi_j)ds\\ \nonumber
	& \leq \frac14 \|\alpha '\|_\infty \! \sum_{j=-\infty}^\infty \! \vert \tilde y_j(\tau_{j+\frac{1}{2}})- y_j^{M_{\mathrm{it}}^k-1}(\tau_{k+1}^*)\vert \chi_{\{m: \tau_k^*<\tau_{m+\frac12}\leq \tau_{k+1}^*\}}(j)\\ \nonumber
	& \qquad \qquad \qquad \qquad\qquad  \times V_{j+\frac12,\xi}(\tau_{k}^*)(\xi_{j+1}-\xi_j)h\\ 
	& \leq \frac12 \|\alpha'\|_\infty  \Big(\|u\|_\infty +\left(1 + \sqrt{2} \right)\!\sqrt{G_\infty}  + \frac{1}{4} G_{\infty}T \Big)G_\infty \Dt_k\Dt \nonumber
	\\ &= \frac{1}{2} \tilde{C}_s G_{\infty}\Dt_k \Dt, 
\end{align}
with $\tilde{C}_s$ defined in \eqref{eq:tildeCs}. 

Next, we focus on $g(\tilde X(t), X(t))$. To this end, let $\Omega_m^j(t) = \Omega_m(\tilde{X}(t)) \cap [\xi_j, \xi_{j+1}]$ for $m \in \{c, d\}$ and $j \in \mathbb{Z}$, then 
\begin{align}\label{eq:g_proof}
g(\tilde X(t), X(t))(\xi)= 
	\begin{cases} 0, &  \xi \in \Omega_c^j(t) \cap B_k^c,\\
	|\alpha(\tilde{y}_j(\tau_{j+\frac{1}{2}}))-\alpha(y_j^{M_{\mathrm{it}}^k-1}(\tau_{k+1}^*))\vert V_{\xi}(\tau_k^*, \xi), &  \xi \in \Omega_c^j(t) \cap B_k, \\ 
\!\! \|\alpha'\|_{\infty}\bigl(|\tilde{y}(t, \xi) - y(t, \xi)| + |\tilde{U}(t, \xi) - U(t, \xi)|\bigr)
\\ \quad \qquad \times V_{\xi}(\tau_k^*, \xi), & 	 \xi \in \Omega_d^j(t). 
\end{cases} \raisetag{-60pt}
\end{align}
This means that $g(\tilde X(t), X(t))(\xi)$ can have jumps across wave breaking times and is in general discontinuous with respect to time. On the other hand, one has, using \eqref{est:yn}, \eqref{error:ds:alpha}, and \eqref{est:Un}, 
\begin{align} \label{est:gn}
	\|g(\tilde X(t), X(t))\|_2 & \leq  \|\alpha'\|_{\infty} \!\left( \|\{\tilde y_j(t)\} - \{y_j(t)\} \|_{\ell^{\infty}} + \|\{\tilde U_j(t)\} - \{U_j(t)\} \|_{\ell^{\infty}} \right) \!\sqrt{G_{\infty}} \nonumber \\ 
	& \qquad + \|\alpha'\|_\infty\|\{\tilde y_j(\tau_{j+\frac{1}{2}})\}- \{y_j^{M_{\mathrm{it}}^k-1}(\tau_{k+1}^*)\}\|_{\ell^{\infty}} \sqrt{\mathrm{meas}(B_k)} \nonumber \\
	& \leq \frac{1}{2}\|\alpha'\|_{\infty}\!\left(\frac{1}{4} + \tilde{C}_s\right)\!G_{\infty}^{\frac{3}{2}}\Dt_k \Dt  +  2\tilde{C}_s\sqrt{\mathrm{meas}(B_k)}\Dt_k, 
\end{align}
where we used that $0\leq V_{\Delta x,\xi}(\tau_k^*)\leq 1$.
As an immediate consequence,
\begin{align}\nonumber
\|\tilde U_\xi(t)-U_{\xi}(t)\|_2& \leq \frac12 \int_0^h \|g(\tilde X(\tau_k^*+s), X(\tau_k^*+s))\|_2 ds\\ \label{est:Uxin}
& \leq \frac14 \|\alpha'\|_\infty \!\left(\frac{1}{4} + \tilde{C}_s\right)\! G_{\infty}^{\frac{3}{2}} \Dt_k^2 \Dt + \tilde{C}_s \sqrt{\mathrm{meas}(B_k)}\Dt_k^2 \nonumber
\\ &\leq  \frac14 \|\alpha'\|_\infty \! \left(\frac{1}{4} + \tilde{C}_s\right)\! G_{\infty}^{\frac{3}{2}} T \Dt_k \Dt + \tilde{C}_s \sqrt{\mathrm{meas}(B_k)}T\Dt_k, 
\end{align}
which holds by \eqref{rel:Vxig}, and, moreover, 
\begin{align} \label{est:yxi:n}
\|\tilde y_\xi(t)-y_{ \xi}(t)\|_2&  \leq \int_0^h \|\tilde U_\xi(\tau_k^*+s)-U_{\xi}(\tau_k^*+s)\|_2ds \nonumber \\ 
& \leq \frac14 \|\alpha'\|_\infty \! \left(\frac{1}{4} + \tilde{C}_s\right)\!  G_{\infty}^{\frac{3}{2}} T^2 \Dt_k \Dt + \tilde{C}_s \sqrt{\mathrm{meas}(B_k)}T^2 \Dt_k.
\end{align}
Combining \eqref{est:yn}, \eqref{est:Un}, and \eqref{est:gn}--\eqref{est:yxi:n} proves the statement.
\end{proof}

\begin{lemma}\label{thm:localError}
Let $X=L\circ P_{\Dx}((u,\mu,\nu))\in \F_0^{\alpha,0}$ for some $(u,\mu,\nu)\in \D_0^\alpha$ and $\Dx\leq 1$. Moreover, recall that $X_{\Delta x}(t)= S_{\Delta x,t}(X)$ for $t=\tau_k^* + h \in [\tau_k^*, \tau_{k+1}^*]$, then 
\begin{align*}
	d_s(S_h(X_{\Dx}(\tau_k^*)), X_{\Dx}(t))&+\sum_{l=0}^{k-1} d_s(S_{\Dt_l}(X_{\Dx}(\tau_{l}^*)),X_{\Dx}(\tau_{l+1}^*))
	\\ & \leq (1+\tilde C_s)(1+T)^2\big(1+\|\alpha'\|_{\infty}\sqrt{G_{\infty}}\big)G_\infty T \Dt   
	\\ & \qquad + 2\tilde C_s (1+T)^3 \sqrt{G_\infty} \sqrt{T}\sqrt{\Dt}, 
\end{align*}
where 
\begin{equation*}
\tilde C_s= \|\alpha'\|_\infty \Big(\|u\|_\infty + (1+\sqrt{2})\sqrt{G_\infty}+ \frac14 G_\infty T\Big).
\end{equation*}
\end{lemma}

\begin{proof}
Introduce 
\begin{equation*}
	\mathcal{K}=\{l: \Dt_l\leq \Dt\}.
\end{equation*}
Then by Lemma~\ref{lem:locerror} and the Cauchy--Schwarz inequality, 
\begin{align*}
d_s(S_h(X_{\Dx}(\tau_k^*))&, X_{\Dx}(t)) +\sum_{l=0}^{k-1} d_s(S_{\Dt_l}(X_{\Dx}(\tau_{l}^*)), X_{\Dx}(\tau_{l+1}^*))\\
	&\leq (1+\tilde C_s)(1+T)^2\big(1+\|\alpha'\|_{\infty}\sqrt{G_{\infty}}\big)G_\infty\Dt \sum_{l\in \mathcal{K}} \Dt_l
	\\ & \qquad + 2\tilde C_s(1+T)^2 \sum_{l\in \mathcal{K}}\sqrt{\mathrm{meas}(B_l)}\Dt_l\\
	& \leq (1+\tilde C_s)(1+T)^2\big(1+\|\alpha'\|_{\infty}\sqrt{G_{\infty}}\big)G_\infty T \Dt  
	\\ & \qquad + 2\tilde C_s(1+T)^2  \sqrt{\sum_{l\in \mathcal{K}}\mathrm{meas}(B_l)} \sqrt{\sum_{l\in \mathcal{K}}\Dt_l^2}\\
& \leq (1+\tilde C_s)(1+T)^2\big(1+\|\alpha'\|_{\infty}\sqrt{G_{\infty}}\big)G_\infty T \Dt  
\\ & \qquad +2\tilde C_s(1+T)^2 \sqrt{\mathrm{meas} (\Omega_d(0)\cap \Omega_c(T))}\sqrt{T}\sqrt{\Dt}\\
& \leq (1+\tilde C_s)(1+T)^2\big(1+\|\alpha'\|_{\infty}\sqrt{G_{\infty}}\big)G_\infty T \Dt   
\\ & \qquad + 2\tilde C_s (1+T)^3 \sqrt{G_\infty} \sqrt{T}\sqrt{\Dt}.
\end{align*}
Here we used that $B_m\cap B_k=\emptyset$ for $m\neq k$ and that $\bigcup_{l =0}^{N-1}B_l= \Omega_d(0) \cap \Omega_c(T)$ combined with \eqref{eq:boundMeasure}. 
\end{proof}

\begin{remark}
Let $X=L\circ P_{\Dx}((u,\mu,\nu)) \in \F_0^{\alpha,0}$ for some $(u,\mu,\nu)\in \D_0^\alpha$. Combining \eqref{eq:inductionArg}, Lemma~\ref{thm:localError}, and \eqref{eq:dt} we conclude that 
\begin{equation}\label{est:errorit}
d_s(S_t(X),S_{\Dx,t}(X))\leq \mathcal{O}(\Dt^{\frac12})=\mathcal{O}(\Dx^\frac14).
\end{equation}
Inspecting the proof of Lemma~\ref{lem:locerror} and Lemma~\ref{thm:localError} reveals that the leading error term is the second term in \eqref{est:gn}, which relies on \eqref{error:ds:alpha}. However, the latter inequality contains no information about which local error, due to approximating $S_{\Dx,t}$ by $S_t$, is responsible for this leading order. 

Pick $j\in \mathbb{Z}$ such that $\tau_{j+\frac12}\in (\tau_k^*, \tau_{k+1}^*)$, then one has
\begin{align*}
		|\tilde{y}_j(\tau_{j+\frac{1}{2}}) - y_j^{M_{\mathrm{it}}^k-1}(\tau_{k+1}^*)| &\leq |\tilde{y}_j(\tau_{j+\frac{1}{2}}) - y_j(\tau_{j+\frac{1}{2}})| + |y_j(\tau_{j+\frac{1}{2}}) - y_j(\tau_{k+1}^*)|
		\\ & \qquad + |y_j^{M_{\mathrm{it}}^k}(\tau_{k+1}^*) - y_{j}^{M_{\mathrm{it}}^k-1}(\tau_{k+1}^*)|. 
	\end{align*}
The first term is of order  $\mathcal{O}(\Dx)$ by \eqref{est:yn}, whereas the third term is of order $\mathcal{O}(\Dx^2)$ by \eqref{eq:dt} and \eqref{eq:stopCond}.  For the second term, on the other hand, one obtains, due to \eqref{eq:upper_boundU}, 
	\begin{equation*}
		 |y_j(\tau_{j+\frac{1}{2}}) - y_j(\tau_{k+1}^*)| \leq \int_{\tau_{j+\frac{1}{2}}}^{\tau_{k+1}^*}|U_j(s)|ds =\mathcal{O}(\Dt)= \mathcal{O}(\sqrt{\Dx}). 
	\end{equation*}
	In other words, the dominating local error is introduced by delaying the computation of the $\alpha$-values, i.e., we compute them at $\tau_{k+1}^*$ instead of $\tau_{j+\frac{1}{2}}$.
\end{remark}

Recalling \eqref{eq:dt},  we have thus shown the following result in this section.
\begin{theorem}\label{thm:convLagr}
	Given $(u, \mu, \nu)\in \D_0^\alpha$ and $t \in [0, T]$, let $X=L((u, \mu,\nu))$ and  $X_{\Dx}=L\circ 	P_{\Dx}((u,\mu,\nu))$, then 
\begin{equation*}
	d(X(t), X_{\Dx}(t))= d(S_t(X), S_{\Dx,t}(X_{\Dx}))\to 0 \quad \text{ as }\quad \Delta x\to 0.
\end{equation*}
\end{theorem}

\subsection{Convergene in Eulerian coordinates}

Finally, we can consider in what sense the mapping $M$ from Definition~\ref{def:MapM} transports the convergence in Theorem~\ref{thm:convLagr} into Eulerian coordinates. 

\begin{theorem}\label{thm:EulerianConvergence}
	Given $(u_0, \mu_0, \nu_0) \in \D_0^{\alpha}$, introduce $(u, \mu, \nu)(t) = T_t((u_0, \mu_0, \nu_0))$ and $(u_{\Dx}, \mu_{\Dx}, \nu_{\Dx})(t) = T_{\Dx, t} ((u_0, \mu_0, \nu_0))$ for $t \in [0, T]$. Then as $\Dx \rightarrow 0$, we have\footnote{Here $\eta_{\Dx} \implies \eta$ indicates that $\int_{\R}\phi d\eta_{\Dx} \rightarrow \int_{\R}\phi d\eta$ for all $\phi \in C_b(\R) = C(\R) \cap L^{\infty}(\R)$, whereas $\eta_{\Dx} \overset{\ast}{\rightharpoonup} \eta$ means that $\int_{\R}\phi d\eta_{\Dx} \rightarrow \int_{\R}\phi d\eta$ for all $\phi \in C_c(\R)$.}
	\begin{subequations}
	\begin{align}
		u_{\Dx}(t) &\rightarrow u(t) \hspace{0.65cm} \mathrm{in } L^{\infty}(\R), \label{eq:conv_u}\\
		u_{\Dx, x}(t) &\rightharpoonup u_x(t)  \hspace{0.5cm} \mathrm{in } L^2(\R), \label{eq:conv_ux}  \\
		\nu_{\Dx}(t) &\hspace{-0.05cm}\implies \nu(t), \label{eq:conv_nu} \\ 
		\left(y_{\Dx}(t)\right)_{\#}\left(g_p(X_{\Dx}(t))d\xi\right) &\overset{\ast}{\rightharpoonup} (y(t))_{\#} \left(g_p(X(t))d\xi\right)\!,   \quad p \in \{1, 2, 3\}. \label{eq:conv_g123} 
		\end{align}
	\end{subequations}
\end{theorem}

\begin{proof}
Since Theorem~\ref{thm:convLagr} implies that $y_{\Dx}(t) \rightarrow y(t)$ in $L^{\infty}(\R)$, $U_{\Dx}(t) \rightarrow U(t)$ in $L^{\infty}(\R)$ and $H_{\Dx, \xi}(t) \rightarrow H_{\xi}$ in $L^1(\R)$, one can prove \eqref{eq:conv_u}--\eqref{eq:conv_nu} by following the proof of \cite[Thm. 4.12]{AlphaAlgorithm}. 

To prove \eqref{eq:conv_g123} for $p=1$, pick any $\phi \in C_c(\R) $. Then 
\begin{align}\label{eq:weak_g}
	\int_{\R} \phi d\big(y(t)_{\#} \big(g_1(X(t))&d\xi\big)\big) - \int_{\R} \phi d\big(y_{\Dx}(t)_{\#}\big(g_1(X_{\Dx}(t))d\xi\big) \big) \nonumber
	\\ &\hspace{-0.4cm}= \int_{\R}\phi(y(t, \xi))g_1(X(t))(\xi)d\xi - \int_{\R}\phi(y_{\Dx}(t, \xi))g_1(X_{\Dx}(t))(\xi)d\xi \nonumber
	\\ &\hspace{-0.4cm}= \int_{\R}\left(\phi(y(t, \xi)) - \phi(y_{\Dx}(t, \xi)) \right)\!g_1(X(t))( \xi)d\xi \nonumber
	\\ &\hspace{-0.4cm} \qquad +  \int_{\R} \phi(y_{\Dx}(t, \xi)) \left(g_1(X(t)) - g_1(X_{\Dx}(t))\right)\!( \xi)d\xi. 
\end{align}
The fact that $y_{\Dx}(t) \rightarrow y(t)$ in $L^{\infty}(\R)$ by Theorem~\ref{thm:convLagr} together with $y-\id \in L^{\infty}(\R)$, implies that $\phi(y_{\Dx}(t, \cdot))$ has compact support for all $\Dx\leq 1$ and in particular that there exists a compact set $K$ such that the support of $\phi(y_{\Dx}(t, \cdot))$ is contained in $K$ for all $\Dx\leq 1$. As a consequence, it follows by the dominated convergence theorem that $\phi(y_{\Dx}(t)) \rightarrow \phi(y(t))$ in $L^2(\R)$ as $\Dx \rightarrow 0$. Furthermore, from \eqref{eq:g} it is apparent that $g_1(X(t)) \in L^2(\R)$. Hence 
\begin{align*}
	\left|\int_{\R}\left(\phi(y(t, \xi)) - \phi(y_{\Dx}(t, \xi)) \right)\!g_1(X(t))( \xi)d\xi\right| &\leq \|\phi(y(t)) - \phi(y_{\Dx}(t))\|_2 \|g_1(X(t))\|_2,
\end{align*}
and the first term on the right-hand side of \eqref{eq:weak_g} tends to $0$ as $\Dx\to 0$. 
Similarly, since $g_1(X_{\Dx}(t)) \rightarrow g_1(X(t))$ in $L^2(\R)$ by Theorem~\ref{thm:convLagr}, we can apply the Cauchy--Schwarz inequality to the other term in \eqref{eq:weak_g} and thereby establish \eqref{eq:conv_g123} for $p=1$. The same argument can be used to show \eqref{eq:conv_g123} for $p \in \{2, 3\}$. 
\end{proof}

\section{Numerical results}\label{sec:NumericalExperiments}
We complement the previous section by investigating two numerical examples for $T=3$. The first example displays a multipeakon solution, i.e., a solution for which the wave profile $u(t,\cdot)$ is piecewise linear for all $t\geq 0$, with three distinct breaking times contained within an interval of length $\frac{2}{9}$. Thus for coarse grids satisfying $\Dt \geq \frac{2}{9}$, a minimal time step is taken. This allows us to numerically examine how minimal time steps affect the approximation accuracy. 

Secondly, we continue our investigation of the cusp data from Example~\ref{ex:cuspData}, for which an infinitesimal amount of energy concentrates at each $t \in [0, 3]$ for the exact solution. Here the sequence $\{\tau_k^*\}_{k=0}^N$ extracted according to Algorithm~\ref{alg:pseudocode} really comes into play, as well as the fact that the number of iterates is uniformly bounded by $3$, see Section~\ref{sec:numericalImp}, which combined reduce the computation time significantly. 

\begin{example}[Multipeakon data]\label{ex:MultipeakonExample}Consider the Cauchy problem \eqref{eq:Hunter-Saxton} with \vspace{-0.1cm}
\begin{align*}
	u_0(x) &= \begin{cases}
		3, & x \leq 0, \\
		3-x & 0 < x \leq 1, \\
		2, & 1 < x \leq \frac{400}{361}, \\
		\frac{58}{19} - \frac{19}{20}x, & \frac{400}{361} \leq x \leq \frac{800}{361}, \\
		 \frac{18}{19}, & \frac{800}{361} < x \leq \frac{200}{81}, \\
		\frac{542}{171} - \frac{9}{10}x, &  \frac{200}{81} < x \leq \frac{100}{27}, \\
		-\frac{28}{171}, & \frac{100}{27} < x, 
		\end{cases}  \\
	d\mu_0 &= d\nu_0 = u_{0, x}^2dx, \\
	\alpha(x) &= \begin{cases}
		0, & x \leq \frac{1434}{361}, \\
		-\frac{478}{127} + \frac{361}{381}x, & \frac{1434}{361} < x \leq \frac{6879}{1444}, \\
		\frac{1}{1705}(361x -441), &\frac{6879}{1444} < x \leq 5, \\
		\frac{4}{5}, &5 < x. 
		\end{cases}
\end{align*}
The exact solution $(u, F)(t)$ experiences wave breaking at $\tau_1= 2$, $\tau_2=\frac{40}{19}$ and $\tau_3=\frac{20}{9}$, and $F(t)$ is discontinuous at each wave breaking occurrence.

In Figure~\ref{fig:comp_multipeakon}, we compare $(u_{\Dx}, F_{\Dx})$ to $(u, F)$ at  $t=0$, $\frac{20}{9}$, and  $3$, for $\Dx=10^{-1}$ and $10^{-4}$. As we approximate $(u_0, F_0)$ by $(u_{\Dx}, F_{\Dx})(0)$, there is in general a difference in the sets of points for which wave breaking occurs and the times for which it takes place. This leads to a discrepancy between $F_{\Dx, \infty}(\frac{20}{9})$ and $F_{\infty}(\frac{20}{9})$.  
Nevertheless, the algorithm captures the wave breaking phenomena for this example well, as the right plot in Figure~\ref{fig:convrateEx1} shows that the wave breaking time $\tau_2=\frac{20}{9}$ is just slightly delayed for the numerical solution with $\Dx=10^{-4}$. Additionally, for comparison, we have also plotted the conservative ($\alpha=0$) and dissipative ($\alpha=1$) solutions in Figure~\ref{fig:comp_multipeakon}. Observe that the dissipative solution is equal to a constant for $t \geq \frac{20}{9}$ since all its energy has been dissipated. 

\begin{figure}
	\includegraphics{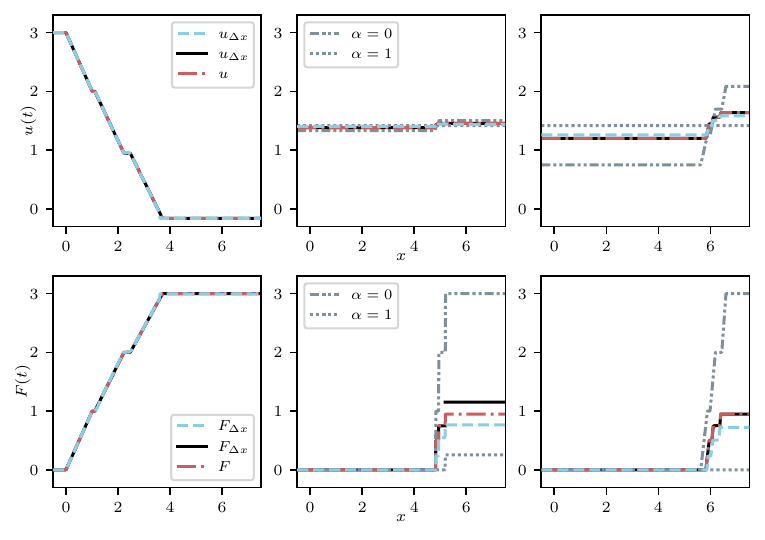}
	\captionsetup{width=.975\linewidth}
	\caption{Comparison of $u$ (top row, red dotted line) and $F$ (bottom row, red dotted line) with $u_{\Dx}$ (upper row) and $F_{\Dx}$ (lower row) for $\Dx = 10^{-1}$ (dashed blue line) and $\Dx=10^{-4}$ (solid black line) for Example~\ref{ex:MultipeakonExample}. The solutions are compared from left to right, at the times $t=0$, $\frac{20}{9}$, and $3$. Moreover, also the conservative ($\alpha=0$, gray dashdotted) and dissipative ($\alpha=1$, gray dotted) solutions are displayed at $t=\frac{20}{9}$ and $3$.}
	\label{fig:comp_multipeakon}
\end{figure}

Figure~\ref{fig:convrateEx1} also displays the numerically computed approximation errors (two leftmost figures). A careful inspection of Figure~\ref{fig:convrateEx1} reveals that a jump in the approximation errors is located between $\Dx=5\cdot 10^{-2}$ and $\Dx=10^{-2}$. This is a consequence of the minimal time step being enforced, which influences the number of grid cells which contribute to the local error. To be more precise, the numerical solution will have at least $3$ and at most $9$ distinct breaking times of which $3$ coincide with $\tau_1$, $\tau_2$, and $\tau_3$ and the remaining ones are attained at one grid cell each. Since \begin{equation*}
	\Dt \approx 1.678 \sqrt{\Dx},
\end{equation*}
by \eqref{eq:dt},
which implies that $\Dt \geq \frac{2}{9}$ whenever $\Dx \geq 1.75 \cdot 10^{-2}$, the number of grid cells contributing to the local error differs significantly for $\Dx=5\cdot 10^{-2}$ and $\Dx=10^{-2}$. 

\begin{figure}
	\includegraphics{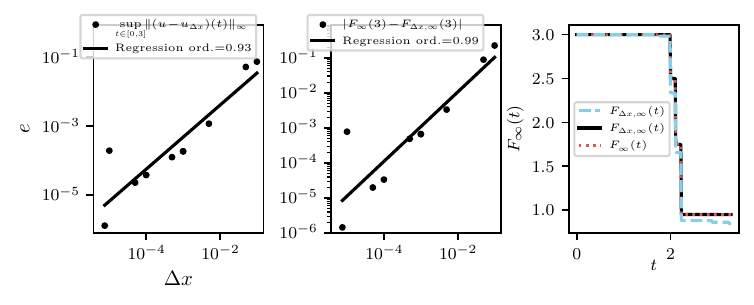}
	\captionsetup{width=.975\linewidth}
	\caption{In the two leftmost figures,  the errors $ \sup_{t \in [0, 3]} \|u(t) - u_{\Dx}(t)\|_{\infty}$ and $|F_{\infty}(3) - F_{\Dx, \infty}(3)|$ are plotted as functions of $\Dx$ for Example~\ref{ex:MultipeakonExample}, while the time evolution of the total numerical energy $F_{\Dx, \infty}(t)$ for $\Dx=5\cdot10^{-1}$ (dashed blue) and $\Dx=10^{-4}$ (solid black) is compared with $F_{\infty}(t)$ (dotted red) in the right figure. }
	\label{fig:convrateEx1}
\end{figure}
\end{example}

\begin{example}[Cusp data]\label{ex:CuspExample}
Consider the initial data from Example~\ref{ex:cuspData} again. The exact solution is not available for comparison, and to compensate, we compute a numerical solution with $\Dx =10^{-5}$, denoted $(u_{\mathrm{ref}}, F_{\mathrm{ref}})$, which is used as a reference solution.

\begin{figure}
	\includegraphics{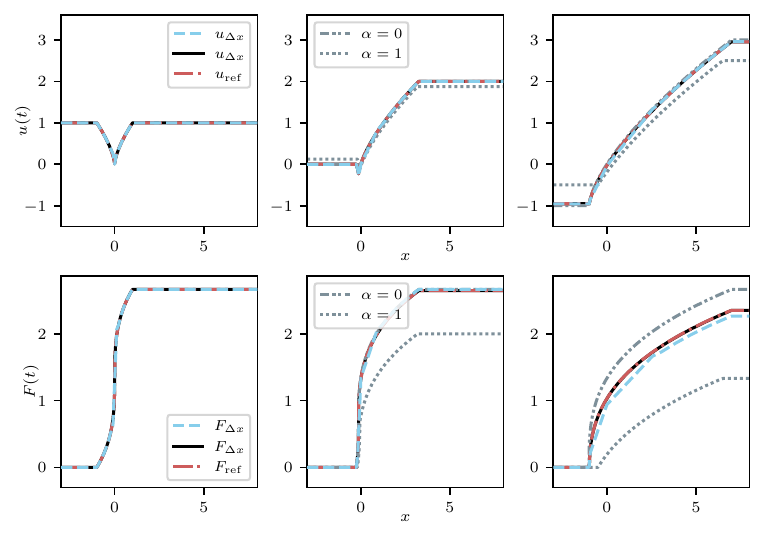}
	\captionsetup{width=.975\linewidth}
	\vspace{-0.4cm}
	\caption{A comparison of $u_{\mathrm{ref}}$ (top row, red dotted line) and $F_{\mathrm{ref}}$ (bottom row, red dotted line) with that of $u_{\Dx}$ (top row) and $F_{\Dx}$ (bottom row) for $\Dx=10^{-1}$ (blue dashed) and $\Dx = 10^{-4}$ (black solid) for Example \ref{ex:CuspExample}. The solutions are compared from left to right, at the times $t =0$, $\frac{3}{2}$, and $3$, with $\beta =\frac{19}{20}$. Moreover, also the conservative ($\alpha=0$, gray dashdotted) and dissipative ($\alpha=1$, gray dotted) solutions are displayed at $t=\frac{3}{2}$ and $3$.} 
	\label{fig:comparison_cusp}
\end{figure}

In Figure~\ref{fig:comparison_cusp}, $(u_{\Dx}, F_{\Dx})$  is compared to $(u_{\mathrm{ref}}, F_{\mathrm{ref}})$ at $t=0$, $\frac{3}{2}$, and $3$ for $\Dx= 10^{-1}$ and $10^{-4}$, where we also have included the exact conservative ($\alpha=0$) and exact dissipative ($\alpha=1$) solutions at $t=\frac{3}{2}$ and $3$ for comparison. The right snapshot in Figure~\ref{fig:convrateEx2} depicts the time evolution of the numerical total energy $F_{\Dx, \infty}(t)$. Again, approximating $(u_0, F_0)$ by $(u_{\Dx}, F_{\Dx})(0)$ introduces a difference in the sets of points for which wave breaking occurs and the times for which it takes place. Nevertheless, the majority of $\{\tilde{\tau}_j\}$ (sequence of distinct numerical breaking times) is concentrated inside the time interval $[0,3]$. On the other hand, $F_{\Dx, \infty}(t)$ for $\Dx=10^{-1}$ remains constant for short amounts of time and then drops abruptly, creating a zig-zag pattern. This behavior can be observed for any $F_{\Dx,\infty}(t)$ with $\Dx>0$ and is a consequence of wave breaking occurring continuously over $[0,3]$ for the exact solution, while it occurs discretely, i.e.,  at the distinct breaking times $\tilde{\tau}_j$ for the numerical solution.

Figure~\ref{fig:convrateEx2} also depicts the numerically computed approximation errors, whereas Table~\ref{tab:RunningTimes} summarizes the running times for $T=3$ and $\Dx=10^{-2}$, $10^{-3}$, $10^{-4}$, and $10^{-5}$. These were computed using Python 3.8 on a 2020 Macbook Pro with  Apple M1 chip and an 8-core CPU. 
Here evolving with a minimal time step refers to computing the solution using $T_{\Dx, t}$, defined in Definition~\ref{def:numSol}, while without minimal times steps means to compute the solution exactly between all successive numerical breaking times $\{\tilde{\tau}_j\}_{j \in \mathbb{N}}$ as outlined in Remark~\ref{rem:evolveExact}. One observes a substantial reduction in
 computation time for $\Dx = 10^{-4}$ and $\Dx=10^{-5}$, which can be explained as follows: For $\Dx=10^{-4}$, the number of sorted distinct breaking times, with value less than or equal to $T=3$,  is about $10^4$, whilst the extracted subsequence $\{\tau_k^*\}_{k=0}^{N}$ has length $169$, which is very close to the predicted value in Example~\ref{ex:cuspData}. At most three iterations and hence three numerical integrations are performed for each minimal time step, which leads to around $500$ numerical integrations in total, in contrast to $10^4$ which are required when evolving without minimal time steps. A similar analysis extends to $\Dx=10^{-5}$, but we ran into memory issues when trying to solve without minimal time steps.

\begin{figure}[H]
	\includegraphics{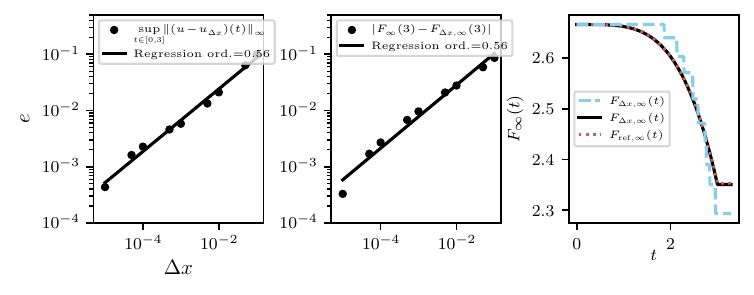}
	\captionsetup{width=.975\linewidth}
	\caption{The errors $\sup_{t \in [0, 3]} \|u_{\mathrm{ref}}(t) - u_{\Dx}(t)\|_{\infty}$ and $|F_{\mathrm{ref}, \infty}(3) - F_{\Dx, \infty}(3)|$ are plotted as functions of $\Dx$ for Example~\ref{ex:CuspExample} in the two left figures, while the time evolution of the total numerical energy $F_{\Dx, \infty}(t)$ for $\Dx=10^{-1}$ (dashed blue) and $\Dx=10^{-4}$ (solid black) is compared to $F_{\mathrm{ref}, \infty}(t)$ (dotted red) in the right plot. }
	\label{fig:convrateEx2}
\end{figure}

\begin{table}[H]
\begin{tabular}{ c|cccc } 
 \hline
 $\Dx$ & $10^{-2}$& $10^{-3}$ & $10^{-4}$ & $10^{-5}$ \\ 
 \hline
 With minimal & 0.0110 & 0.2655 & 8.4128 & 264.8032\\ 
 \hline
 Without minimal & 0.0378 & 3.5321 & 367.4902 & N/A\\ 
 \hline
\end{tabular}
\vspace{-0.3cm}
\caption{Running times in seconds for Example~\ref{ex:CuspExample} when evolving with and without minimal time steps.}
\label{tab:RunningTimes}
\end{table}

\end{example}

\bibliographystyle{plain}

\appendix

\section{Exact solution to the multipeakon example}
Consider the initial data 
\begin{subequations}\label{eq:exInitial2}
\begin{align}
		u_0(x) &= \begin{cases}
		3, & x \leq 0, \\
		3-x & 0 < x \leq 1, \\
		2, & 1 < x \leq \frac{400}{361}, \\
		\frac{58}{19} - \frac{19}{20}x, & \frac{400}{361} \leq x \leq \frac{800}{361}, \\
		 \frac{18}{19}, & \frac{800}{361} < x \leq \frac{200}{81}, \\
		\frac{542}{171} - \frac{9}{10}x, &  \frac{200}{81} < x \leq \frac{100}{27}, \\
		-\frac{28}{171}, & \frac{100}{27} < x, 
		\end{cases} \\ 
		F_0(x) &= G_0(x) =  \begin{cases}
		0, & x \leq 0, \\
		x, & 0 < x \leq 1, \\
		1, & 1 < x \leq \frac{400}{361}, \\
		\frac{361}{400}x, & \frac{400}{361} < x \leq \frac{800}{361}, \\
		2,  &  \frac{800}{361} < x \leq \frac{200}{81}, \\
		\frac{81}{100}x, &  \frac{200}{81} < x \leq \frac{100}{27}, \\
		3, & \frac{100}{27} < x, 
		\end{cases} 
	\end{align}
	\end{subequations}
with $\alpha(x)$ given by 
	\begin{equation} \label{eq:exInitial}
		\alpha(x) = \begin{cases}
		0, & x \leq \frac{1434}{361}, \\
		-\frac{478}{127} + \frac{361}{381}x, & \frac{1434}{361} < x \leq \frac{6879}{1444}, \\
		\frac{361}{1705}x -\frac{441}{1705}, &\frac{6879}{1444} < x \leq 5, \\
		\frac{4}{5}, &5 < x. 
		\end{cases}
	\end{equation}
	
After applying $L$ from Definition~\ref{def:MapL}, we obtain
\begin{align*}
		y_0(\xi) &= \begin{cases}
		\xi, & \xi \leq 0, \\
		\frac{1}{2}\xi, &0 < \xi \leq 2, \\
		\xi - 1, & 2 < \xi \leq \frac{761}{361}, \\
		\frac{400}{761}\xi, &  \frac{761}{361} < \xi \leq \frac{1522}{361}, \\
		\xi - 2, & \frac{1522}{361} < \xi \leq \frac{362}{81}, \\
		\frac{100}{181} \xi, & \frac{362}{81} < \xi \leq \frac{181}{27}, \\
		\xi - 3, & \frac{181}{27} < \xi, 
		\end{cases}\\
		U_0(\xi) &= \begin{cases}
		3, & \xi \leq 0, \\
		3 - \frac{1}{2}\xi, & 0 < \xi \leq 2, \\ 
		\!\!2, & 2 < \xi \leq \frac{761}{361}, \\ 
		\!\! \frac{58}{19} - \frac{380}{761}\xi, &  \frac{761}{361} < \xi \leq \frac{1522}{361}, \\
		\frac{18}{19}, & \frac{1522}{361} < \xi \leq \frac{362}{81}, \\
		\frac{542}{171} - \frac{90}{181}\xi, &  \frac{362}{81} < \xi \leq \frac{181}{27}, \\
		-\frac{28}{171}, &  \frac{181}{27} < \xi, 
		\end{cases} \\
		V_0(\xi) = H_0(\xi) &= \begin{cases}
		0, & \xi \leq 0, \\
		\frac{1}{2}\xi, &0 < \xi \leq 2, \\
		1, & 2 < \xi \leq \frac{761}{361}, \\
		\frac{361}{761}\xi, & \frac{761}{361} < \xi \leq \frac{1522}{361}, \\
		2, & \frac{1522}{361} < \xi \leq \frac{362}{81}, \\
		\frac{81}{181}\xi, & \frac{362}{81} < \xi \leq \frac{181}{27}, \\
		3, &  \frac{181}{27} < \xi, 
		\end{cases}
	\end{align*}
		and from \eqref{eq:waveBreakingFunc} we have 
	\begin{align}\label{tauex}
		\tau(\xi) &= \begin{cases}
		2, & \xi \in [0, 2], \\
		\frac{40}{19}, & \xi \in [\frac{761}{361}, \frac{1522}{361}], \\
		\frac{20}{9}, & \xi \in [\frac{362}{81}, \frac{181}{27}], \\
		\infty, & \text{otherwise}. 
		\end{cases}
	\end{align}
	Thus, no wave breaking occurs for $t < 2$. Hence, solving \eqref{eq:intLagrSystem} for $t \in [0, 2)$ yields
\begin{align} \label{eq:LagrSol1}
		y(t, \xi) &= \begin{cases}
		\xi - \frac{3}{8}t^2 + 3t, & \xi \leq 0, \\
		\frac{1}{8}(t-2)^2\xi - \frac{3}{8}t^2 + 3t, & 0 < \xi \leq 2, \\
		\xi -\frac{1}{8}t^2 + 2t - 1, & 2 < \xi \leq \frac{761}{361}, \\
		\frac{361}{3044}(t-\frac{40}{19})^2\xi - \frac{3}{8}t^2 + \frac{58}{19}t, & \frac{761}{361}< \xi \leq \frac{1522}{361}, \\
		\xi + \frac{1}{8}t^2 + \frac{18}{19}t - 2, & \frac{1522}{361} < \xi \leq \frac{362}{81}, \\
		\frac{81}{724}(t-\frac{20}{9})^2\xi - \frac{3}{8}t^2 + \frac{542}{171}t, &  \frac{362}{81} < \xi \leq \frac{181}{27}, \\
		\xi  + \frac{3}{8}t^2 - \frac{28}{171}t - 3, &  \frac{181}{27} \leq \xi, 
		\end{cases}  \\
		U(t, \xi) &= \begin{cases}
		- \frac{3}{4}t + 3, & \xi \leq 0, \\
		\frac{1}{4}(t-2)\xi - \frac{3}{4}t + 3, & 0 < \xi \leq 2, \\
		-\frac{1}{4}t + 2, & 2 < \xi  \leq \frac{761}{361}, \\
		\frac{361}{1522}(t - \frac{40}{19})\xi - \frac{3}{4}t + \frac{58}{19}, &  \frac{761}{361} < \xi \leq \frac{1522}{361}, \\
		\frac{1}{4}t + \frac{18}{19}, &  \frac{1522}{361} < \xi \leq \frac{362}{81}, \\
		\frac{81}{362}(t-\frac{20}{9})\xi - \frac{3}{4}t + \frac{542}{171}, &  \frac{362}{81} < \xi \leq \frac{181}{27}, \\
		\frac{3}{4}t -\frac{28}{171}, & \frac{181}{27} < \xi, 
		\end{cases} \nonumber\\
		V(t, \xi) &= H(t,\xi)=V_0(\xi) \nonumber. 
	\end{align}
	Note that the function $y(t, \cdot)$ is strictly increasing for each $t\in [0,2)$, hence invertible. Thus, Definition~\ref{def:MapM} implies that for $t\in[0,2)$ the pair $(u,F)(t,\cdot)$ is given by 
	\begin{subequations}\label{eq:exEul1}
		\begin{align}
		u(t, x) &= \begin{cases}
		-\frac{3}{4}t + 3, & x \leq -\frac{3}{8}t^2 + 3t, \\
		\frac{4x -3t - 12}{2(t-2)}, & -\frac{3}{8}t^2 + 3t < x \leq -\frac{1}{8}t^2 + 2t +1, \\
		-\frac{1}{4}t + 2, & -\frac{1}{8}t^2 + 2t +1 < x \leq -\frac{1}{8}t^2 + 2t + \frac{400}{361}, \\
		\frac{2(361x - 266t - 1160)}{19(19t-40)}, & -\frac{1}{8}t^2 + 2t + \frac{400}{361} < x \leq \frac{1}{8}t^2 + \frac{18}{19}t + \frac{800}{361}, \\
		\frac{1}{4}t + \frac{18}{19}, & \frac{1}{8}t^2 + \frac{18}{19}t + \frac{800}{361} < x \leq \frac{1}{8}t^2 + \frac{18}{19}t + \frac{200}{81}, \\
		\frac{3078x - 2313t - 10840}{171(9t-20)}, & \frac{1}{8}t^2 + \frac{18}{19}t + \frac{200}{81} < x \leq \frac{3}{8}t^2 - \frac{28}{171}t + \frac{100}{27}, \\
		\frac{3}{4}t - \frac{28}{171}, & \frac{3}{8}t^2 - \frac{28}{171}t + \frac{100}{27} < x, 
		\end{cases} \\
		F(t, x) &= \begin{cases}
		0, & x \leq -\frac{3}{8}t^2 + 3t, \\
		\frac{4}{(t-2)^2}(x+\frac{3}{8}t^2 -3t), & -\frac{3}{8}t^2 + 3t < x \leq -\frac{1}{8}t^2 + 2t +1, \\
		1, & -\frac{1}{8}t^2 + 2t +1 < x \leq -\frac{1}{8}t^2 + 2t + \frac{400}{361}, \\
		\frac{4}{(t-\frac{40}{19})^2}(x + \frac{3}{8}t^2 - \frac{58}{19}t), & -\frac{1}{8}t^2 + 2t + \frac{400}{361} < x \leq \frac{1}{8}t^2 + \frac{18}{19}t + \frac{800}{361}, \\
		2, & \frac{1}{8}t^2 + \frac{18}{19}t + \frac{800}{361} < x \leq \frac{1}{8}t^2 + \frac{18}{19}t + \frac{200}{81}, \\
		\frac{4}{(t-\frac{20}{9})^2}(x + \frac{3}{8}t^2 - \frac{542}{171}t), & \frac{1}{8}t^2 + \frac{18}{19}t + \frac{200}{81} < x \leq \frac{3}{8}t^2 - \frac{28}{171}t + \frac{100}{27}, \\
		3, & \frac{3}{8}t^2 - \frac{28}{171}t + \frac{100}{27} < x.  
		\end{cases}
	\end{align}
	\end{subequations}
	
	At $t=2$ wave breaking takes place for all $\xi\in [0,2]$, cf. \eqref{tauex}, and in particular, $\lim_{t \uparrow 2}y(t, \xi) = \frac{9}{2}$ for all $\xi \in [0, 2]$ by \eqref{eq:LagrSol1}, which yields $\alpha(y(\tau(\xi), \xi)) =\frac{1}{2}$. Thus, after computing $V(2, \cdot)$ and using \eqref{eq:intLagrSystem}, we find that the solution in Lagrangian coordinates for $ t\in [2, \frac{40}{19})$ reads  
\begin{align}\label{eq:LagrSol2}
	y(t, \xi) &= \begin{cases}
	\xi - \frac{5}{16}t^2 + \frac{11}{4}t + \frac{1}{4}, & \xi \leq 0, \\
	\frac{1}{16}(t-2)^2\xi  - \frac{5}{16}t^2 + \frac{11}{4}t + \frac{1}{4}, & 0 < \xi \leq 2, \\
	\xi - \frac{3}{16}t^2 + \frac{9}{4}t - \frac{5}{4}, &2 < \xi \leq \frac{761}{361}, \\
	\frac{361}{3044}(t-\frac{40}{19})^2\xi - \frac{7}{16}t^2 + \frac{251}{76}t -\frac{1}{4}, & \frac{761}{361} < \xi \leq\frac{1522}{361}, \\
	\xi + \frac{1}{16}t^2 + \frac{91}{76}t -\frac{9}{4}, & \frac{1522}{361} < \xi \leq \frac{362}{81}, \\
	\frac{81}{724}(t-\frac{20}{9})^2\xi - \frac{7}{16}t^2 + (\frac{371}{171}+\frac{5}{4})t - \frac{1}{4}, &\frac{362}{81} < \xi \leq \frac{181}{27}, \\
	\xi + \frac{5}{16}t^2 + \frac{59}{684}t -\frac{13}{4}, &  \frac{181}{27} < \xi, 
	\end{cases} \\
	U(t, \xi) &= \begin{cases}
	-\frac{5}{8}t + \frac{11}{4}, & \xi \leq 0, \\
	\frac{1}{8}(t-2)\xi -\frac{5}{8}t + \frac{11}{4}, &0 < \xi \leq 2, \\
	-\frac{3}{8}t + \frac{9}{4}, & 2 < \xi \leq \frac{761}{361}, \\
	\frac{361}{1522}(t-\frac{40}{19})\xi-\frac{7}{8}t + \frac{251}{76}, & \frac{761}{361} < \xi 		\leq\frac{1522}{361}, \\
	\frac{1}{8}t + \frac{91}{76}, & \frac{1522}{361} < \xi \leq \frac{362}{81}, \\
	\frac{81}{362}(t-\frac{20}{9})\xi -\frac{7}{8}t + \frac{371}{171} + \frac{5}{4}, &\frac{362}{81} < \xi \leq \frac{181}{27}, \\
	\frac{5}{8}t + \frac{59}{684}, & \frac{181}{27} < \xi, 
	\end{cases} \nonumber \\ 
		V(t, \xi) &= \begin{cases}
		0, & \xi \leq 0, \\
		\frac{1}{4}\xi, & 0 < \xi \leq 2, \\
		\frac{1}{2}, & 2 < \xi \leq \frac{761}{361}, \\
		\frac{361}{761}\xi - \frac{1}{2}, & \frac{761}{361} < \xi \leq \frac{1522}{361}, \\
		\frac{3}{2}, & \frac{1522}{361} < \xi \leq \frac{362}{81}, \\
		\frac{81}{181}\xi - \frac{1}{2}, & \frac{362}{81} < \xi \leq \frac{181}{27}, \\
		\frac{5}{2}, & \frac{181}{27} < \xi,
		\end{cases} \nonumber \\
		H(t,\xi)&= V_0(\xi) \nonumber.
	\end{align}
	
	As before, the function $y(t, \cdot)$ is strictly increasing and therefore invertible for all $t\in (2, \frac{40}{19})$. Hence, it follows by Definition \eqref{def:MapM} that the solution in Eulerian coordinates for $t\in [2, \frac{40}{19})$ takes the form 
	\begin{subequations}\label{eq:exEul2}
	\begin{align}
		u(t, x) &= \begin{cases}
		-\frac{5}{8}t + \frac{11}{4} & x \leq -\frac{5}{16}t^2 + \frac{11}{4}t + \frac{1}{4}, \\
		\frac{4x -3t - 12}{2(t-2)} &  -\frac{5}{16}t^2 + \frac{11}{4}t + \frac{1}{4} < x \leq -\frac{3}{16}t^2 + \frac{9}{4}t + \frac{3}{4}, \\
		-\frac{3}{8}t + \frac{9}{4} & -\frac{3}{16}t^2 + \frac{9}{4}t + \frac{3}{4} < x \leq -\frac{3}{16}t^2 + \frac{9}{4}t + \frac{1239}{1444}, \\
		\frac{2}{(t-\frac{40}{19})}(x - \frac{111}{152}t - \frac{4659}{1444}) &  -\frac{3}{16}t^2 + \frac{9}{4}t + \frac{1239}{1444} < x \leq \frac{1}{16}t^2 + \frac{91}{76}t + \frac{2839}{1444}, \\
		\frac{1}{8}t + \frac{91}{76} & \frac{1}{16}t^2 + \frac{91}{76}t + \frac{2839}{1444} < x \leq \frac{1}{16}t^2 + \frac{91}{76}t + \frac{719}{324}, \\
		\frac{2}{(t-\frac{20}{9})}(x - \frac{1009}{1368}t - \frac{21851}{6156}) & \frac{1}{16}t^2 + \frac{91}{76}t + \frac{719}{324} < x \leq \frac{5}{16}t^2 + \frac{59}{684}t + \frac{373}{108}, \\ 
		\frac{5}{8}t + \frac{59}{684} & \frac{5}{16}t^2 + \frac{59}{684}t + \frac{373}{108} <x, 
		\end{cases}\\
		F(t, x) &= \begin{cases}
		0 & x \leq -\frac{5}{16}t^2 + \frac{11}{4}t + \frac{1}{4}, \\
		\frac{4}{(t-2)^2}(x + \frac{5}{16}t^2 - \frac{11}{4}t -\frac{1}{4}) &  -\frac{5}{16}t^2 + \frac{11}{4}t + \frac{1}{4} < x \leq -\frac{3}{16}t^2 + \frac{9}{4}t + \frac{3}{4}, \\
		\frac{1}{2} & -\frac{3}{16}t^2 + \frac{9}{4}t + \frac{3}{4} < x \leq -\frac{3}{16}t^2 + \frac{9}{4}t + \frac{1239}{1444}, \\
		\frac{4}{(t-\frac{40}{19})^2}(x + \frac{5}{16}t^2 -\frac{211}{76}t - \frac{439}{1444}) &  -\frac{3}{16}t^2 + \frac{9}{4}t + \frac{1239}{1444} < x \leq \frac{1}{16}t^2 + \frac{91}{76}t + \frac{2839}{1444}, \\
		\frac{3}{2} & \frac{1}{16}t^2 + \frac{91}{76}t + \frac{2839}{1444} < x \leq \frac{1}{16}t^2 + \frac{91}{76}t + \frac{719}{324}, \\
		\frac{4}{(t-\frac{20}{9})^2}(x + \frac{5}{16}t^2 - \frac{653}{228}t - \frac{119}{324}) & \frac{1}{16}t^2 + \frac{91}{76}t + \frac{719}{324} < x \leq \frac{5}{16}t^2 + \frac{59}{684}t + \frac{373}{108}, \\ 
		\frac{5}{2} & \frac{5}{16}t^2 + \frac{59}{684}t + \frac{373}{108} <x. 
	\end{cases}
	\end{align}
	\end{subequations}
	
	At $t=\frac{40}{19}$ wave breaking occurs for all $\xi\in [\frac{761}{361}, \frac{1522}{361}]$, cf. \eqref{tauex} and in particular,  $\displaystyle \lim_{t\uparrow \frac{40}{19}} y(t, \xi) = \frac{6879}{1444} \approx 4.764$ for all  $\xi \in[\frac{761}{361}, \frac{1522}{361}]$ by  \eqref{eq:LagrSol2}, implying that  $\alpha(y(\tau(\xi), \xi)) = \frac{3}{4}$. Computing $V(\frac{40}{19}, \cdot)$ and using  \eqref{eq:intLagrSystem}, therefore reveals that the solution in Lagrangian coordinates for $t\in [\frac{40}{19}, \frac{20}{9})$ reads 
	\begin{align*}
		y(t, \xi) &= \begin{cases}
		\xi -\frac{7}{32}t^2 + \frac{179}{76}t + \frac{961}{1444}, & \xi \leq 0, \\
		\frac{1}{16}(t-2)^2\xi  -\frac{7}{32}t^2 + \frac{179}{76}t + \frac{961}{1444}, &0 < \xi \leq 2, \\
		\xi -\frac{3}{32}t^2 +\frac{141}{76}t - \frac{1205}{1444}, & 2 < \xi \leq  \frac{761}{361}, \\
		\frac{361}{12176}(t-\frac{40}{19})^2\xi -\frac{5}{32}t^2 +\frac{161}{76}t +\frac{1439}{1444}, & \frac{761}{361} \leq \xi < \frac{1522}{361}, \\
		\xi - \frac{1}{32}t^2 + \frac{121}{76}t  - \frac{3849}{1444}, & \frac{1522}{361} < \xi \leq \frac{362}{81}, \\
		\frac{81}{724}(t-\frac{20}{9})^2\xi -\frac{17}{32}t^2 + \frac{2609}{684}t -\frac{961}{1444}, &\frac{362}{81} < \xi \leq \frac{181}{27}, \\
		\xi + \frac{7}{32}t^2 + \frac{329}{684}t -\frac{5293}{1444}, & \frac{181}{27} \leq \xi, 
		\end{cases} \\ 
		U(t, \xi) &= \begin{cases}
		-\frac{7}{16}t +\frac{179}{76}, & \xi \leq 0, \\
		\frac{1}{8}(t-2)\xi -\frac{7}{16}t +\frac{179}{76}, & 0 < \xi \leq 2, \\
		-\frac{3}{16}t + \frac{141}{76}, & 2 < \xi \leq \frac{761}{361}, \\
		\frac{361}{6088}(t-\frac{40}{19})\xi - \frac{5}{16}t + \frac{161}{76}, & \frac{761}{361} < \xi \leq \frac{1522}{361}, \\
		- \frac{1}{16}t + \frac{121}{76}, &  \frac{1522}{361} < \xi \leq \frac{362}{81}, \\
		\frac{81}{362}(t-\frac{20}{9})\xi  -\frac{17}{16}t + \frac{2609}{684}, &\frac{362}{81} < \xi \leq \frac{181}{27}, \\
		\frac{7}{16}t + \frac{329}{684}, & \frac{181}{27} \leq \xi, 
		\end{cases}\\
		V(t, \xi) &= \begin{cases}
		0, & \xi \leq 0, \\ 
		\frac{1}{4}\xi, & 0 < \xi \leq 2, \\
		\frac{1}{2}, & 2 < \xi \leq \frac{761}{361}, \\
		\frac{361}{3044}\xi + \frac{1}{4}, & \frac{761}{361} \leq \xi < \frac{1522}{361}, \\
		\frac{3}{4}, & \frac{1522}{361} < \xi \leq \frac{362}{81}, \\
		\frac{81}{181}\xi - \frac{5}{4},  &\frac{362}{81} < \xi \leq \frac{181}{27}, \\
		\frac{7}{4}, & \frac{181}{27} \leq \xi,
		\end{cases}\\
		H(t,\xi)&=V_0(\xi).
	\end{align*}
	
	Using Definition~\ref{def:MapM} once more and the fact that $y(t,\cdot)$ is strictly increasing for each $t\in (\frac{40}{19}, \frac{20}{9})$, yields the following solution in Eulerian coordinates for $t\in [\frac{40}{19}, \frac{20}{9})$, 
	\begin{subequations}\label{eq:exEul3}
	\begin{align}
		u(t, x) &= \begin{cases}
		-\frac{7}{16}t + \frac{179}{76}, & x \leq x_1(t), \\
		\frac{2}{t-2}(x-\frac{225}{304}t -\frac{2181}{722}), & x_1(t) < x \leq x_2(t), \\
		-\frac{3}{16}t +\frac{141}{76}, & x_2(t) < x \leq x_3(t), \\
		\frac{2}{t-\frac{40}{19}}(x-\frac{111}{152}t -\frac{4659}{1444}), & x_3(t) < x \leq x_4(t), \\
		-\frac{1}{16}t +\frac{121}{76}, & x_4(t) < x\leq x_5(t), \\ 
		\frac{2}{t-\frac{20}{9}}\left(x-\frac{497}{684}t - \frac{417869}{116964}\right)\!, &x_5(t) < x \leq x_6(t), \\
		\frac{7}{16}t + \frac{329}{684}, & x_6(t) < x, 
		\end{cases} \\
		F(t, x) &= \begin{cases}
		0, & x \leq x_1(t), \\
		\frac{4}{(t-2)^2}(x+\frac{7}{32}t^2 -\frac{179}{76}t -\frac{961}{1444}), & x_1(t) < x \leq x_2(t), \\
	\frac{1}{2}, & x_2(t) < x \leq x_3(t), \\
	\frac{4}{(t-\frac{40}{19})^2}\left(x +\frac{7}{32}t^2 -\frac{181}{76}t -\frac{1039}{1444}\right)\!, & x_3(t) < x \leq x_4(t), \\ 
	\!\! \frac{3}{4}, & x_4(t) < x \leq x_5(t), \\
	\frac{4}{(t-\frac{20}{9})^2} \left(x + \frac{7}{32}t^2 - \frac{553}{228}t +\frac{961}{1444} -\frac{125}{81}\right)\!, & x_5(t) < x \leq x_6(t), \\
	\frac{7}{4}, & x_6(t) < x, 
	\end{cases}
	\end{align}
	\end{subequations}
	where 
	\begin{align*}
		x_1(t) &= y(t, 0) = -\frac{7}{32}t^2 + \frac{179}{76}t + \frac{961}{1444}, \\
		x_2(t) &= y(t, 2) = -\frac{3}{32}t^2 +\frac{141}{76}t +\frac{1683}{1444}, \\
		x_3(t) &= y\!\left(\!t, \frac{761}{361}\right) = -\frac{3}{32}t^2 +\frac{141}{76}t + \frac{1839}{1444}, \\
		x_4(t) &= y\!\left(\!t, \frac{1522}{361}\right) =  -\frac{1}{32}t^2 + \frac{121}{76}t + \frac{2239}{1444}, \\
		x_5(t) &=  y \! \left(\!t, \frac{362}{81}\right) = -\frac{1}{32}t^2 + \frac{121}{76}t + \frac{200}{81} - \frac{961}{1444}, \\
		x_6(t) &= y \! \left(\!t, \frac{181}{27} \right) = \frac{7}{32}t^2 + \frac{329}{684}t + \frac{100}{27} - \frac{961}{1444}.
	\end{align*}
	
	It remains to compute the solution for $t \geq \frac{20}{9}$. At $t= \frac{20}{9}$ wave breaking takes place for all  $\xi \in [\frac{362}{81}, \frac{181}{27}]$, cf. \eqref{tauex}, and it therefore holds that \small $\displaystyle \lim_{t \uparrow \frac{20}{9}}y(t, \xi) = \frac{6005}{1026}-\frac{961}{1444} \approx 5.182 $\normalsize, which yields $\alpha(y(\tau(\xi), \xi))=\frac{4}{5}$. Thus, after computing $V(\frac{20}{9}, \cdot)$ and using \eqref{eq:intLagrSystem}, we find that the solution in Lagrangian coordinates for $t\in [\frac{20}{9}, \infty)$ reads
\begin{align*}
y(t, \xi) &= \begin{cases}
	\xi - \frac{19}{160}t^2 + \frac{1307}{684}t + \frac{40}{81} + \frac{961}{1444}, & \xi \leq 0, \\
	\frac{1}{16}(t-2)^2\xi - \frac{19}{160}t^2 + \frac{1307}{684}t + \frac{40}{81} + \frac{961}{1444}, & 0 < \xi \leq 2, \\
	\xi + \frac{1}{160}t^2 + \frac{965}{684}t +\frac{40}{81} - \frac{1205}{1444}, & 2 < \xi \leq \frac{761}{361},  \\
	\frac{361}{12176}(t-\frac{40}{19})^2\xi -\frac{9}{160}t^2 + \frac{1145}{684}t + \frac{40}{81} + \frac{1439}{1444}, & \frac{761}{361} \leq \xi < \frac{1522}{361},\\
	\xi + \frac{11}{160}t^2 + \frac{785}{684}t + \frac{40}{81} - \frac{3849}{1444}, &\frac{1522}{361} < \xi \leq \frac{362}{81}, \\
	\frac{81}{3620}(t-\frac{20}{9})^2\xi - \frac{1}{32}t^2 + \frac{121}{76}t + \frac{200}{81} - \frac{961}{1444}, & \frac{362}{81} < \xi \leq \frac{181}{27}, \\
	\xi + \frac{19}{160}t^2 + \frac{211}{228}t -\frac{40}{81} -\frac{5293}{1444}, & \frac{181}{27} \leq \xi,
	\end{cases} \\
	U(t, \xi) &= \begin{cases}
	-\frac{19}{80}t + \frac{1307}{684}, & \xi \leq 0, \\
	\frac{1}{8}(t-2)\xi -\frac{19}{80}t + \frac{1307}{684}, & 0 < \xi \leq 2, \\
	\frac{1}{80}t + \frac{965}{684}, & 2 < \xi \leq \frac{761}{361},  \\
	\frac{361}{6088}(t-\frac{40}{19})\xi -\frac{9}{80}t + \frac{1145}{684}, & \frac{761}{361} \leq \xi < \frac{1522}{361},\\
	\frac{11}{80}t + \frac{785}{684}, & \frac{1522}{361} < \xi \leq \frac{362}{81}, \\
	\frac{81}{1810}(t-\frac{20}{9})\xi -\frac{1}{16}t + \frac{121}{76}, &\frac{362}{81} < \xi \leq \frac{181}{27}, \\
	\frac{19}{80}t + \frac{211}{228}, & \frac{181}{27} \leq \xi,
	\end{cases}\\
	V(t, \xi) &= \begin{cases} 
		\!0, & \xi \leq 0, \\ 
		\frac{1}{4}\xi, & 0 < \xi \leq 2, \\
		\frac{1}{2}, & 2 < \xi \leq \frac{761}{361}, \\
		\frac{361}{3044}\xi + \frac{1}{4}, & \frac{761}{361} \leq \xi < \frac{1522}{361}, \\
		\frac{3}{4}, & \frac{1522}{361} < \xi \leq \frac{362}{81}, \\
		\frac{81}{905}\xi + \frac{7}{20}, &\frac{362}{81} < \xi \leq \frac{181}{27}, \\
		\frac{19}{20}, & \frac{181}{27} \leq \xi,
		\end{cases}\\
		H(t,\xi)&=V_0(\xi).
\end{align*}

Applying \eqref{def:MapM} one last time, reveals that the solution in Eulerian coordinates takes the following form for $t\in [\frac{20}{9}, \infty)$  
\begin{subequations}\label{eq:exEul4}
\begin{align}
	u(t, x) &= \begin{cases}
	-\frac{19}{80}t + \frac{1307}{684}, & x \leq \bar{x}_1(t), \\
	\frac{2}{t-2}\left(x - \frac{9821}{13680}t -\frac{40}{81} -\frac{16741}{6498}\right), & \bar{x}_1(t) < x \leq \bar{x}_2(t), \\
	\frac{1}{80}t + \frac{965}{684}, & \bar{x}_2(t) < x \leq \bar{x}_3(t), \\
	\frac{2}{t-\frac{40}{19}}\left(x - \frac{983}{1368}t -\frac{40}{81} - \frac{35851}{12996}\right)\!, & \bar{x}_3(t) < x \leq \bar{x}_4(t), \\
	\frac{11}{80}t + \frac{785}{684}, & \bar{x}_4(t) < x \leq \bar{x}_5(t), \\
	\frac{2}{t-\frac{20}{9}}\left( x-\frac{497}{684}t -\frac{200}{81} - \frac{14341}{12996}\right)\!, & \bar{x}_5(t) < x \leq \bar{x}_6(t), \\
	\frac{19}{80}t + \frac{211}{228}, & \bar{x}_6(t) < x, 
	\end{cases} \\
	F(t, x) &= \begin{cases}
	0, & x \leq \bar{x}_1(t), \\
	\frac{4}{(t-2)^2}\left(x + \frac{19}{160}t^2 -\frac{1307}{684}t -\frac{40}{81} - \frac{961}{1444}\right)\!, & \bar{x}_1(t) < x \leq \bar{x}_2(t), \\
	\frac{1}{2}, & \bar{x}_2(t) < x \leq \bar{x}_3(t), \\
	\frac{4}{(t-\frac{40}{19})^2} \left(x + \frac{19}{160}t^2 - \frac{1325}{684}t -\frac{40}{81} -\frac{1039}{1444}\right) \!, & \bar{x}_3(t) < x \leq \bar{x}_4(t), \\
	\frac{3}{4}, & \bar{x}_4(t) < x \leq \bar{x}_5(t), \\
	\frac{4}{(t-\frac{20}{9})^2} \left(x + \frac{19}{160}t^2 - \frac{1355}{684}t -\frac{55}{27} + \frac{961}{1444} \right)\!, & \bar{x}_5(t) < x \leq \bar{x}_6(t), \\
	\frac{19}{20},  & \bar{x}_6(t) < x, 
	\end{cases}
\end{align}
\end{subequations}
where
\begin{align*}
	\bar{x}_1(t) &= y(t, 0) = -\frac{19}{160}t^2 + \frac{1307}{684}t + \frac{40}{81} + \frac{961}{1444}, \\
	\bar{x}_2(t) &= y(t, 2) = \frac{1}{160}t^2 + \frac{965}{684}t + \frac{40}{81} + \frac{1683}{1444}, \\
	\bar{x}_3(t)&= y\!\left(t, \frac{761}{361}\right) = \frac{1}{160}t^2 + \frac{965}{684}t + \frac{40}{81} + \frac{1839}{1444}, \\
	\bar{x}_4(t) &= y\!\left(t, \frac{1522}{361}\right) = \frac{11}{160}t^2 + \frac{785}{684}t + \frac{40}{81} + \frac{2239}{1444}, \\
	\bar{x}_5(t) &= y\!\left(t, \frac{362}{81} \right) =  \frac{11}{160}t^2 + \frac{785}{684}t + \frac{134}{27} - \frac{3849}{1444}, \\
	\bar{x}_6(t) &= y\!\left(t, \frac{181}{27}\right) =   \frac{19}{160}t^2 + \frac{211}{228}t + \frac{503}{81} - \frac{5293}{1444}. 
\end{align*}
Combining \eqref{eq:exEul1} and \eqref{eq:exEul2}--\eqref{eq:exEul4} finally yields the globally defined $\alpha$-dissipative solution $(u, F)(t, \cdot)$ with initial data \eqref{eq:exInitial2} and $\alpha$ given by \eqref{eq:exInitial}.  
\end{document}